\documentclass[oneside,english]{amsart}
\usepackage[T1]{fontenc}
\usepackage{color}
\usepackage{amsthm}
\usepackage{bm}
\usepackage{amstext}
\usepackage{amssymb}
\usepackage[all]{xy}

\makeatletter
\numberwithin{equation}{section} 
\numberwithin{figure}{section} 

\usepackage{amsmath}
\usepackage{amsfonts}
\usepackage{amsthm}
\usepackage{amssymb}
\usepackage{amscd}
\usepackage[all]{xy}
\usepackage{enumerate}
\usepackage{mathrsfs}
\usepackage{graphicx, bm, color, ulem}

\def\ft#1{{\mathsf #1}}
\def\chow{{\mathscr X}}
\def\hcoY{{\mathscr Y}}
\def\Hes{{\mathscr H}}
\def\UU{{\mathscr U}}
\def\Zpq{{\mathscr Z}}

\def\hchow{{\mathaccent20\chow}}
\def\Prt{{\mathscr P}}
\def\Lrho{\text{\large{\mbox{$\rho\hskip-2pt$}}}}
\def\tLrho{\text{\large{\mbox{$\tilde\rho\hskip-2pt$}}}}
\def\Lpi{\text{\large{\mbox{$\pi\hskip-2pt$}}}}

\font\eu=eusm10 at 10pt

\def\eF{\text{\eu F}}
\def\eG{\text{\eu G}}
\def\eQ{\text{\eu W}}
\def\eS{\text{\eu U}}

\input xyimport.tex

\newcommand{\vcorr}[3][1]{%
  \begingroup
    \tabcolsep=.5\tabcolsep
    \sbox0{%
      \begin{tabular}[b]{@{}l}%
        #3%
         \tabularnewline
      \end{tabular}%
    }%
    \settoheight{\dimen0 }{%
      \rotatebox{#2}{%
        \copy0 %
        \kern-\tabcolsep
      }%
    }%
    \rule{0pt}{#1\dimen0}%
    \setlength{\wd0 }{1em}%
    \setlength{\ht0 }{1em}%
    \rotatebox{#2}{\usebox{0}}%
  \endgroup
}

\newenvironment{fcaption}{\begin{list}{}{
\setlength{\leftmargin}{35pt}
\setlength{\rightmargin}{35pt}
\setlength{\labelsep}{5pt}
}}{\end{list}}

\newenvironment{Screen}{\begingroup \endgroup}

\newenvironment{myitem}{\begin{list}{}{
\setlength{\leftmargin}{25pt}
\setlength{\labelsep}{5pt}
}}{\end{list}}
\newenvironment{myitem2}{\begin{list}{}{
\setlength{\leftmargin}{0.5cm}
\setlength{\itemindent}{-0.3cm}
\setlength{\itemsep}{0cm}
}}{\end{list}}

\newtheorem{thm}{Theorem}[subsection]
\newtheorem{prop}[thm]{Proposition}

\newtheorem{lem}[thm]{Lemma}
\theoremstyle{definition}
\newtheorem{defn}[thm]{Definition}
\newtheorem{expl}[thm]{Example}
\theoremstyle{remark}
\newtheorem*{rem}{Remark}

\makeatother
\usepackage{babel}

\def\towerI{
\begin{matrix}
\emptyset \subset &\chow_0  \; \subset 
& \hskip-0.5cm \Sec^1 \chow_0 \hskip-0.5cm  
& \; \subset \;
 \Sec^2 \chow_0  \;\subset\; \cdots \;\subset\; \Sec^{n-1} \chow_0  
 \;\subset  
& \hskip-0.5cm \Sec^{n} \chow_0.  \\ 
&\vcorr{90}{=} \;\;\;\;& \vcorr{90}{=} && \hskip-0.5cm \vcorr{90}{=} &\\
&v_2(\mP(V)) &  \chow && \mP(\ft{ S}^2 V)  
\end{matrix}
}
\def\towerII{
\begin{matrix}
 \mP({\ft S}^2 V^*) \supset  &(\chow_0)^* \supset 
 &(\Sec^1\chow_0)^*  \;\supset 
 (\Sec^2 \chow_0)^* \supset \cdots \supset (\Sec^{n-1} \chow_0)^*  
 \supset  \emptyset.  \; \\ 
 &  \vcorr{90}{$=$} \quad & \;\; \vcorr{90}{=}  \hfill  \\
& \Hes \quad  & \;\; \chow^* \hfill \\
\end{matrix}
}
\def\xyquiverI{
\begin{matrix}
\begin{xy}
(5,-12)*+{\circ}="c3",
(20,-12)*+{\circ}="c2",
(35,0)*+{\circ}="c1a",
(35,-24)*+{\circ}="c1b",
(0,-12)*{\mathcal{F}_3},
(20,-15)*{\mathcal{F}_2},
(38,3)*{\mathcal{F}_{1a}},
(38,-27)*{\mathcal{F}_{1b}},
(18,1)*{{\ft S}^2V^*},
(18,-25)*{\wedge^2V^*},
(30,-8)*{V^*},
(15,-10)*{V^*},
(30,-16)*{V^*},
\ar "c3";"c2"
\ar "c2";"c1a"
\ar "c2";"c1b"
\ar @/^1pc/@{->} "c3";"c1a"
\ar @/_1pc/@{->} "c3";"c1b"
\end{xy}
\end{matrix} }
\def\FigQuad{\resizebox{11cm}{!}{\includegraphics{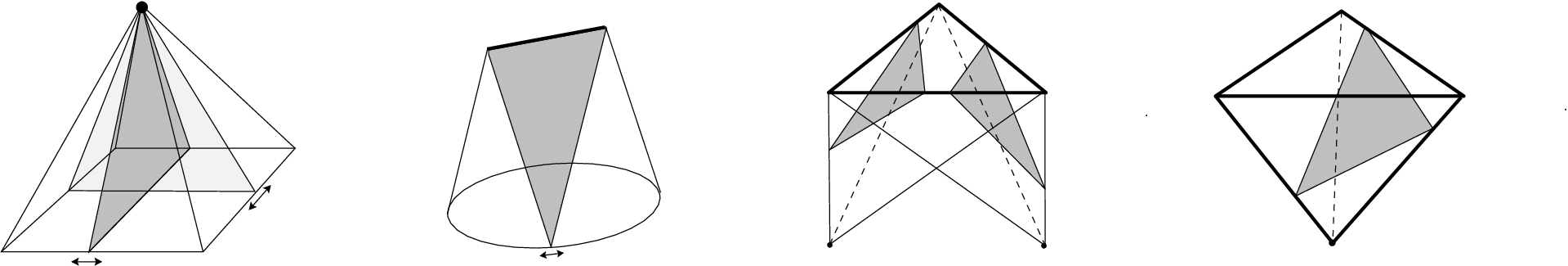}}}
\def\FigQuadDisplay{
\begin{xy}
(0,0)*{\FigQuad},
(-45,12)*{[V_1]},
(-16,12)*{[V_2]},
( 11,12)*{[V_3]},
( 40,12)*{[V_4]},
(-45,-12)*{\mP^1 \sqcup \mP^1},
(-16,-12)*{\mP^1},
( 11,-12)*{\mP^3*\mP^3},
( 40,-12)*{2\,\mP^3}
\end{xy}
}
\def\xyHYZG{
\begin{matrix}
\begin{xy}
(35,0)*+{\Zpq}="Z",  (60,0)*+{\mathrm{G}(3,V)}="G",
(35,-13)*+{\hcoY}="Y",
(35,-26)*+{\Hes,}="H",
(18,-13)*+{G_\hcoY^1\;\;\subset \;\;\;G_\hcoY \;\subset\;\;}, 
(15,-20)*+{\vcorr{-90}{$\simeq$} \;\qquad\;\; \vcorr{-90}{$\simeq$}},
(15,-26)*+{v_2(\mathbb{P}(V))\subset \ft{S}^2\mathbb{P}(V) \;\subset \;},
\ar^{\Lpi_\Zpq} "Z";"Y"
\ar^{\Lrho_\hcoY} "Y";"H"
\ar_{\Lrho_\Zpq\;\;}^{\;_{\mathbb{P}^8\text{-bundle}}\;\;\;\;\;} "Z";"G"
\end{xy}
\end{matrix}
}
\def\xyYYYYZYH{
\begin{matrix}
\begin{xy}
(40,0)*+{\Zpq}="Z",  (55,0)*+{\mathrm{G}(3,V)}="G",
(40,-12)*+{\hcoY}="Y",
(40,-24)*+{\Hes}="H",
(14,0)*+{\hcoY_2}="Yii",
(1,-12)*+{\hcoY_3}="Yiii", (5,-12)*+{}="yiii", (22,-12)*+{}="ty",
(27,-12)*+{\widetilde{\hcoY}}="tY",
(14,-24)*+{\overline{\hcoY}}="bY", (19,-24)*+{}="by", (35,-24)*+{}="h",
(34,-10)*+{\,_{\Lrho_{\widetilde{\hcoY}}}},
(30,-18)*+{\,_{\tilde{\varphi}_{DS}}},
\ar^{} "Z";"Y"
\ar^{\Lrho_\hcoY\;\;} "Y";"H"
\ar^{} "Z";"G"
\ar "Yii";"tY"
\ar "Yii";"Yiii"
\ar "Yiii";"bY"
\ar  "tY";"H"
\ar  "tY";"Y"
\ar "tY";"bY"
\ar @{-->}^{\,_\text{(anti-)flip}} "yiii";"ty"
\ar @{-->}_{\varphi_{DS}} "by";"h"
\end{xy}
\end{matrix}
}
\def\FigHilbCoYs{\resizebox{10cm}{!}{\includegraphics{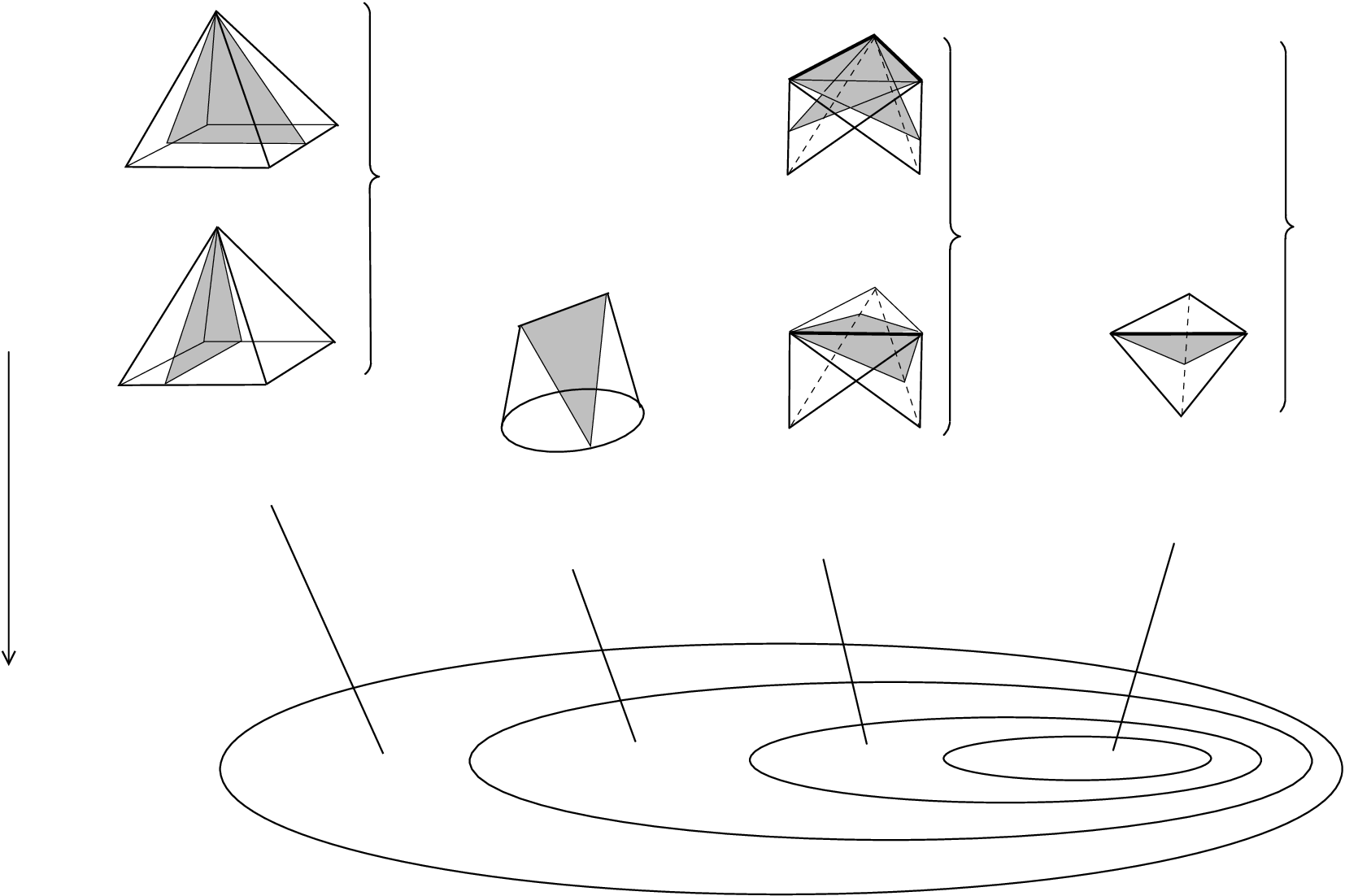}}}
\def\upIN{\vcorr{90}{$\in$}}
\def\xyFigCoYs{
\begin{xy}
(0,0)*{\FigHilbCoYs},
(-49,13)*{\widetilde\hcoY},
(-49,-20)*{\Hes},
(-18,-27)*{\Hes^4},
( 0,-26)*{\Hes^3},
( 17,-25)*{G_\hcoY^2},
(30,-24)*{G_\hcoY^1},
(-15,20)*{\tau\text{-conics}},
(-15,17)*{\,_{(\text{rk} \tau =3)}},
(-7,-3)*{\rho\text{-conic}},
(-7,-6)*{\,_{(\text{rk} \rho =3)}},
( 13, 18)*{\tau\text{-conics}},
( 13, 15)*{\,_{(\text{rk} \tau =2)}},
( 13,-1)*{\rho\text{-conics}},
( 13,-4)*{\,_{(\text{rk} \rho =2)}},
(38, 22)*{\tilde{\varphi}_{DS}^{-1}([V_4])},
(38,16)*{\text{\upIN}},
(38,-1)*{\rho\text{-conics}},
(38,-4)*{\,_{(\text{rk} \rho=1)}},
( 56, 18)*{\text{double lines}},
( 52, 14)*{\text{and}},
( 53, 10)*{\sigma\text{-planes}},
(0,0)*{}
\end{xy}
}
\def\FigwG25{\resizebox{7cm}{!}{\includegraphics{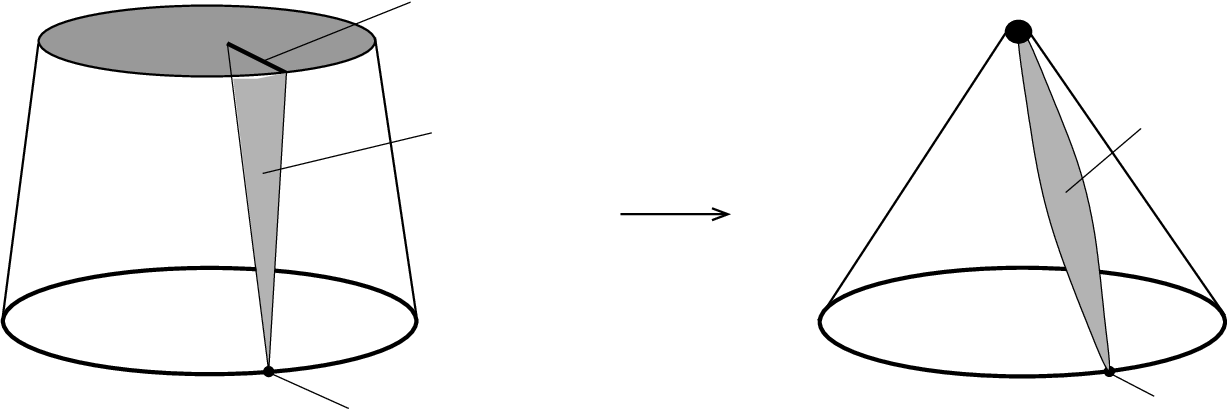}}}
\def\xyFigwG{
\begin{xy}
(0,0)*{\FigwG25},
(-22,-18)*{\mP(\sO_{\rG(2,V_4)}\oplus\sU_{\rG(2,V_4)}^*(1))},
( 28,-18)*{\Lrho_{\widetilde\hcoY}^{-1}([V_4])\simeq \rm{w}\rG(2,5)},
(-4,-7)*{\rG(2,V_4)},
( 43,-7)*{\rG(2,V_4)},
(-9,-11)*{\rho\text{-conic}},
( 36,-11)*{\rho\text{-conic}},
(-5,12)*{\sigma\text{-conics}},
(-3,5)*{\tau\text{-conics}},
( 37,5)*{\tau\text{-conics}},
(-15,15)*{E_\sigma=\mP(\sU_{\rG(2,V_4)}^*(1))},
( 30,14)*{\mP(V_4)\simeq \mP^3},
( 34,10)*{\sigma\text{-planes}},
(0,0)*{}
\end{xy}
}
\def\FigYsRed{\resizebox{10cm}{!}{\includegraphics{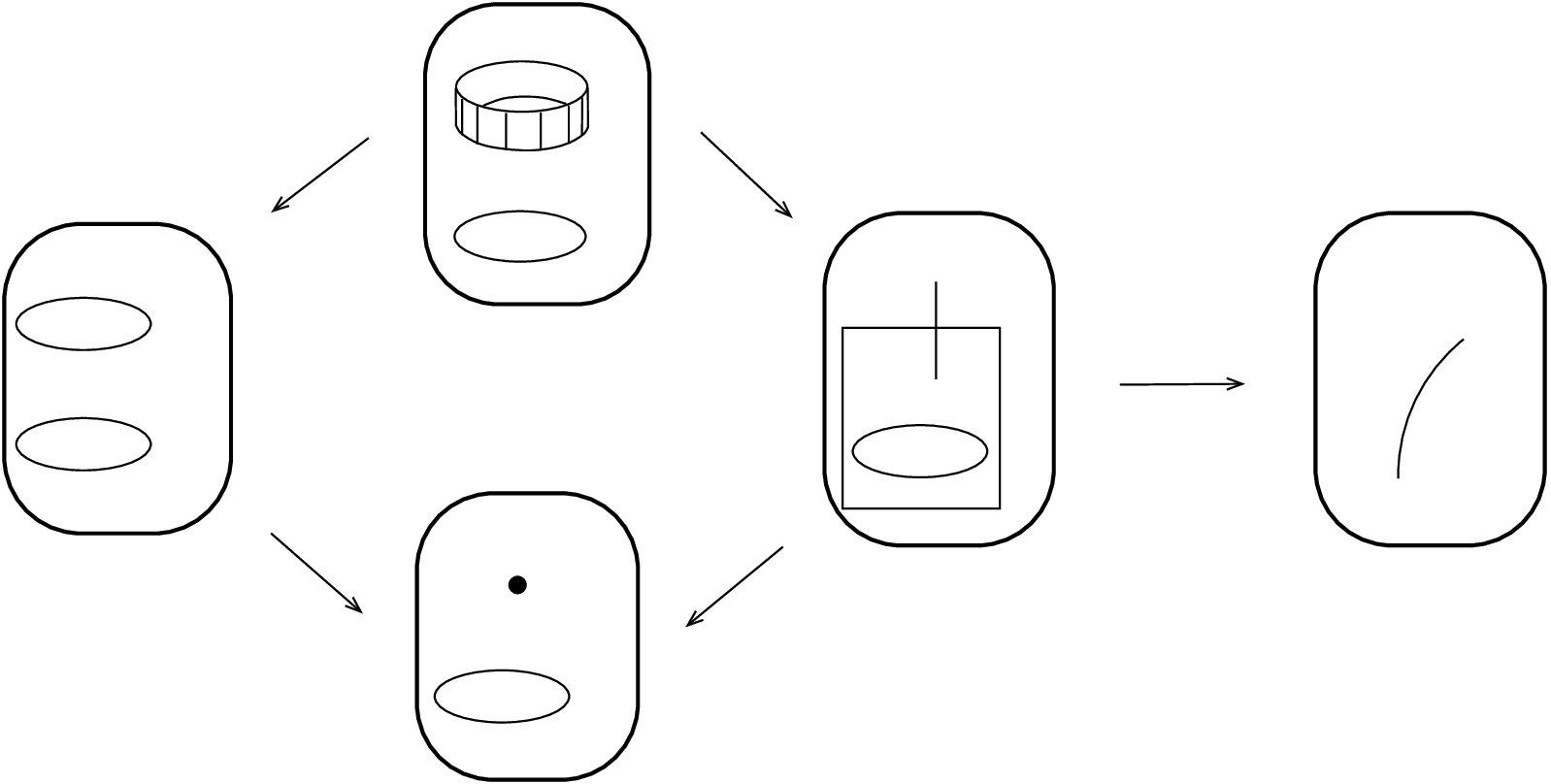}}}
\def\FigYsDisplay{
\begin{xy}
(0,0)*{\FigYsRed},
(-54,5.5)*{\hcoY_3},
(-38.5,3.5)*{\Prt_\rho},
(-38.5,-5)*{\Prt_\sigma},
(-28,12)*{\,_{\Lrho_{\hcoY_2}}},
{(-41.5,-12) \ar (-41.5,-20)^{\Lpi_{\hcoY_3}}_{\overset{\rG(3,6)\text{-}}{\text{bundle}}}},
(-41.5,-24)*{\mP(V)},
{(-31,0) \ar @{-->} (-2,0)^{\text{(anti-)flip}}},
(-25,24)*{\hcoY_2},
(-10, 16.5)*{F_\rho},
(-11,9)*{\Prt_\sigma},
(-2,11)*{\,_{\tLrho_{\hcoY_2}}},
(10,15)*{\widetilde\hcoY},
(7,1.7)*{F_{\widetilde\hcoY}},
(13,7.5)*{G_\rho},
(15,-5)*{\Prt_\sigma},
(26,-3)*{\,_{\Lrho_{\widetilde\hcoY}}},
(-16,-3.8)*{\overline{\hcoY}},
(-28,-8.7)*{\,_{\Lrho_{\hcoY_3}}},
(-16,-15)*{\rG(2,V)},
(-11.5,-20.5)*{\overline{\Prt}_\sigma},
(41.5,14)*{\hcoY},
(44.5,-3)*{G_\hcoY},
\end{xy}
}
\def\xyQuiverII{
\begin{matrix}
\begin{xy}
(5,-12)*+{\circ}="c3",
(20,-12)*+{\circ}="c2",
(35,0)*+{\circ}="c1a",
(35,-24)*+{\circ}="c1b",
(0,-12)*{\mathcal{E}_3},
(20,-15)*{\mathcal{E}_2},
(38,3)*{\mathcal{E}_{1a}},
(38,-27)*{\mathcal{E}_{1b}},
(18,1)*{\wedge^2V},
(18,-25)*{{\ft S}^2V},
(30,-8)*{V},
(15,-9)*{V},
(30,-16)*{V},
\ar "c3";"c2"
\ar "c2";"c1a"
\ar "c2";"c1b"
\ar @/^1pc/@{->} "c3";"c1a"
\ar @/_1pc/@{->} "c3";"c1b"
\end{xy}
\end{matrix}
}

\begin{document}
\global\long\def\sA{\mathcal{A}}
 \global\long\def\sB{\mathcal{B}}
 \global\long\def\sC{\mathcal{C}}
 \global\long\def\sD{\mathcal{D}}
 \global\long\def\sE{\mathcal{E}}
 \global\long\def\sF{\eF}
 \global\long\def\sG{\eG}
 \global\long\def\sH{\mathcal{H}}
 \global\long\def\sI{\mathcal{I}}
 \global\long\def\sJ{\mathcal{J}}
 \global\long\def\sK{\mathcal{K}}
 \global\long\def\sL{\mathcal{L}}
 \global\long\def\sN{\mathcal{N}}
 \global\long\def\sM{\mathcal{M}}
 \global\long\def\sO{\mathcal{O}}
 \global\long\def\sP{\mathcal{P}}
 \global\long\def\sS{\mathcal{S}}
 \global\long\def\sR{\mathcal{R}}
 \global\long\def\sQ{\mathcal{Q}}
 \global\long\def\sT{\mathcal{T}}
 \global\long\def\sU{\mathcal{U}}
 \global\long\def\sV{\mathcal{V}}
 \global\long\def\sW{\mathcal{W}}
 \global\long\def\sX{\mathcal{X}}
 \global\long\def\sY{\mathcal{Y}}
 \global\long\def\sZ{\mathcal{Z}}
 \global\long\def\tA{{\widetilde{A}}}
 \global\long\def\mA{\mathbb{A}}
 \global\long\def\mC{\mathbb{C}}
 \global\long\def\mF{\mathbb{F}}
 \global\long\def\mG{\mathbb{G}}
 \global\long\def\G{{\bf G}}
 \global\long\def\mN{\mathbb{N}}
 \global\long\def\mP{\mathbb{P}}
 \global\long\def\mQ{\mathbb{Q}}
 \global\long\def\mZ{\mathbb{Z}}
 \global\long\def\mW{\mathbb{W}}
 \global\long\def\Ima{\mathrm{Im}\,}
 \global\long\def\Ker{\mathrm{Ker}\,}
 \global\long\def\Alb{\mathrm{Alb}\,}
 \global\long\def\ap{\mathrm{ap}}
 \global\long\def\Bs{\mathrm{Bs}\,}
 \global\long\def\Chow{\mathrm{Chow}}
 \global\long\def\CP{\mathrm{CP}}
 \global\long\def\Div{\mathrm{Div}\,}
 \global\long\def\divi{\mathrm{div}\,}
 \global\long\def\expdim{\mathrm{expdim}\,}
 \global\long\def\ord{\mathrm{ord}\,}
 \global\long\def\Aut{\mathrm{Aut}\,}
 \global\long\def\Hilb{\mathrm{Hilb}}
 \global\long\def\Hom{\mathrm{Hom}}
 \global\long\def\id{\mathrm{id}}
 \global\long\def\Ext{\mathrm{Ext}}
 \global\long\def\sHom{\mathcal{H}{\!}om\,}
 \global\long\def\Lie{\mathrm{Lie}\,}
 \global\long\def\mult{\mathrm{mult}}
 \global\long\def\opp{\mathrm{opp}}
 \global\long\def\Pic{\mathrm{Pic}\,}
 \global\long\def\Pf{{\bf Pf}}
 \global\long\def\Sec{\mathrm{Sec}}
 \global\long\def\Spec{\mathrm{Spec}\,}
 \global\long\def\Sym{\mathrm{Sym}}
 \global\long\def\sQpec{\mathcal{S}{\!}pec\,}
 \global\long\def\Proj{\mathrm{Proj}\,}
 \global\long\def\Rhom{{\mathbb{R}\mathcal{H}{\!}om}\,}
 \global\long\def\aw{\mathrm{aw}}
 \global\long\def\exc{\mathrm{exc}\,}
 \global\long\def\emb{\mathrm{emb\text{-}dim}}
 \global\long\def\codim{\mathrm{codim}\,}
 \global\long\def\OG{\mathrm{OG}}
 \global\long\def\pr{\mathrm{pr}}
 \global\long\def\Sing{\mathrm{Sing}\,}
 \global\long\def\Supp{\mathrm{Supp}\,}
 \global\long\def\SL{\mathrm{SL}\,}
 \global\long\def\Reg{\mathrm{Reg}\,}
 \global\long\def\rank{\mathrm{rank}\,}
 \global\long\def\VSP{\mathrm{VSP}\,}
 \global\long\def\B{B}
 \global\long\def\Q{Q}
 \global\long\def\rG{\mathrm{G}}

\title{Duality between $\ft{S}^{2}\mP^{4}$ and the Double Quintic Symmetroid}

\author{Shinobu Hosono and Hiromichi Takagi}
\begin{abstract}
Let $\chow=\ft{S}^{2}\mP^{4}$ be the second symmetric product of
$\mP^{4}$ and $\hcoY$ the double cover of the symmetric determinantal
quintic hypersurface in $\mP^{14}$ considered in \cite{Tj}. We study
homological properties of $\chow$ and $\hcoY$ which indicate the
homological projective duality between (suitable noncommutative resolutions
of) $\chow$ and $\hcoY$. Among other things, we construct good desingularizations
$\hchow$ and $\widetilde{\hcoY}$ of $\chow$ and $\hcoY$, respectively,
and also a dual Lefschetz collection in $\sD^{b}(\hchow)$ and a Lefschetz
collection in $\sD^{b}(\widetilde{\hcoY})$. These are expected to
give respective (dual) Lefschetz decompositions of suitable noncommutative
resolutions of $\sD^{b}(\chow)$ and $\sD^{b}(\hcoY)$. 
\end{abstract}
\maketitle
\vspace{0.5cm}

\section{Introduction}

In the previous work \cite{HoTa1}, we have encountered an interesting
new geometry of Reye congruences in dimension three through our study
of mirror symmetry of Calabi-Yau manifolds. Reye congruences in dimension
two have been attracting attention for long in relation to geometries
of Enriques surfaces \cite{Co}. In dimension three, they define smooth
Calabi-Yau threefolds. It was found in \cite{HoTa1} that each of
these Calabi-Yau threefolds is paired with another smooth Calabi-Yau
threefold which arises naturally in the projective geometry of Reye
congruences.

Let $\chow={\ft S}^{2}\mP(V)$ be the symmetric product of the projective
space $\mP(V)\cong\mP^{n}$. In terms of the so-called Chow form,
we can embed $\chow$ into the projective space $\mP({\ft S}^{2}V)$
of symmetric $(n+1)\times(n+1)$ matrices. Then $\chow$ is identified
with the Chow variety of $0$-cycles of length two, and may also be
identified with the rank $\leq2$ locus of $\mP({\ft S}^{2}V)$ in
the natural stratification by matrix rank. It is a well-known fact
in classical projective geometry that this stratification is reversed
to the corresponding one in the dual projective space $\mP({\ft S}^{2}V^{*})$.

The Reye congruences are defined as general linear sections of $\chow$
by $(n+1)$ linear forms on $\mP({\ft S}^{2}V)$. Since giving $(n+1)$
linear forms is equivalent to fixing a linear subspace $L\simeq\mC^{n+1}\subset{\ft S}^{2}V^{*}$
in the vector space ${\ft S}^{2}V^{*}$ dual to $\ft{S}^{2}V$, the
corresponding Reye congruence may be written by $X=\chow\cap\mP(L^{\perp})$.
In this form, one may notice that the Reye congruence is in accord
with the Mukai's constructions~\cite{Mu} of Fano manifolds associated
to homogeneous spaces. For example, in his classification of prime
Fano threefolds, the Fano threefolds of genus $7,8,9,10$ are constructed
by a similar linear sections of suitable homogeneous spaces. Furthermore,
as an outcome of his construction, in the case of genus 7 for example,
it was found that the intermediate Jacobian of the Fano threefold
is isomorphic to the Jacobian of the curve which is obtained as the
`orthogonal' linear section of the projective dual of the homogeneous
space.

Observing this similarity, we first considered in \cite{HoTa1} the
Hessian hypersurface in the dual projective space by $H=\Hes\cap\mP(L)$,
where $\Hes$ represents the rank $\leq n$ locus in $\mP({\ft S}^{2}V^{*})$.
Assume that $n$ is even. While the Reye congruence $X$ is a smooth
Calabi-Yau manifold, $H$ is a Calabi-Yau variety which is singular
along a codimension two subvariety. The new geometry found in {[}\textit{loc.cit.}{]}
is a double covering $Y\to H$ branched along the singular locus of
$H$. It was shown that $Y$ is a smooth Calabi-Yau threefold when
$n=4$.

To clarify the relations to the previous construction, let us note
that $\Hes$ is the determinantal hypersurface of degree $n+1$. We
call $\Hes$ \textit{symmetroid} in this paper. When $n$ is even,
we will define the double covering $\hcoY\to\Hes$ branched along
the rank $\leq(n-1)$ locus of $\Hes$ (Proposition \ref{cla:double}).
We call $\hcoY$ \textit{double symmetroid}. $\hcoY$ is singular
for $n\geq4$ but still has nice properties in view of the minimal
model program (Proposition \ref{cla:ZY}). We denote by $\hcoY\cap\mP(L)$
the pull-back of $\Hes\cap\mP(L)$ under $\hcoY\to\Hes$. We show
that $Y=\hcoY\cap\mP(L)$ is a Calabi-Yau variety in general (Proposition
\ref{prop:CYgen}), and is smooth when $n=4$ \textcolor{black}{as
studied in the previous work. }

The recent proposal in \cite{HPD1,HPD2}, called homological projective
duality, describes the Mukai's construction in terms of the derived
category of coherent sheaves and a suitable decomposition (Lefschetz
decomposition) of it. More generally, for a singular variety, the
proposal deals with the so-called noncommutative resolution which
is a full subcategory of the derived category of a suitable resolution
of the singularities. The classical examples of the Fano threefolds
of genus $7,8,9,10$ due to Mukai, which are related to some nice
homogeneous spaces, are described in this framework \cite{HypSec}.
Also the homological projective dual of the Grassmann variety ${\rm G}(2,7)$
was shown to be the noncommutative resolution of its projective dual
variety $\mathrm{Pf}(7)$, the Pfaffian variety \cite{HPD2}. Under
this duality, two Calabi-Yau threefolds are obtained \cite{BC,HPD2}
as suitable linear sections in a similar way to the above. While we
can observe many similarities between our case and the Grassmann-Pfaffian
case, there are also many dis-similarities between the two. One important
difference we should note is that $\chow$ as well as $\hcoY$ are
not homogeneous spaces although they admit natural quasi-homogeneous
$\SL(n)$ actions.

\vspace{0.2cm}

In this article, we make a first step toward formulating the new geometry
appeared in the previous work within the framework of the homological
projective duality.

\vspace{0.3cm}

The homological projective duality, if applies to our case, provides
a systematic way to describe the derived categories of the (noncommutative
resolutions of) linear sections of $\chow$ and $\hcoY$. Showing
the homological projective duality in general consists of two major
steps of finding suitable categorical/noncommutative resolutions and
making suitable (dual) Lefschetz decompositions of them. The categorical/noncommutative
resolutions of $\chow$ and $\hcoY$, respectively, should be identified
up to equivalences inside the derived categories $\sD^{b}(\hchow)$
and $\sD^{b}(\widetilde{\hcoY})$ as full subcategories by finding
suitable resolutions $\hchow\to\chow$ and $\widetilde{\hcoY}\to\hcoY$.

A natural resolution of $\chow$ is given by the Hilbert-Chow morphism
$\hchow=\Hilb^{2}\mP^{n}$ $\to\chow$(cf. Subsection \ref{HilbChow}).
In this case, it should be rather easy to find the noncommutative
resolution of $\chow$ in the derived category $\sD^{b}(\hchow)$
based on the theory of \cite{Lef} (see \cite{HoTa4}). In contrast
to this, the singularity of $\hcoY$ turns out to be more involved
(see Subsection \ref{sub:tildeY-Y}). Because of this complication,
the theorem \cite[Theorem 4.3]{Lef} does not apply to this for example,
and having the noncommutative resolution seems to be a more difficult
task although we will find a nice resolution $\widetilde{\hcoY}\to\hcoY$
(\textbf{Subsections~\ref{subsection:BlowUp}}, \textbf{\textcolor{black}{\ref{sub:tildeY-Y}}}
). In this paper, as a strong indication for the homological projective
duality between $\hcoY$ and $\chow$, we find a Lefschetz collection
which generates a full subcategory of $\sD^{b}(\widetilde{\hcoY})$
(\textbf{Theorem~\ref{thm:Gvan}}) and also the corresponding dual
Lefschetz collection in $\sD^{b}(\hchow)$ (\textbf{Theorem~\ref{thm:Gvan1}}).
These two theorems are main results of this paper. We expect that
these (dual) Lefschetz collections are actually the (dual) Lefschetz
decompositions of the noncommutative resolutions of $\hcoY$ and $\chow$.

\vspace{0.3cm}

The construction of this paper is as follows: In Section \ref{section:Chow},
we summarize some basic properties of the variety $\chow$, and construct
the Hilbert-Chow morphism $\hchow\to\chow$. Using these, we construct
the dual Lefschetz collection in $\sD^{b}(\hchow)$. Also some basic
properties of Calabi-Yau manifolds of Reye congruences are summarized.
In Section \ref{section:double}, we introduce the determinantal hypersurface
(symmetroid) $\Hes$, and using the geometry of singular quadric parametrized
by $\Hes$, we define its double cover, i.e., the double symmetroid
$\hcoY$. We define Calabi-Yau variety $Y$ in $\hcoY$, and for $n=4$,
we determine topological invariants of $Y$ from geometries of $\hcoY$
(\textbf{Proposition \ref{prop:Y}}). In Section \ref{sec:The-resolution-Y},
we study the birational geometry of $\hcoY$ by introducing a subvariety
$\overline{\hcoY}$ in $\rG(3,\wedge^{3}V)$. Although we will not
go into the details, we find that the birational geometry of $\hcoY$
has a close relation to the Hilbert scheme of conics on Grassmann
${\rm G}(3,V)$ (Proposition \ref{prop:Y34}). As a resolution of
the singularity of $\overline{\hcoY},$ we introduce the Grassmann
bundle ${\rm G}(3,T(-1)^{\wedge2})$ over $\mP(V)$, and the so-called
two ray game (Sarkisov link) of this Grassmann bundle is studied in
detail to obtain our desingularization $\widetilde{\hcoY}$ of $\hcoY$
(see (\ref{eq:smallDiag}) and also Fig.4 in Section \ref{sec:SheavesDef}).
The morphism $\widetilde{\hcoY}\to\hcoY$ turns out to be a divisorial
contraction which is negative with respect to the canonical divisor
(Propositions \ref{cla:FY} and \ref{cla:F}). In Section \ref{sec:SheavesDef},
we define three locally free sheaves $\widetilde{\sS}_{L}^{*}$, $\widetilde{\sQ}$,
and $\widetilde{\sT}$ on $\widetilde{\hcoY}$ (Definitions \ref{def:SQT1}
and \ref{def:tidelT}), which will generate a Lefschetz collection
in $\sD^{b}(\widetilde{\hcoY})$. In Section \ref{subsubsection:Global},
we further study the exceptional divisor $F_{\widetilde{\hcoY}}$
of divisorial contraction $\widetilde{\hcoY}\to\hcoY$. We describe
some birational models of $F_{\widetilde{\hcoY}}$ in Proposition
\ref{cla:F'} and (\ref{eq:house}). This section is necessary for
our cohomology calculations in the subsequent section. In Section
\ref{section:comp}, we find a Lefschetz collection in $\sD^{b}(\widetilde{\hcoY})$
and observe a certain duality between the quiver diagrams associated
to the Lefschetz collection in $\sD^{b}(\widetilde{\hcoY})$ and the
dual Lefschetz collection in $\sD^{b}(\hchow)$ respectively ((\ref{eq:quiver1})
and (\ref{eq:quiver2})).

\vspace{1cm}

\vspace{0.5cm}
 \textbf{Acknowledgements:} This paper is supported in part by Grant-in
Aid Scientific Research (C 18540014, S.H.) and Grant-in Aid for Young
Scientists (B 20740005, H.T.). The authors are grateful to Prof. R.F.
for variable suggestions to improve our construction of $\widetilde{\hcoY}$.
They also thank Nicolas Addington and Sergey Galkin for useful communications.

\noindent 
\newpage{}

\vspace{0.3cm}
\noindent\textbf{\textcolor{black}{Notation: }}Throughout the paper,
we work over $\mC$, the complex number field. We will use the following
notation which simplifies lengthy formulas:

\vspace{0.3cm}

\global\long\def\Homega{H_{\mP(\Omega(1))}}
 \global\long\def\Hwomega{H_{\mP(\Omega(1)^{\wedge2})}}
 \global\long\def\Lwomega{L_{\mP(\Omega(1)^{\wedge2})}}
 \global\long\def\HGo{H_{\mathrm{G}(\Omega(1)^{\wedge2})}}
 \global\long\def\HGt{H_{\mathrm{G}(\wedge^{2} T(-1))}}

\global\long\def\Lt{L_{\mP(T(-1))}}
 \global\long\def\Lwt{L_{\mP(T(-1)^{\wedge2})}}
 \global\long\def\Ht{H_{\mP(T(-1))}}
 \global\long\def\Hwt{H_{\mP(T(-1)^{\wedge2})}}

\begin{myitem} 

\item $V$: a (fixed) $n+1$ dimensional complex vector space. $\mP^{n}:=\mP(V)$. 

\item $V_{i}$: an $i$-dimensional vector subspace of $V$. \item $\Omega(1):=\Omega_{\mP(V)}(1)$. 

\item $\Omega(1)^{\wedge i}:=\wedge^{i}(\Omega_{\mP(V)}(1))$ for
$i\geq2$. 

\item $T(-1):=T_{\mP(V)}(-1)$. 

\item $T(-1)^{\wedge i}:=\wedge^{i}(T(-1))$ for $i\geq2$. 

\item $\sO(i):=\sO_{\mP(V)}(i)$ for $i\in\mZ$. \end{myitem}


$\;$

$\;$

$\;$

$\;$

$\;$

$\;$

\section{Preliminaries}

For the computations of cohomology groups which appear in this paper,
we use the Bott theorem about the cohomology groups of Grassmann bundles
extensively.

\vspace{0.3cm}

For a locally free sheaf $\mathcal{E}$ of rank $r$ on a variety
and a nonincreasing sequence $\beta=(\beta_{1},\beta_{2},\dots,\beta_{r})$
of integers, we denote by $\Sigma^{\beta}\mathcal{E}$ the associated
locally free sheaf with the Schur functor $\Sigma^{\beta}$.

\begin{thm} \textbf{$(${\bf Bott theorem}$)$} \label{thm:Bott}
Let $\pi\colon\mathrm{G}(r,\sA)\to X$ be a Grassmann bundle for a
locally free sheaf $\sA$ on a variety $X$ of rank $n$ and $0\to\sS\to\sA\to\sQ\to0$
the universal exact sequence. For $\beta:=(\alpha_{1},\dots,\alpha_{r})\in\mZ^{r}$
$(\alpha_{1}\geq\dots\geq\alpha_{r})$ and $\gamma:=(\alpha_{r+1},\dots,\alpha_{n})\in\mZ^{n-r}$
$(\alpha_{r+1}\geq\dots\geq\alpha_{n})$, we set $\alpha:=(\beta,\gamma)$
and $\sV(\alpha):=\Sigma^{\beta}\sS^{*}\otimes\Sigma^{\gamma}\sQ^{*}$.
Finally, let $\rho:=(n,n-1,\dots,1)$, and, for an element $\sigma$
of the $n$-th symmetric group $\mathfrak{S}_{n}$, we set $\sigma^{\bullet}(\alpha):=\sigma(\alpha+\rho)-\rho$. 

\begin{myitem2} \item[\rm (1)] If $\sigma(\alpha+\rho)$ contains
two equal integers, then $R^{i}\pi_{*}\sV(\alpha)=0$ for any $i\geq0$. 

\item[\rm (2)] If there exists an element $\sigma\in\mathfrak{S}_{n}$
such that $\sigma(\alpha+\rho)$ is strictly decreasing, then $R^{i}\pi_{*}\sV(\alpha)=0$
for any $i\geq0$ except $R^{l(\sigma)}\pi_{*}\sV(\alpha)=\Sigma^{\sigma^{\bullet}(\alpha)}\sA^{*}$,
where $l(\sigma)$ represents the length of $\sigma\in\mathfrak{S}_{n}$. 

\end{myitem2}

\end{thm}

\begin{proof} See \cite{Bo}, \cite{D}, or \cite[(4.19) Corollary]{W}.
\end{proof}

\vspace{0.3cm}
 In this paper, we adopt the following definition of Calabi-Yau variety
and also Calabi-Yau manifold.

\begin{defn} We say a normal projective variety $X$ \textit{a Calabi-Yau
variety} if $X$ has only Gorenstein canonical singularities, the
canonical bundle of $X$ is trivial, and $h^{i}(\sO_{X})=0$ for $0<i<\dim X$.
If $X$ is smooth, then $X$ is called \textit{a Calabi-Yau manifold}.
A smooth Calabi-Yau threefold is abbreviated as a Calabi-Yau threefold.~\hfill{}{[}{]}\end{defn}

\vspace{1cm}

\section{The Geometry of $\ft{S}^{2}\mP(V)$}

\label{section:Chow}

\subsection{$\ft{S}^{2}\mP(V)$ and quadrics in $\mP(V)$}

\label{Chow}

Let ${\chow}:=\ft{S}^{2}\mP(V)$ be the symmetric product of $\mP(V)$.
Since we can identify $\chow$ with the Chow variety of 0-cycles in
$\mP(V)$ of length $2$, ${\chow}$ can be embedded in the projective
space $\mP(\ft{S}^{2}V)$ for the symmetric product $\ft{S}^{2}V:={\rm Sym}^{2}V$
in terms of the so-called \textit{Chow form}; \begin{equation}
w_{ii}=x_{i}y_{i},\, w_{ij}:=x_{i}y_{j}+x_{j}y_{i}\quad(i\not=j),\label{eq:Chow}\end{equation}
where $w_{ij}=w_{ji}\,(1\leq i,j\leq n+1)$ and $x_{i},y_{i}$ $(1\leq i\leq n+1)$
are coordinates of $\mP(\ft{S}^{2}V)$ and $\mP(V)$, respectively
(cf. \cite[Theorem 2.2]{GKZ}). If we view this as giving a morphism
$\mP(V)\times\mP(V)\to\mP(\ft{S}^{2}V)$, the isomorphism of $\ft{S}^{2}\mP(V)$
to the Chow variety follows from the fact that $w_{ij}$ in (\ref{eq:Chow})
generates all the invariant polynomials under $x_{i}\leftrightarrow y_{i}$
{[}loc.cit.{]}. We may identify $\mP(\ft{S}^{2}V)$ with the dual
to the space of the symmetric $(1,1)$-divisors on $\mP(V)\times\mP(V)$.
$\ft{S}^{2}\mP(V)$ in $\mP(\ft{S}^{2}V)$ is defined by the zero
set of all the $3\times3$ minors of $(n+1)\times(n+1)$ symmetric
matrix $(w_{ij})$ representing the coordinate of $\mP(\ft{S}^{2}V)$.

Quadrics in $\mP(V)$ are given by the equations ${\empty^{t}{\bm{x}}}A{\bm{x}}=0$
with ${\empty^{t}{\bm{x}}}=\left(\begin{matrix}x_{1},\cdots,x_{n+1}\end{matrix}\right)$
and an $(n+1)\times(n+1)$ symmetric matrix $A$. We denote by $Q_{A}$
the quadric defined by a symmetric matrix $A$, and by $q_{A}({\bm{x}})$
the quadratic form ${\empty^{t}{\bm{x}}}A{\bm{x}}$. The projectivization
of the vector space of all the $(n+1)\times(n+1)$ symmetric matrices
may be identified with $\mP(\ft{S}^{2}V^{*})$, which is dual to $\mP(\ft{S}^{2}V)$.
More explicitly, we write the dual pairing by \[
A\cdot w=\sum_{1\leq i\leq j\leq n+1}a_{ij}w_{ij},\]
 where $w=(w_{ij})\in\mP(\ft{S}^{2}V)$. Using (\ref{eq:Chow}), we
have the equality \begin{equation}
A\cdot w_{{\bm{x}}{\bm{y}}}={\empty^{t}{\bm{x}}}A{\bm{y}},\label{eq:pairing}\end{equation}
 where $w_{{\bm{x}}{\bm{y}}}$ is the image in $\mP(\ft{S}^{2}V)$
of $({\bm{x}},{\bm{y}})\in\mP(V)\times\mP(V)$.

\vspace{0.5cm}

\subsection{Projective duals of ${\rm Sec}^{i}\chow_{0}$}

\label{subsection:PD}

When we restrict the morphism $\mP(V)\times\mP(V)\to\mP({\ft S}^{2}V)$
described above to the diagonal, we obtain the second Veronese morphism
of $\mP(V)$. We denote by $\chow_{0}=v_{2}(\mP(V))$ its image. Then
it is easy to see that $\ft{S}^{2}\mP(V)$ is defined by all the $2\times2$
minors of the generic $(n+1)\times(n+1)$ symmetric matrix $(w_{ij})$,
namely, $\chow_{0}$ is the locus of symmetric matrices of rank one.
By the characterizations of $\chow$ and $\chow_{0}$ with rank condition
as above, we see that $\chow$ is the secant variety of $\chow_{0}$,
namely, \[
\chow=\overline{\cup\{\langle p,q\rangle\mid p,q\in{{}{\chow}}_{0},p\not=q\}},\]
 where $\langle p,q\rangle$ is the line through $p$ and $q$.

The second Veronese variety $\chow_{0}$ is one of the Scorza varieties
classified by Zak (see \cite{Z}, \cite{Ch}). Associated to $\chow_{0}$,
we naturally have the tower of higher secant varieties. Recall that,
for projective varieties $X,Y\subset\mP^{N}$, in general, their \textit{join}
is defined as \[
J(X,Y):=\overline{\cup\{\langle p,q\rangle\mid p\in X,q\in Y,p\not=q\}}.\]
 The higher secant varieties are defined inductively by $\Sec^{m}X:=J(\Sec^{m-1}X,X)$
with $\Sec^{0}X:=X$. Note that $\Sec^{1}X$ is nothing but the secant
variety of $X$. Since $\chow_{0}$ is the locus of symmetric matrices
of rank one, one may identify $\Sec^{i}\chow_{0}$ with the locus
of symmetric matrices of rank $\leq i+1$ since a general point of
$\Sec^{i}\chow_{0}$ corresponds to a sum of $(i+1)$ matrices of
rank one. In particular, it holds that $\Sec^{n}\chow_{0}=\mP(\ft{S}^{2}V)$
and $\Sec^{n-1}\chow_{0}$ is the hypersurface of degree $n+1$ defined
by the determinant of the generic symmetric matrix $(w_{ij})$. In
summary, we have the tower of the secant varieties: \begin{equation}
\towerI\label{eq:tower1}\end{equation}
 It is known that $\dim\Sec^{i}\chow_{0}=(i+1)n-\frac{i(i-1)}{2}$.
This tower gives the orbit decomposition of the action of $\SL_{n+1}$
on $\mP(\ft{S}^{2}V)$. Precisely, $\Sec^{i}\chow_{0}\setminus\Sec^{i-1}\chow_{0}$
is an $\SL_{n+1}$-orbit for any $i$ since any two symmetric matrices
of the same rank are transformed to each other by $\SL_{n+1}$. It
is known that $\Sec^{i+1}\chow_{0}$ is the singular locus of $\Sec^{i}\chow_{0}$.
In particular, $\chow_{0}=\Sing\chow$.

In the dual projective space $\mP(\ft{S}^{2}V^{*})$ to $\mP(\ft{S}^{2}V)$,
we consider the dual varieties $(\Sec^{i}\chow_{0})^{*}$ of $\Sec^{i}\chow_{0}$,
namely, $(\Sec^{i}\chow_{0})^{*}$ is the closure of the locus of
the hyperplanes of $\mP({\ft S}^{2}V)$ tangent to $\Sec^{i}\chow_{0}$
at smooth points. Verifying the tangent space of $\Sec^{i}\chow_{0}$
at a general point, we see that $(\Sec^{i}\chow_{0})^{*}$ is the
locus of symmetric matrices $A$ of rank $\leq n-i$. In summary,
we have the dual tower to (\ref{eq:tower1}): \begin{equation}
\towerII\label{eq:tower2}\end{equation}
 Throughout this paper, we set \[
\Hes:=(\chow_{0})^{*}.\]
 $\Hes$ is the symmetric determinantal hypersurface in $\mP(\ft{S}^{2}V^{*})$,
which is called \textit{the symmetroid}.

\vspace{0.5cm}

\subsection{$\hchow=\Hilb^{2}\mP(V)$ and $\mathrm{G}(2,V)$}

\label{HilbChow}

By the construction of the double cover $\mP(V)\times\mP(V)\to\chow$,
we see that $\chow$ has quotient singularities of type $\frac{1}{2}(1^{n})$
along the singular locus $\chow_{0}=v_{2}(\mP(V))$. In particular,
$\chow$ is Gorenstein if and only if $n$ is even. Hereafter in this
subsection, we assume that $n$ is even. Now we set \[
{\hchow}:=\Hilb^{2}\mP(V),\]
 which is the Hilbert scheme of $0$-dimensional subschemes of $\mP(V)$
of length $2$ (we simply call it the Hilbert scheme of two points
on $\mP(V)$). Recall that the Hilbert-Chow morphism $f\colon\hchow\to\chow$
is the blow-up of $\chow$ along the singular locus $\chow_{0}$. 

Since the $0$-dimensional subscheme on $\mP(V)$ of length two determines
a line on $\mP(V)$, we have a natural morphism $g\colon{\hchow}\to\mathrm{G}(2,V)$:
\[
\xymatrix{ & {\hchow}\ar[dl]_{f}\ar[dr]^{g}\\
\chow &  & \mathrm{G}(2,V).}
\]
 $\;$

Let $\sF$ be the universal subbundle of rank two on $\mathrm{G}(2,V)$.
Then the Hilbert scheme $\hchow$ of two points on $\mP(V)$ is isomorphic
to $\mP(\ft{S}^{2}\sF)$. This follows from the fact that the fiber
of $g$ over a point $[l]\in\mathrm{G}(2,V)$ is identified with ${\rm Hilb}^{2}\mP(l)=\ft{S}^{2}\mP(l)\simeq\mP(\ft{S}^{2}l)$,
where $l\simeq\mC^{2}$ is the affine two plane representing $[l]$. 

We set \[
\text{\ensuremath{H_{{\hchow}}=f^{*}\sO_{\chow}(1)}and \ensuremath{L_{{\hchow}}=g^{*}\sO_{\mathrm{G}(2,V)}(1)}.}\]
 By $\hchow\simeq\mP(\ft{S}^{2}\sF)$, we see that the morphism $f$
follows from the natural morphism $\mP(\ft{S}^{2}\sF)\to\ft{S}^{2}\mP(V)$.
Hence $H_{{\hchow}}$ is the tautological divisor for $\mP(\ft{S}^{2}\sF)$.

Let us denote by $E_{f}$ the $f$-exceptional divisor. We see that
$E_{f}\simeq\mP(\sF)$, namely, $E_{f}$ is isomorphic to the total
space of the universal family of lines on $\mP(V)$ since $E_{f}\subset\Hilb^{2}\mP(V)$
parameterizes pairs of points $x\in\mP(V)$ and lines $l$ through
$x$. Moreover, from the relative Euler sequence $0\to\sO_{\hchow}\to g^{*}(\ft{S}^{2}\sF)\otimes\sO_{\hchow}(1)\to T_{\hchow/G(2,V)}\to0$,
we have \begin{equation}
K_{\hchow}=-3H_{\hchow}-(n-2)L_{\hchow}.\label{eq:Khchow2}\end{equation}
 On the other hand, since $f\colon\hchow\to\chow$ is the blow-up
of $\chow$ along the singular locus $\chow_{0}$, we have \begin{equation}
K_{\hchow}=-(n+1)H_{\hchow}+\frac{n-2}{2}E_{f}.\label{eq:Khchow1}\end{equation}
Therefore we have \begin{equation}
\frac{n-2}{2}E_{f}\sim(n-2)(H_{\hchow}-L_{\hchow}).\end{equation}

\vspace{0.5cm}

\subsection{Constructing a dual Lefschetz collection in $\sD^{b}(\hchow)$}


We recall some basic definitions from the theory of triangulated categories
(cf.~\cite{Bondal,BO}).

\begin{defn} \label{defn:excepcollect} An object $\sE$ in a triangulated
category $\sD$ is called \textit{an exceptional object} if $\Hom(\sE,\sE)\simeq\mC$
and $\Hom^{\bullet}(\sE,\sE)=0$ for $\bullet\not=0$. \hfill{}{[}{]}
\end{defn}

\begin{defn}  A triangulated subcategory $\sD'$ of $\sD$ is called
\textit{admissible} if there are right and left adjoint functors for
the inclusion functor $i_{*}:\sD'\to\sD$. \hfill{}{[}{]}\end{defn}

\begin{defn} \label{def:semi-orth} A sequence $\sD_{1},\dots,\sD_{m}$
of\textcolor{black}{{} admissible} triangulated subcategories in a triangulated
category $\sD$ is called a \textit{semiorthogonal collection} if
$\Hom_{\sD}(\sD_{i},\sD_{j})=0$ for any $i>j$. Moreover, if $\sD_{1},\dots,\sD_{m}$
generates $\sD$, then it is called a \textit{semiorthogonal decomposition}.

A semiorthogonal collection of exceptional objects $\sE_{1},\dots,\sE_{n}$
is called \textit{an exceptional collection}. Moreover, if $\Hom^{\bullet}(\sE_{i},\sE_{j})=0$
holds for any $i,j$ and any $\bullet\not=0$, then it is called \textit{an
strongly exceptional collection}. \hfill{}{[}{]} \end{defn}

Hereafter, in this article, we restrict our attention to the cases
of the derived categories of bounded complexes of coherent sheaves
on a variety. In such cases, a special type of semiorthogonal collection
plays an important role (cf.~\cite{HPD1,HPD2}).

\begin{defn} For a variety $X$, a \textit{Lefschetz collection}
of $\sD^{b}(X)$ is a semiorthogonal collection of the following form:
\[
\sD_{0},\sD_{1}(1),\dots,\sD_{m-1}(m-1),\]
 where it holds that $0\subset\sD_{m-1}\subset\sD_{m-2}\subset\dots\subset\sD_{0}\subset\sD^{b}(X)$
and $(k)$ means the twist by $L^{\otimes k}$ with a fixed invertible
sheaf $L$. Moreover, if $\sD_{0},\sD_{1}(1),...,\sD_{m-1}(m-1)$
generate $\sD^{b}(X)$, then it is called a \textit{Lefschetz decomposition}.

Similarly, a \textit{dual Lefschetz collection} of $\sD^{b}(X)$ is
a semiorthogonal collection of the following form: \[
\sD_{m-1}(-(m-1)),\sD_{m-2}(-(m-2)),\dots,\sD_{0},\]
 where it holds that $0\subset\sD_{m-1}\subset\sD_{m-2}\subset\dots\subset\sD_{0}\subset\sD^{b}(X)$.
Moreover, if $\sD_{m-1}(-(m-1)),\sD_{m-2}(-(m-2)),\dots,\sD_{0}$
generate $\sD^{b}(X)$, then it is called a \textit{dual Lefschetz
decomposition}. \hfill{}{[}{]} \end{defn}

\;

\;

Now, based on the geometry of the projective bundle $\hchow=\mP(\ft{S}^{2}\sF)$
over $G(2,V)$, we construct a dual Lefschetz collection in $\sD^{b}(\hchow)$
by restricting our attention to the case $n=4$. Other cases of $n>4$
should be done in a similar way, but we confine ourselves to this
case to avoid possible complications.

We may naturally conceive the sheaves $\sO_{\hchow}$ and $g^{*}\sF^{*}$
as the objects in the (dual) Lefschetz collection. Recall the isomorphism
$\hchow\simeq\mP(\ft{S}^{2}\sF)$ and consider associated Euler sequence
$0\to\sO_{\mP(\ft{S}^{2}\sF)}(-1)\to g^{*}\ft{S}^{2}\sF\to T_{\mP(\ft{S}^{2}\sF)/G(2,V)}(-1)\to0$.
Twisting this by $2L_{\hchow}$ we obtain an injection \begin{equation}
\varphi:\,\sO_{\hchow}(-H_{\hchow}+2L_{\hchow})\to(g^{*}\ft{S}^{2}\sF)(2L_{\hchow})\simeq g^{*}{\ft S}^{2}\sF^{*},\label{eq:F1a}\end{equation}
 where we use $\sF\otimes\sO_{\mathrm{G}(2,V)}(1)=\Sigma^{(0,-1)}\sF^{*}\otimes\Sigma^{(1,1)}\sF^{*}\simeq\sF^{*}$.
The cokernel of this injection, which is $T_{\mP(\ft{S}^{2}\sF)/G(2,V)}(-H_{\hchow}+2L_{\hchow})$,
plays a role in the following theorem:

\vspace{0.5cm}
 \begin{thm} \label{thm:Gvan1} 

$\;$

\begin{myitem2}

\item[\rm (1)] Let \[
(\mathcal{F}_{3},\mathcal{F}_{2},\mathcal{F}_{1a},\mathcal{F}_{1b})=(\,\sO_{\hchow},\; g^{*}\sF^{*},\;{\rm Coker}\,\varphi,\;\sO_{\hchow}(L_{\hchow})\,)\]
 be an ordered collection of sheaves on $\hchow$. Then $(\sK_{i})_{1\leq i\leq4}:=$
$(\mathcal{F}_{1b}^{*},\mathcal{F}_{1a}^{*},\mathcal{F}_{2}^{*},\mathcal{F}_{3}^{*})$
is a strongly exceptional collection of $\sD^{b}(\hchow)$, namely
satisfies \[
H^{\bullet}({\sK}_{i}^{*}\otimes{\sK}_{j})=0\text{ for }1\leq i,j\leq4\text{ and }\bullet>0\]
 and $H^{0}({\sK}_{i}^{*}\otimes{\sK}_{j})=0\,(i>j),\; H^{0}({\sK}_{i}^{*}\otimes{\sK}_{i})\simeq\mC\,(1\leq i\leq4)$.

\item[\rm (2)] For $i<j$, $H^{0}({\sK}_{i}^{*}\otimes{\sK}_{j})$
are given by \[
\begin{aligned} & H^{0}(\mathcal{F}_{3}^{*}\otimes\mathcal{F}_{2})\simeq V^{*},\, H^{0}(\mathcal{F}_{3}^{*}\otimes\mathcal{F}_{1a})\simeq\ft{S}^{2}V^{*},\, H^{0}(\mathcal{F}_{3}^{*}\otimes\mathcal{F}_{1b})\simeq\wedge^{2}V^{*},\\
 & H^{0}(\mathcal{F}_{2}^{*}\otimes\mathcal{F}_{1a})\simeq V^{*},\, H^{0}(\mathcal{F}_{2}^{*}\otimes\mathcal{F}_{1b})\simeq V^{*},\;\; H^{0}(\mathcal{F}_{1a}^{*}\otimes\mathcal{F}_{1b})=0,\end{aligned}
\]
 which may be summarized in the following quiver diagram: \begin{equation}
\xyquiverI\label{eq:quiver1}\end{equation}

\item[\rm (3)] Set $\sD_{\hchow}:=\langle\mathcal{F}_{1b}^{*},\mathcal{F}_{1a}^{*},\mathcal{F}_{2}^{*},\mathcal{F}_{3}^{*}\rangle\subset\sD^{b}(\hchow).$
Then \[
\sD_{\hchow}(-4),\dots,\sD_{\hchow}(-1),\sD_{\hchow}\]
 is a dual Lefschetz collection, where $(-t)$ represents the twist
by the sheaf $\sO_{\hchow}(-tH_{\hchow})$. 

\end{myitem2}\end{thm}

\vspace{0.8cm}
 We prepare the following lemma for our proof of the theorem.

\begin{lem} \label{cla:key1} Set $\sC_{ij}=\sK_{i}^{*}\otimes\sK_{j}\,(1\leq i,j\leq4)$.
For $\sC=\sC_{ij}$, it holds that \[
H^{\bullet}(\hchow,\sC(-t))\simeq H^{8-\bullet}(\hchow,\sC^{*}(t-5))\text{ for any }t.\]
 \end{lem}

\begin{proof} By the Serre duality and $K_{\hchow}=-5H_{\hchow}+E_{f}$,
we have $H^{\bullet}(\hchow,\sC(-t))\simeq H^{8-\bullet}(\hchow,\sC^{*}((t-5)H_{\hchow}+E_{f}))$
for any $t$. By the exact sequence \[
0\to\sC^{*}((t-5))\to\sC^{*}((t-5)H_{\hchow}+E_{f})\to\sC^{*}((t-5)H_{\hchow}+E_{f})|_{E_{f}}\to0,\]
 we have only to show that $H^{8-\bullet}(E_{f},\sC^{*}((t-5)H_{\hchow}+E_{f})|_{E_{f}})=0$.
Since $E_{f}\to\chow_{0}\simeq\mP^{4}$ is a $\mP^{3}$-bundle, it
suffices to show the vanishing of cohomology groups of the restriction
of $\sC^{*}((t-5)H_{\hchow}+E_{f})|_{E_{f}}$ to a fiber $\Gamma$
of $E_{f}\to\chow_{0}$. Note that $\sO_{\hchow}(E_{f})|_{\Gamma}\simeq\sO_{\mP^{3}}(-2)$
and $\sO_{\hchow}(H_{\hchow})|_{\Gamma}\simeq\sO_{\mP^{3}}$. As we
note in the end of Subsection \ref{HilbChow}, $E_{f}$ parameterizes
pairs of a point $x\in\mP(V)$ and a line $l$ through $x$. Therefore
a fiber $\Gamma\cong\mathbb{P}^{3}$ parameterizes lines through one
fixed point. This implies that $g^{*}\sF^{*}|_{\Gamma}\simeq\sO_{\mP^{3}}\oplus\sO_{\mP^{3}}(1)$.
Restricting the natural injection $\sO_{\hchow}(-H_{\hchow})\to g^{*}\ft{S}^{2}\sF$
to $\Gamma$, we have an injection \[
\sO_{\mP^{3}}\to\sO_{\mP^{3}}\oplus\sO_{\mP^{3}}(-1)\oplus\sO_{\mP^{3}}(-2).\]
 Therefore, by the definition of $\mathcal{F}_{1a}$, we have $\mathcal{F}_{1a}|_{\Gamma}\simeq\sO_{\mP^{3}}\oplus\sO_{\mP^{3}}(1)$.
Consequently, $\sC^{*}((t-5)H_{\hchow}+E_{f})|_{\Gamma}$ is a direct
sum of $\sO_{\mP^{3}}(-1)$, $\sO_{\mP^{3}}(-2)$ and $\sO_{\mP^{3}}(-3)$
for any $\sC=\sC_{ij}$ and $t$, hence all of its cohomology groups
vanish. \end{proof}

\vspace{0.5cm}
 \noindent\textit{\textcolor{black}{Proof of Theorem $\ref{thm:Gvan1}$.}}
Note that the subcategory $\sD_{\hchow}$ is admissible if the property
(1) holds \cite[Theorem 3.2. a)]{Bondal}. Then, for the proof of
(3), it suffices to verify \[
H^{\bullet}({\sK}_{i}^{*}\otimes{\sK}_{j}(-t))=0\;\;(1\leq i,j\leq4,1\leq t\leq4)\]
 for $\bullet>0$, since $\sD_{\hchow}$ is generated by $\sK_{i}(1\leq i\leq4)$
(cf. \cite[Lemma 2.2]{KuzGrIsotropic}). Therefore, the claims (1),(2),(3)
follow from explicit evaluations of the cohomology groups. Since the
computations are similar, we only explain some of them below. Note
also that we may assume that $t=0,1,2$ by Lemma \ref{cla:key1},
which simplifies the computations considerably.

As for $H^{\bullet}(\hchow,\sC_{44}(-t))=H^{\bullet}(\hchow,\sO_{\hchow}(-t))$
$(0\leq t\leq2)$, these vanish except in the case where $t=0$ since
$g$ is a $\mP^{2}$-bundle and hence $Rg_{*}^{i}\sO_{\hchow}(-t)=0$
for $t=1,2$ and $i\geq0$. $H^{\bullet}(\hchow,\sO_{\hchow})$ vanish
except $H^{0}(\hchow,\sO_{\hchow})$ by the Kodaira vanishing theorem.

As for $H^{\bullet}(\hchow,\sC_{34}(-t))=H^{\bullet}(\hchow,g^{*}\sF^{*}(-tH_{\hchow}))$
$(0\leq t\leq2)$, these vanish except in the case where $t=0$ by
a similar reason. We have \[
H^{\bullet}(\hchow,g^{*}\sF^{*})\simeq H^{\bullet}(\mathrm{G}(2,V),\sF^{*}),\]
 which vanish except $H^{0}(\mathrm{G}(2,V),\sF^{*})\simeq V^{*}$
by the Bott theorem \ref{thm:Bott}.

As for $H^{\bullet}(\hchow,\sC_{24}(-t))=H^{\bullet}(\hchow,\mathrm{Coker}\,\varphi\,(-t))$
($0\leq t\leq2$), consider the exact sequence \[
0\to\sO_{\hchow}(-(t+1)H_{\hchow}+2L_{\hchow})\to\ft{S}^{2}(g^{*}\sF^{*})(-t)\to\mathrm{Coker}\,\varphi\,(-t)\to0.\]
 Since $g$ is a $\mP^{2}$-bundle, $H^{\bullet}(\hchow,\sO_{\hchow}(-(t+1)H_{\hchow}+2L_{\hchow}))$
vanish except possibly in the case where $t=2$, and $H^{\bullet}(\hchow,\ft{S}^{2}(g^{*}\sF^{*})(-t))$
vanish except possibly in the case where $t=0$. Since $K_{\hchow}=-3H_{\hchow}-2L_{\hchow}$,
each cohomology of $H^{\bullet}(\hchow,\sO_{\hchow}(-3H_{\hchow}+2L_{\hchow}))$
is Serre dual to $H^{8-\bullet}(\hchow,\sO_{\hchow}(-4L_{\hchow}))$,
which is isomorphic to $H^{8-\bullet}(\mathrm{G}(2,V),\sO_{\mathrm{G}(2,V)}(-4))$,
and hence vanishes. We have $H^{\bullet}(\hchow,\ft{S}^{2}(g^{*}\sF^{*}))\simeq H^{\bullet}(\mathrm{G}(2,V),\ft{S}^{2}\sF^{*})$,
which vanish except $H^{0}(\mathrm{G}(2,V),\ft{S}^{2}\sF^{*})\simeq\ft{S}^{2}V^{*}$
by the Bott theorem \ref{thm:Bott}. \hfill $\square$

\vspace{1cm}

The triangulated subcategory generated by the dual Lefschetz collection
is contained in the derived category $\sD^{b}(\hchow)$, \begin{equation}
\langle\sD_{\hchow}(-4),\dots,{\sD_{\hchow}(-1),}\,\sD_{\hchow}\rangle\subset\sD^{b}(\hchow).\label{eqn:Lefschetz-D-X}\end{equation}
Since $\hchow\to\chow$ is a resolution of rational singularities
whose exceptional locus is an irreducible divisor $E_{f}$, we have
a triangulated subcategory $\hat{\sD}\subset\sD^{b}(\hchow)$ called
a categorical resolution of $\sD^{b}(\chow)$ for every dual Lefschetz
decomposition of $\sD^{b}(E_{f})$ \cite[Theorem 1]{Lef}. There is
a natural dual Lefschetz decomposition of $\sD(E_{f})$ for the $\mP^{3}$-fibration
$E_{f}\to\chow_{0}=v_{2}(\mP^{4})$ \cite{Sa}. Then we can verify
that the categorical resolution $\check{\sD}$ is strongly crepant
since $\chow$ is Gorenstein and the dual Lefschetz decomposition
is rectangular and also has length equal to the discrepancy of the
resolution $f$.

It is expected that the subcategory (\ref{eqn:Lefschetz-D-X}) coincides
with the categorical, strongly crepant resolution $\check{\sD}$.
Namely, the dual Lefschetz collection is expected to give the dual
Lefschetz \textit{decomposition} of the triangulated subcategory $\check{\sD}$.
It is also expected that $\check{\sD}$ is equivalent to the noncommutative
resolution $\sD^{b}(\chow,\sR)$ of $\sD^{b}(\chow)$ associated to
the sheaf \[
\sR:=f_{*}\sE\!\text{{\it {nd}}}\,(\sO_{\hchow}\oplus\sO_{\hchow}(-L_{\hchow})).\]
 Detailed study will be presented in a future publication \cite{HoTa4}.

In the next section, we will introduce the double cover $\hcoY$ of
the symmetroid $\Hes$. After making a nice resolution of $\widetilde{\hcoY}\to\hcoY$
in Section \ref{sec:The-resolution-Y}, we find a Lefschetz collection
in the derived category $\sD^{b}(\widetilde{\hcoY})$ in Theorem \ref{thm:Gvan},
\begin{equation}
\langle\sD_{\widetilde{\hcoY}},\sD_{\widetilde{\hcoY}}(1),\cdots\sD_{\widetilde{\hcoY}}(9)\rangle\subset\sD^{b}(\widetilde{\hcoY}).\label{eqn:Lefschetz-D-Y}\end{equation}
 Since the singularity of the double symmetroid $\hcoY$ is complicated,
the theorem \cite[Theorem 1]{Lef} seems to be difficult to apply
for the resolution $\widetilde{\hcoY}\to\hcoY$ to obtain the categorical
resolution $\widetilde{\sD}$ of $D^{b}(\hcoY)$. However, we expect
that the Lefschetz collection (\ref{eqn:Lefschetz-D-Y}) gives a Lefschetz
decomposition of a categorical crepant resolution, if exists, of $\sD^{b}(\hcoY)$. 

\vspace{1cm}

\subsection{Calabi-Yau manifold $X$ of a Reye congruence}

\label{subsection:CYX}

For generality of arguments below, consider $\chow$ for any $n>0$.
We choose a general linear subsystem $P$ in $|\sO_{\chow}(1)|$ of
dimension $n$. We regard $P$ as a general $n$-plane in $\mP(\ft{S}^{2}V^{*})$
associated with a linear subspace $L\simeq\mC^{n+1}\subset\ft{S}^{2}V^{*}$,
i.e., $P=\mP(L)$. Explicitly we assume the form $P=|Q_{A_{1}},Q_{A_{2}},\cdots,$
$Q_{A_{n+1}}|$ with the quadratic forms $q_{A_{i}}$ $(1\leq i\leq n+1)$
on $\mP(V)$.

Let $P^{\perp}=\mP(L^{\perp})\subset\mP(\ft{S}^{2}V)$ be the orthogonal
projective space with $L^{\perp}\subset\ft{S}^{2}V$, and define $X:=\chow\cap P^{\perp}$.
$X$ is naturally identified with the complete intersection in $\chow=\ft{S}^{2}\mP(V)$,
\[
X=\{A_{1}\cdot w_{\bm{xy}}=\cdots=A_{n+1}\cdot w_{\bm{xy}}=0\}\cap\chow.\]
Since $P$ is general, the orthogonal space $P^{\perp}$ is disjoint
from $\Sing\chow$ from dimensional reason, and hence $X$ is smooth
by the Bertini theorem. Conversely several properties of $X$ follow
from assuming only that $X$ is smooth; first of all, smoothness implies
that $P^{\perp}$ is disjoint from $\Sing\chow=\chow_{0}$, hence
we may consider $X$ is embedded in ${\hchow}$. Moreover, we see
the following properties:

\begin{prop} \label{prop:Reye} The morphism $g\colon\hchow\to\mathrm{G}(2,V)$
gives an isomorphism $X\cong g(X)$. The linear system $P$ of quadrics
is free from the base points. \end{prop}

\begin{proof} Both assertions follow from $X\cap\chow_{0}=\emptyset$
in $\chow$.

We start with the first assertion. Since $X$ is fiberwise a linear
section with respect to $g$, we have only to show that if a fiber
$\ell$ of $g$ intersects $X$, then $X\cap\ell$ consists of only
one point. Assume the contrary. Then $X\cap\ell$ is a linear subspace
of $\ell$ of positive dimension. Since $E_{f}\cap\ell$ is a conic
in $\ell$, we have $E_{f}\cap X\not=\emptyset$, a contradiction
to that $X\cap\chow_{0}=\emptyset$ in $\chow$.

As for the second assertion, we note that the base locus of $P$ is
given by $\{\bm{x}\mid A\cdot w_{{\bm{x}}{\bm{x}}}={\empty^{t}{\bm{x}}}A{\bm{x}}=0\ \text{for any}\ [A]\in P\}$,
where $A\cdot w_{{\bm{x}}{\bm{x}}}={\empty^{t}{\bm{x}}}A{\bm{x}}$
follows from (\ref{eq:pairing}). This is empty if and only if $X\cap\chow_{0}=\emptyset$.
\end{proof}

\textit{$\;$}

\textit{Reye congruence} is a line congruence in $G(2,V)$, which
is nothing but the isomorphic image of $X$ under $g:{\hchow}\to\mathrm{G}(2,V)$
(cf.~\cite{Ol}). In this paper we often identify $X$ with $g(X)$.
Below, we present a characterization of $X\subset\mathrm{G}(2,V)$
by the projective geometry of quadrics in $P$. For a point $[l]\in\mathrm{G}(2,V)$,
we denote by $P_{l}$ the subspace of $P$ which parameterizes quadrics
in $\mP(V)$ containing the line $l$, namely, if we write $l=\langle\bm{x},\bm{y}\rangle$
with $\bm{x},\bm{y}\in\mP(V)$, then $P_{l}=\{[A]\in P\mid{\empty^{t}{\bm{x}}}A{\bm{y}}={\empty^{t}{\bm{x}}}A{\bm{x}}={\empty^{t}{\bm{y}}}A{\bm{y}}=0\}$.

\begin{prop} A point $[l]\in\mathrm{G}(2,V)$ is contained in $X\subset\mathrm{G}(2,V)$
if and only if $P_{l}$ is an $(n-2)$-plane. \end{prop}

\begin{proof} First we show that, for any $[l]\in\mathrm{G}(2,V)$,
$\dim P_{l}\leq n-2$. Assume by contradiction that $\dim P_{l}\geq n-1$.
Then all the quadrics $[Q_{A}]\in P$ contains $l$ or the restrictions
of quadrics $[Q_{A}]\in P$ to $l$ reduce to a unique quadric on
$l$. This is a contradiction to the second statement of Proposition
\ref{prop:Reye}.

Assume that $[l]\in X$. We have only to show $\dim P_{l}\geq n-2$.
By Proposition \ref{prop:Reye}, there exists a unique $w_{\bm{x}\bm{y}}\in X\subset\chow$
such that $l=\langle\bm{x},\bm{y}\rangle$. Since $X=\chow\cap P^{\perp}$,
we have ${\empty^{t}{\bm{x}}}A{\bm{y}}=A\cdot w_{{\bm{x}}{\bm{y}}}=0$
for any $[A]\in P$ by (\ref{eq:pairing}). Therefore $P_{l}=\{[A]\in P\mid{\empty^{t}{\bm{x}}}A{\bm{x}}={\empty^{t}{\bm{y}}}A{\bm{y}}=0\}$,
and hence $\dim P_{l}\geq n-2$.

Conversely, assume that $\dim P_{l}=n-2$. We may choose a basis $A_{1},\dots,A_{n+1}$
of $P$ such that $A_{1},\dots,A_{n-1}$ form a basis of $P_{l}$.
Then the restrictions of quadrics $[Q_{A}]\in P$ to $l$ form the
pencil of quadrics on $l$ spanned by $Q_{A_{n}}|_{l}$ and $Q_{A_{n+1}}|_{l}$.
The corresponding pencil of symmetric $(1,1)$-divisors on $l\times l$
has two base points $({\bm{x}},{\bm{y}})$ and $({\bm{y}},{\bm{x}})$
with ${\bm{x}}\not={\bm{y}}$ since the base points of this pencil
are disjoint from the diagonal. Then note that $l=\langle\bm{x},\bm{y}\rangle$.
By the definition of $({\bm{x}},{\bm{y}})$, we have ${\empty^{t}{\bm{x}}}A_{n}{\bm{y}}={\empty^{t}{\bm{x}}}A_{n+1}{\bm{y}}=0$.
By the choice of $A_{1},\dots,A_{n-1}$, we have ${\empty^{t}{\bm{x}}}A_{1}{\bm{y}}=\cdots={\empty^{t}{\bm{x}}}A_{n-1}{\bm{y}}=0$.
Therefore ${\empty^{t}{\bm{x}}}A{\bm{y}}=0$ for any $[A]\in P$.
This implies that $w_{\bm{x}\bm{y}}\in X$ by (\ref{eq:pairing}).
\end{proof}

$\;$ 

When $n$ is even, $X$ is a Calabi-Yau $(n-1)$-fold and satisfies: 

\begin{prop} $\pi_{1}(X)\simeq\mZ_{2}$ and $\Pic X\simeq\mZ\oplus\mZ_{2}$,
where the free part of $\Pic X$ is generated by the class $D$ of
a hyperplane section of $\chow$ restricted to $X$. \end{prop}

\begin{proof} Consider the complete intersection $\widetilde{X}$
in $\mP(V)\times\mP(V)$ defined by the pull-back of $P$. Then we
have the projection morphism $\Lpi_{\widetilde{X}}\colon\widetilde{X}\to X$.

By the Lefschetz theorem, $\pi_{1}(\widetilde{X})=\{1\}$. Since the
map $\Lpi_{\widetilde{X}}$ is an \'etale double cover, we have $\pi_{1}({X})\simeq\mZ_{2}$.

Let $E$ be any effective divisor on $X$. Since $\Lpi_{\widetilde{X}}^{\;*}E$
is $\mZ_{2}$-invariant, it is of type $(m,m)$ with some non-negative
integer $m$. We may choose a homogeneous equation $F_{E}$ of $\Lpi_{\widetilde{X}}^{\;*}E$
as symmetric or skew-symmetric. If $F_{E}$ is symmetric, then $E\sim mD$.
Assume that $F_{E}$ is skew-symmetric. Let $\widetilde{D}$ be a
skew-symmetric $(1,1)$-divisor. Then $\Lpi_{\widetilde{X}}^{\;*}E-m\widetilde{D}=\divi(\alpha)$,
where $\alpha$ is a $\mZ_{2}$-invariant rational function of $\widetilde{X}$.
Then $\alpha$ is the pull-back of a rational function $\beta$ of
${X}$, and we see that $E-m\Lpi_{\widetilde{X}*}\tilde{D}=\divi(\beta)$
by looking at the zero and pole of $\beta$. Therefore $\Pic X$ is
generated by the classes of $D$ and $\Lpi_{\widetilde{X}*}\widetilde{D}$.
Since $2D\sim2\Lpi_{\widetilde{X}*}{\widetilde{D}}$, we conclude
that $\Pic X\simeq\mZ\oplus\mZ_{2}$. \end{proof}

When $n=4$, $X$ is a Calabi-Yau threefold with the following invariants
\cite[Proposition 2.1]{HoTa1}: \[
\deg(X)=35,\;\; c_{2}.D=50,\;\; h^{2,1}(X)=26,\; h^{1,1}(X)=1,\]
 where $c_{2}$ is the second Chern class of $X$.


\newpage{}

\section{The double symmetroid $\hcoY$ and Calabi-Yau variety $Y$}

\label{section:double}

Consider Calabi-Yau $(n-1)$-fold $X$ of a Reye congruence for arbitrary
even $n$. Under the projective duality (\ref{eq:tower2}), we find
that $X$ is naturally paired with another Calabi-Yau variety $H$
in the symmetroid $\Hes$ (see \cite{HoTa1} for $n=4$). The geometry
of the symmetroid $\Hes$ was studied in detail by Tjurin \cite{Tj}.
In particular, he defined a double cover $\hcoY\to\Hes$, which we
call \textit{double symmetroid}. In this section, we elaborate Tjurin's
construction. By considering linear sections of this double cover,
we obtain a Calabi-Yau variety $Y$ for arbitrary even $n$. For $n=4$,
$Y$ turns out to be a smooth Calabi-Yau threefold. These Calabi-Yau
threefolds $X$ and $Y$ arise naturally from the (dual) Lefschetz
collections (\ref{eqn:Lefschetz-D-X}) and (\ref{eqn:Lefschetz-D-Y})
assuming the projective homological duality \cite{HPD1,HPD2}.

\subsection{Resolution $\UU$ of $\Hes$}

\label{subsection:U}

Let us define the following projective bundle over $\mP(V)$: \[
\UU=\mP(\ft{S}^{2}(\Omega(1))),\]
where $\Omega(1)$ is the cotangent sheaf of $\mP(V)$. Considering
the (dual of the) Euler sequence $0\to\sO(-1)\to V\otimes\sO\to T(-1)\to0$,
we see that there is a canonical injection $\Omega(1)\hookrightarrow V^{*}\otimes\sO$,
which entails the injection $\ft{S}^{2}\Omega(1)\hookrightarrow\ft{S}^{2}V^{*}\otimes\sO.$
With this injection, we have a morphism\begin{equation}
i_{u}:\UU=\mP(\ft{S}^{2}(\Omega(1)))\to\mP(\ft{S}^{2}V^{*}).\label{eq:UUtoS2V}\end{equation}

\begin{prop}

{\rm (1)} The image of the morphism $i_{u}$ coincides with $\Hes$.
In particular, the morphism gives a resolution of $\Hes$.

\noindent{\rm (2)} $\UU\simeq\{([\bm{x}],[A])\mid A{\bm{x}}={\bf {0}}\}\subset\mP(V)\times\mP(\ft{S}^{2}V^{*})$.

\end{prop}\begin{proof}

(1) Since the fiber of $\Omega(1)$ over a point $[V_{1}]\in\mP(V)$
is $(V/V_{1})^{*}$, the fiber of the projective bundle $\UU\to\mP(V)$
over $[V_{1}]$ is given by $\mP(\ft{S}^{2}(V/V_{1})^{*})$. The morphism
$i_{u}$ sends $\mP(\ft{S}^{2}(V/V_{1})^{*})$ into $\mP(\ft{S}^{2}V^{*})$.
Then the image is identified with quadrics in $\mP(V)$ which are
singular at $[V_{1}]$, or equivalently, symmetric matrices whose
kernels contain $[V_{1}]$. Hence the image is contained in the symmetroid
$\Hes$. The surjectivity is clear since $\Hes$ consists of singular
symmetric matrices. Finally, $\UU$ is smooth since it is a projective
bundle. 

(2) Let us denote the r.h.s.$\;$by $\UU'$. There is a natural projection
from $\UU'$ to $\mP(V)$. Then the fiber $\UU'_{[\bm{x}]}$ over
$[\bm{x}]\in\mP(V)$ is the projective space of singular symmetric
matrices whose kernels contain $V_{1}=\mC\bm{x}$. Namely, $\UU'_{[\bm{x}]}$
coincides with the isomorphic image of $\mP(\ft{S}^{2}(V/V_{1})^{*})$
above. Hence $\UU'\simeq\UU$. \end{proof}

We summarize the resolution in the diagram,\[
\xymatrix{\mP(V) & \ar[l]_{\pi_{\UU}\qquad}\UU=\mP(\ft{S}^{2}(\Omega(1)))\ar[r]^{\qquad\quad i_{u}} & \Hes.}
\]

$\;$

\subsection{The double covering $\hcoY$ of $\Hes$}

\label{subsection:double}

Here we construct the double cover $\hcoY$ of $\Hes$ by considering
$\frac{n}{2}$-planes contained in each singular quadric.

Let us first consider a variety $\Zpq$ which parameterizes the pairs
of quadrics $Q$ and $\frac{n}{2}$-planes $\mP(\Pi)$ such that $\mP(\Pi)\subset Q$.
To parametrize $\frac{n}{2}$-planes in $\mP(V)$, consider the Grassmann
$\mathrm{G}(\frac{n+2}{2},V)$. Let \begin{equation}
0\to{\eQ}^{*}\to V^{*}\otimes\sO_{\mathrm{G}(\frac{n+2}{2},V)}\to\eS^{*}\to0\label{eq:Q*S}\end{equation}
 be the dual of the universal exact sequence on $\mathrm{G}(\frac{n+2}{2},V)$,
where $\eQ$ is the universal quotient bundle of rank $\frac{n}{2}$
and $\eS$ is the universal subbundle of rank $\frac{n+2}{2}$. For
an $\frac{n}{2}$-plane $\mP(\Pi)\subset\mP(V)$, there exists a natural
surjection $\ft{S}^{2}V^{*}\to\ft{S}^{2}H^{0}(\mP(\Pi),\sO_{\mP(\Pi)}(1))$
such that the projectivization of the kernel consisting of the quadrics
containing $\mP(\Pi)$. By relativizing this surjection over $\mathrm{G}(\frac{n+2}{2},V)$,
we obtain the following surjection: $\ft{S}^{2}V^{*}\otimes\sO_{\mathrm{G}(\frac{n+2}{2},V)}\to\ft{S}^{2}\eS^{*}.$
Let $\sE^{*}$ be the kernel of this surjection, and consider the
following exact sequence: \begin{equation}
0\to\sE^{*}\to\ft{S}^{2}V^{*}\otimes\sO_{\mathrm{G}(\frac{n+2}{2},V)}\to\ft{S}^{2}\eS^{*}\to0.\label{eq:sE0}\end{equation}
 Set $\Zpq=\mP(\sE^{*})$ and denote by $\Lrho_{\Zpq}$ the projection
$\Zpq\to\mathrm{G}(\frac{n+2}{2},V)$. By (\ref{eq:sE0}), $\Zpq$
is contained in $\mathrm{G}(\frac{n+2}{2},V)\times\mP(\ft{S}^{2}V^{*})$.
Since the fiber of $\sE^{*}$ over $[\Pi]$ parameterizes quadrics
in $\mP(V)$ containing $\mP(\Pi)$, we have \[
\Zpq=\{([\Pi],[Q])\mid\mP(\Pi)\subset Q\}\subset\mathrm{G}(\frac{n+2}{2},V)\times\mP(\ft{S}^{2}V^{*}).\]
 Note that $Q$ in $([\Pi],[Q])\in\Zpq$ is a singular quadric since
a smooth quadric does not contain $\frac{n}{2}$-planes. Hence the
symmetroid $\Hes$ is the image of the natural projection $\Zpq\to\mP(\ft{S}^{2}V^{*})$.
Now we introduce \[
\xymatrix{ & \Zpq\;\ar[r]^{\;\;\Lpi_{\Zpq}\;\;} & \;\hcoY\;\ar[r]^{\;\;\Lrho_{\hcoY}\;\;} & \;\Hes}
\]
by the Stein factorization of $\Zpq\to\Hes$. By (\ref{eq:sE0}),
the tautological divisor $H_{\mP(\sE^{*})}$ of $\mP(\sE^{*})\to\mathrm{G}(\frac{n+2}{2},V)$
is nothing but the pull-back of a hyperplane section of $\Hes$. We
set \[
M_{\Zpq}:=H_{\mP(\sE^{*})}=\Lpi_{\Zpq}^{\;*}\circ\Lrho_{\hcoY}^{\;*}\sO_{\Hes}(1).\]
 We denote by $\Zpq_{[Q]}$ the fiber of $\Zpq\to\Hes$ over a point
$[Q]\in\Hes$.

\vspace{0.3cm}
 \begin{lem} \label{Z_Q} For a quadric $Q$ of rank $n$, the fiber
${\Zpq}_{[Q]}$ is the orthogonal Grassmann $\OG(\frac{n}{2},n)$
\textcolor{black}{which consists of two connected components. }\end{lem}

\begin{proof} The quadric $Q$ of rank $n$ induces a non-degenerate
symmetric bilinear form $q$ on the quotient $V/V_{1}$, where $V_{1}$
is the $1$-dimensional vector space such that $[V_{1}]$ is the vertex
of $Q$. Then $\frac{n}{2}$-planes on $Q$ naturally correspond to
the maximal isotropic subspaces in $V/V_{1}$ with respect to $q$,
which are parameterized by the orthogonal Grassmann $\OG(\frac{n}{2},n)$.
\end{proof}

\vspace{0.3cm}
 \begin{prop} \label{cla:double} The morphism $\hcoY\to\Hes$ is
of degree two and is branched along the locus of quadrics of rank
less than or equal to $n-1$. \end{prop} 

\begin{proof} By Lemma~\ref{Z_Q}, the degree of $\hcoY\to\Hes$
is two since $\Zpq_{[Q]}$ has two connected components for a quadric
$Q$ of rank $n$. If a quadric $Q$ has rank less than or equal to
$n-1$, the family of $\frac{n}{2}$-planes in $Q$ is connected.
Hence we have the assertion. \end{proof}

By this proposition, we see that $\hcoY$ parameterizes connected
families of $\frac{n}{2}$-planes in singular quadrics in $\mP(V)$
(cf. Fig.1).

\vspace{0.3cm}
 \begin{defn} \label{defn:Gy} Related to the morphism $\Lrho_{\hcoY}:\hcoY\to\Hes$,
we define $G_{\hcoY}$ to be the inverse image by $\Lrho_{\hcoY}$
of the locus of quadrics of rank less than or equal to $n-2$.~{\hfill{}{[}{]}}
\end{defn}

Since $G_{\hcoY}$ is contained in the ramification locus of $\Lrho_{\hcoY}$,
it is clearly isomorphic to the locus of quadrics of rank less than
or equal to $n-2$.

$\hcoY$ has the following nice properties in view of the minimal
model program.

\vspace{0.3cm}
 \begin{prop} \label{cla:ZY} The Picard number of $\Zpq$ is two
and $\Lpi_{\Zpq}\colon\Zpq\to\hcoY$ is a Mori fiber space. In particular,
$\hcoY$ is a $\mQ$-factorial Gorenstein canonical Fano variety with
Picard number one. Moreover, $\hcoY$ is smooth outside $G_{\hcoY}$
and the Fano index of $\hcoY$ is $\frac{n(n+1)}{2}$. \end{prop} 

\begin{proof} Since $\Zpq$ is a projective bundle over $\mathrm{G}({\frac{n+2}{2}},V)$,
its Picard number is two. {Note that} the relative Picard number
of $\Lpi_{\Zpq}\colon\Zpq\to\hcoY$ is one, then {we see that} the
Picard number of $\hcoY$ is one. Since a general fiber of $\Lpi_{\Zpq}$
is a Fano variety by Lemma \ref{Z_Q}, $\Lpi_{\Zpq}$ is a Mori fiber
space. By \cite[Lemma 5-1-5]{KMM}, $\hcoY$ is $\mQ$-factorial,
and, by \cite[Corollary 4.6]{Fuj}, $\hcoY$ has only kawamata log
terminal singularities. $\hcoY$ is a Gorenstein Fano variety since
it is a double cover of a Gorenstein Fano variety $\Hes$ and the
ramification locus has codimension greater than one by Proposition
\ref{cla:double}. Thus $\hcoY$ has only canonical singularities.

As we mentioned in Subsection \ref{subsection:PD}, the singular locus
of $\Hes$ is the locus of quadrics of rank less than or equal to
$n-1$. Since $\Lrho_{\hcoY}\colon\hcoY\to\Hes$ is \'etale outside
this locus, $\hcoY$ is smooth outside the inverse image of this locus.
Moreover, by the proof of \cite[Lemma 3.6]{HoTa1}, $\Hes$ has only
ordinary double points along the locus of quadrics of rank $n-1$.
Since $\Lrho_{\hcoY}$ is ramified along this locus, $\hcoY$ is smooth
along the inverse image of this locus. Therefore $\hcoY$ is smooth
outside $G_{\hcoY}$.

Since $K_{\Hes}=\sO_{\Hes}(\frac{n(n+1)}{2})$, the Fano index of
$\hcoY$ is $\frac{n(n+1)}{2}$. \end{proof}

\begin{rem} We will show that $\hcoY$ has only terminal singularities
when $n=4$ (see Proposition \ref{cla:F}). \hfill{}{[}{]} \end{rem}

\vspace{0.5cm}

When $n=4$, we have more detailed descriptions of the fibers of $\Lrho_{\hcoY}$.

\begin{prop} \label{cla:double2} If $\rank Q=4$, then $\Zpq_{[Q]}$
is a disjoint union of two smooth rational curves. Each connected
component is identified with a conic on $\mathrm{G}(3,V)$. If $\rank Q=3$,
then $\Zpq_{[Q]}$ is a smooth rational curve, which is also identified
with a conic on $\mathrm{G}(3,V)$. If $\rank Q=2$, then $\Zpq_{[Q]}$
is the union of two $\mP^{3}$'s intersecting at one point. If $\rank Q=1$,
then $\Zpq_{[Q]}$ is a $($non-reduced\,$)$ $\mP^{3}$. \end{prop} 

\begin{proof} If $\rank Q=4$, the fiber $\Zpq_{[Q]}$ consists of
two disconnected components, and is isomorphic to the orthogonal Grassmann
$\OG(2,4)$ by Lemma~\ref{Z_Q}. To be more explicit, let $[V_{1}]\in\mP(V)$
be the vertex of $Q$. Then the quadric $Q$ is the cone over $\mP^{1}\times\mP^{1}$
with the vertex $[V_{1}]$. There are two distinct $\mP^{1}$-families
of lines on $\mP^{1}\times\mP^{1}$. Each of the families can be understood
as a corresponding conic on $\mathrm{G}(2,V/V_{1})$, which gives
one of the connected components of $\OG(2,4)$. Under the natural
map $\mathrm{G}(2,V/V_{1})\rightarrow\mathrm{G}(3,V)$, we have a
$\mP^{1}$ family of $2$-planes on $Q$ parameterized by a conic
on $\mathrm{G}(3,V)$.

If $\rank Q=3$, the vertex of the quadric $Q$ is a line $[V_{2}]\subset\mP(V)$.
The quadric $Q$ is the cone over a conic with the vertex $[V_{2}]$.
The conic is contained in $\mP(V/V_{2})=\mathrm{G}(1,V/V_{2})$, and
can be identified with a conic in $\mathrm{G}(3,V)$ under the natural
map $\mathrm{G}(1,V/V_{2})\rightarrow\mathrm{G}(3,V)$.

If $\rank Q=2$, then the quadric $Q$ has a vertex $[V_{3}]\subset\mP(V)$
and is the union of two $3$-planes intersecting along the $2$-plane
$[V_{3}]$. Hence $\Zpq_{[Q]}\subset\mathrm{G}(3,V)$ is given by
the union of the corresponding $\mP^{3}$'s, i.e., $\mathrm{G}(3,4)$'s
in $\mathrm{G}(3,V)$, which intersect at one point $[V_{3}]$.

If $\rank Q=1$, then $Q$ is a double $3$-plane. Thus $\Zpq_{[Q]}$
is a (non-reduced) $\mP^{3}\cong\mathrm{G}(3,4)$. \end{proof}

\[
\FigQuadDisplay\]
 \vspace{0.2cm}
 \begin{fcaption} 

\item \textbf{Fig.1. Quadrics $Q$ in $\mP(V)$ and families of planes
therein.} The singular loci of $Q$ are written by $[V_{k}]$ with
$k=5-{\rm {rk}\, Q}$. Also the parameter spaces of the planes in
each $Q$ are shown. See also Fig.2 in Section \ref{sec:The-resolution-Y}.
\end{fcaption} 
\vspace{0.5cm}

We write by $G_{\hcoY}^{1}$ (resp. $G_{\hcoY}^{2}$) the inverse
image under $\Lrho_{\hcoY}$ of the locus of quadrics of rank one
(resp. two). We see that $G_{\hcoY}\simeq\ft{S}^{2}\mP(V^{*})$, $G_{\hcoY}^{1}\simeq v_{2}(\mP(V^{*}))$
and $G_{\hcoY}^{2}=G_{\hcoY}\setminus G_{\hcoY}^{1}$. Using these,
we summarize our construction above for $n=4$ in the following diagram:
\begin{equation}
\xyHYZG\label{eq:Z}\end{equation}
where $\Lpi_{\Zpq}$ is a $\mP^{1}$-fibration over $\hcoY\setminus G_{\hcoY}$
by Proposition \ref{cla:double2}. In Section \ref{sec:The-resolution-Y},
we will construct a nice desingularization $\widetilde{\hcoY}$ of
$\hcoY$. There we will also study the geometry of $\widetilde{\hcoY}\to\hcoY$
along the loci $G_{\hcoY}$ and $G_{\hcoY}^{1}$ in full details. 

$\;$

\subsection{Calabi-Yau variety $Y$}

\label{subsection:CY3Y}

Assume that a Reye congruence Calabi-Yau $(n-1)$-fold $X=\chow\cap P^{\perp}$
is given by a general $n$-plane $P$ in $\mP(\ft{S}^{2}V^{*})$,
i.e., $P=\mP(L)$ with $L\simeq\mC^{n+1}\subset\ft{S}^{2}V^{*}$.
Define $H:=\Hes\cap P$. Then $H$ is a determinantal hypersurface
of degree $n+1$ in $P\simeq\mP(L)$ and hence the canonical bundle
of $H$ is trivial.

Let $Y$ be the pull-back of $H$ on $\hcoY$. According to \cite{HPD1},
we say that $Y$ \textit{is orthogonal} to $X$ and vice versa.

\vspace{0.3cm}
 \begin{Screen} \begin{prop} \label{prop:CYgen} $Y$ is a Calabi-Yau
variety. \end{prop} \end{Screen}

\begin{proof} The canonical divisor $K_{Y}$ is trivial since $K_{H}\sim0$
and the branch locus of $Y\to H$ has the codimension greater than
or equal to two.

By Proposition \ref{cla:ZY}, $\hcoY$ has only Gorenstein canonical
singularities and $h^{i}(\sO_{\hcoY})=0$ for $i>0$. Then, by a version
of the Bertini theorem,\textcolor{red}{{} }\textcolor{black}{which says
that general sections of Gorenstein canonical varieties are also Gorenstein
and canonical}, we see that $Y$ has only Gorenstein canonical singularities.
It is standard to derive $h^{i}(\sO_{Y})=0$ for $0<i<\dim Y$ from
$h^{i}(\sO_{\hcoY})=0$ for $i>0$ by the Kodaira-Kawamata-Viewheg
vanishing theorem. \end{proof}

In the rest of this paper, we restrict our attention to $n=4$. In
this case, $Y$ is smooth by Proposition \ref{cla:ZY} since $Y\cap G_{\hcoY}=\emptyset$
by dimensional reason. This Calabi-Yau manifold $Y$ coincides with
the double covering $Y$ defined in \cite{HoTa1}, where $Y$ is called
the \textit{$($shifted$)$ Mukai dual}\, to $X$. 

In the previous paper \cite[Prop.3.11 and Prop.3.12]{HoTa1}, we have
determined invariants of the Calabi-Yau threefold $Y$. Here we reproduce
these invariants using the construction summarized in (\ref{eq:Z}).

\;

\;

Let us first recall (\ref{eq:sE0}) for the definition of $\sE^{*}$
and set $\sE:=(\sE^{*})^{*}$. Then we have

\begin{lem} \label{cla:c1c2} $c_{1}(\sE)=c_{1}(\sO_{\mathrm{G}(3,V)}(4))$.
\end{lem}

\begin{proof} Note that $c_{1}(\sO_{\mathrm{G}(3,V)}(1))$ is given
by the Schubert cycle $\sigma_{1}$, which is $c_{1}(\eS^{*})$ in
our notation. Since ${\rm rk}\,\eS=3$, we have $c_{1}(\sE)=c_{1}(\ft{S}^{2}\eS^{*})=4c_{1}(\eS^{*})$.
Thus we have the assertion. \end{proof}

Now let us note the relative Euler sequence associated with the projective
bundle $\Lrho_{\Zpq}\colon\Zpq=\mP(\sE^{*})\to\mathrm{G}(3,V)$: \begin{equation}
0\to\sO_{\Zpq}\to\Lrho_{\Zpq}^{\;*}\sE^{*}(M_{\Zpq})\to T_{\Zpq}\to\Lrho_{\Zpq}^{\;*}T_{\mathrm{G}(3,V)}\to0.\label{eq:EulerZG}\end{equation}
 From this we obtain the following:

\vspace{0.3cm}
 \begin{lem} \label{lem:cohZ} Let $N_{\Zpq}:=\Lrho_{\Zpq}^{\;*}\sO_{G(3,V)}(1)$.
$K_{\Zpq}=-9M_{\Zpq}-N_{\Zpq}$ holds for the canonical divisor $K_{\Zpq}$
on $\Zpq$, and we have the following cohomologies for the sheaves
on $\Zpq$ and for $0\leq k\leq10$ $(-1\leq k\leq10$ for $(2))$:
\[
\begin{aligned}(1)\;\; & H^{\bullet}(\sO_{\Zpq}(-(k+1)M_{\Zpq}+N_{\Zpq}))=0.\\
(2)\;\; & H^{\bullet}(\sO_{\Zpq}(-kM_{\Zpq}))\simeq H^{\bullet}(\Lrho_{\Zpq}^{\;*}\sE^{*}(-(k\text{{--}}1)M_{\Zpq}))\simeq\begin{cases}
\ft{S}^{2}V & (\bullet,k)=(0,-1)\\
\mC & (\bullet,k)=(0,0),(13,10)\\
0 & (otherwise)\end{cases}.\\
(3)\;\; & H^{\bullet}(\Lrho_{\Zpq}^{\;*}T_{\mathrm{G}(3,V)}(-kM_{\Zpq}))\simeq H^{\bullet}(T_{\Zpq}(-kM_{\Zpq}))\simeq\begin{cases}
sl(V) & (\bullet,k)=(0,0)\\
\mC & (\bullet,k)=(12,10)\\
0 & (otherwise)\end{cases},\end{aligned}
\]
 \end{lem}

\begin{proof} The claimed formula of $K_{\Zpq}$ follows by taking
the determinant of the Euler sequence (\ref{eq:EulerZG}). We should
note that ${\rm rank}~\sE=9$ and $N_{\Zpq}=\Lrho_{\Zpq}^{\;*}\sO_{\mathrm{G}(3,V)}(1)$.

For the calculations of the cohomologies in (1), (2), (3), we use
the Serre duality, the Kodaira vanishing theorem and also the Bott
theorem \ref{thm:Bott} as well as the defining exact sequence (\ref{eq:sE0})
of $\sE^{*}$.

(1) By the Serre duality, we have $H^{\bullet}(\sO_{\Zpq}(-(k+1)M_{\Zpq}+N_{\Zpq})\simeq H^{14-\bullet}(\sO_{\Zpq}((k-8)M_{\Zpq}-2N_{\Zpq}))$.
From this, the claimed vanishings follow for the range $0\leq k\leq7$
for all $\bullet$ since $\Lrho_{\Zpq}\colon\Zpq\to G(3,V)$ is a
$\mP^{8}$-bundle. When $k=8$, the vanishing follows from $H^{14-\bullet}(\sO_{\Zpq}(-2N_{\Zpq}))\simeq H^{14-\bullet}(\mathrm{G}(3,V),\sO_{\mathrm{G}(3,V)}(-2))=0$.
When $k=9$, we need to evaluate $H^{14-\bullet}(\sO_{\Zpq}(M_{\Zpq}-2N_{\Zpq}))\simeq H^{14-\bullet}(\mathrm{G}(3,V),\sE(-2))$.
By tensoring the dual of (\ref{eq:sE0}) by $\sO_{\mathrm{G}(3,V)}(-2)$
and using the Bott theorem, it is easy to obtain the claimed vanishing.
When $k=10$, we have $H^{14-\bullet}(\sO_{\Zpq}(2M_{\Zpq}-2N_{\Zpq}))\simeq H^{14-\bullet}(\mathrm{G}(3,V),\ft{S}^{2}\sE(-2))$.
For the calculation of the cohomologies of $\ft{S}^{2}\sE(-2)=\ft{S}^{2}\sE\otimes\sO_{\mathrm{G}(3,V)}(-2)$,
we note the following exact sequence\footnote[1]{The authors would like to thank Prof. R.F.  for suggesting this exact sequence.}:
\[
0\to\wedge^{2}(\ft{S}^{2}\eS)\to\ft{S}^{2}V\otimes\ft{S}^{2}\eS\to\ft{S}^{2}(\ft{S}^{2}V)\otimes\sO_{\rG(3,V)}\to\ft{S}^{2}\sE\to0.\]
 Tensoring by $\sO_{\mathrm{G}(3,V)}(-2)$ and using the Bott theorem,
the claimed vanishing follows for all degree.

(2) Since the calculations of $H^{\bullet}(\sO_{\Zpq}(-kM_{\Zpq}))$
is easy, we omit them. For the cohomologies $H^{\bullet}(\Lrho_{\Zpq}^{\;*}\sE^{*}(-(k-1)M_{\Zpq}))$,
when $k=0$, we have to evaluate $H^{\bullet}(\Lrho_{\Zpq}^{\;*}\sE^{*}(M_{\Zpq}))=H^{\bullet}(\mathrm{G}(3,V),\sE^{*}\otimes\sE)$.
This can be done by considering two short exact sequences; one from
tensoring the defining exact sequence (\ref{eq:sE0}) by $\sE$ and
the other from tensoring the dual of (\ref{eq:sE0}) by $\ft{S}^{2}\eS^{*}$.
The cases of other values of $k$ are rather easy, so we omit their
details.

(3) For the calculations of $H^{\bullet}(\Lrho_{\Zpq}^{\;*}T_{\mathrm{G}(3,V)}(-kM_{\Zpq}))$,
we use $T_{\mathrm{G}(3,V)}=\eS^{*}\otimes\eQ$. For example, for
$k=0$, we evaluate $H^{\bullet}(\Lrho_{\Zpq}^{\;*}T_{\mathrm{G}(3,V)})=H^{\bullet}(\mathrm{G}(3,V),\eS^{*}\otimes\eQ)$,
which is non-vanishing only for $\bullet=0$ with the result $\Sigma^{(1,0,0,0,-1)}V^{*}\simeq sl(V)\simeq\mC^{\oplus24}$.
For $k\geq1$, use the Serre duality and the defining exact sequence
(\ref{eq:sE0}).

Finally the calculations of $H^{\bullet}(T_{\Zpq}(-kM_{\Zpq}))$ are
done with the relative Euler sequence (\ref{eq:EulerZG}) and also
using the results obtained so far. Since they are straightforward,
we omit them here. \end{proof}

Let $M$ be the pull-back of $\sO_{H}(1)$ to $Y$. The following
proposition refines the results in \cite[Propositions 3.11 and 3.12]{HoTa1}:

\vspace{0.5cm}
 \begin{Screen} \begin{prop} \label{prop:Y} The $3$-fold $Y$
is a simply connected Calabi-Yau $3$-fold such that $\Pic Y=\mZ[M]$,
${M}^{3}=10$, $c_{2}(Y).{M}=40$ and $e(Y)=-50$. In particular,
$h^{1,1}(Y)=1$ and $h^{1,2}(Y)=26$. \end{prop} \end{Screen}

\begin{proof} We have already shown that $Y$ is a smooth Calabi-Yau
threefold. Since $Y\to H$ is a double cover, we have $M^{3}=2c_{1}(\sO_{H}(1))^{3}=10$.

To calculate other invariants, we use the restriction of $\Lpi_{\Zpq}:\Zpq\to\hcoY$
over $Y$, which is a $\mP^{1}$-fibration. Set $Z:=\Lpi_{\Zpq}^{-1}(Y)$
and $\Lpi_{Z}$ be the restriction of $\Lpi_{\Zpq}$ to $Z$. We also
set $N_{Z}$ and $M_{Z}$ be the restrictions of $N_{\Zpq}$ and $M_{\Zpq}$
to $Z$, respectively.

Let us first note that we have $K_{Z}=M_{Z}-N_{Z}$ for the canonical
divisor by Lemma \ref{lem:cohZ}, since $Z$ is a complete intersection
of ten members of $|M_{\Zpq}|$. Also we note the following Koszul
resolution of $\sO_{Z}$ as a $\sO_{\Zpq}$-module: \begin{equation}
0\to\wedge^{10}\{\sO_{\Zpq}(-M_{\Zpq})^{\oplus10}\}\to\cdots\to\sO_{\Zpq}(-M_{\Zpq})^{\oplus10}\to\sO_{\Zpq}\to\sO_{Z}\to0.\label{eq:KosZ}\end{equation}

We observe the following isomorphisms: \begin{equation}
H^{\bullet}(Z,T_{Z})\simeq H^{\bullet}(Z,\Lpi_{Z}^{\;*}T_{Y})\simeq H^{\bullet}(Y,T_{Y}),\label{eq:ZTisoYT}\end{equation}
 where we note that $H^{\bullet}(Y,T_{Y})$ vanishes for $\bullet=0,3$
since $Y$ is a Calabi-Yau threefold. The second isomorphism follows
from the fact that $Z\to Y$ is a $\mP^{1}$-fibration. To see the
first isomorphism, let us note the exact sequence $0\to T_{Z/Y}\to T_{Z}\to\Lpi_{Z}^{\;*}T_{Y}\to0$,
from which we have $T_{Z/Y}=\sO_{Z}(-K_{Z})$ since $K_{Y}=0$ and
$T_{Z/Y}$ is an invertible sheaf. Then we have $H^{\bullet}(Z,T_{Z/Y})=H^{\bullet}(Z,\sO_{Z}(-M_{Z}+N_{Z}))$.
Tensoring the resolution (\ref{eq:KosZ}) by $\sO(-M_{\Zpq}+N_{\Zpq})$
and using (1) of Lemma~\ref{lem:cohZ}, we see the vanishing $H^{\bullet}(Z,T_{Z/Y})=0$.
This entails the first isomorphism.

Next let us consider the exact sequence \begin{equation}
0\to T_{Z}\to T_{\Zpq}|_{Z}\to\sO_{Z}(M_{Z})^{\oplus10}\to0.\label{eq:TZY}\end{equation}
 Since $Z\to Y$ is a $\mP^{1}$-fibration, we have $H^{\bullet}(Z,\sO_{Z}(M_{Z}))\simeq H^{\bullet}(Y,\sO_{Y}(M))$,
where the r.h.s.\,vanish by the Kodaira vanishing theorem except
$H^{0}(Y,\sO_{Y}(M))\simeq\mC^{5}$. Therefore, by (\ref{eq:ZTisoYT})
and (\ref{eq:TZY}), we have \begin{equation}
\begin{matrix}H^{\bullet}(T_{Z})\simeq H^{\bullet}(T_{\Zpq}|_{Z})\text{ for }\bullet\geq2\\
0\to H^{0}(T_{\Zpq}|_{Z})\to\mC^{\oplus50}\to H^{1}(T_{Z})\to H^{1}(T_{\Zpq}|_{Z})\to0.\end{matrix}\label{eq:TZYcoh}\end{equation}

We finally calculate the cohomology of the sheaf $T_{\Zpq}\vert_{Z}$
as \begin{equation}
H^{0}(T_{\Zpq}\vert_{Z})=\mC^{\oplus24},\;\; H^{2}(T_{\Zpq}\vert_{Z})=\mC,\;\; H^{i}(T_{\Zpq}\vert_{Z})=0\;(i\not=0,2).\label{eq:TzRestrict}\end{equation}
 These results follow from tensoring the resolution (\ref{eq:KosZ})
by $T_{\Zpq}$ and using (3) of Lemma~\ref{lem:cohZ}. Now combining
(\ref{eq:TzRestrict}) with (\ref{eq:TZYcoh}) and (\ref{eq:ZTisoYT}),
we obtain $h^{1}(T_{Z})=h^{1}(T_{Y})=26$ and $h^{2}(T_{Z})=h^{2}(T_{Y})=1$.

Since $Y$ is a Calabi-Yau threefold, we have $e(Y)=2(h^{2}(T_{Y})-h^{1}(T_{Y}))=-50$.
Also by the Riemann-Roch formula, and the vanishings derived above,
we have \[
\chi(Y,\sO_{Y}(M))=\frac{M^{3}}{6}+\frac{c_{2}(Y).M}{12}=\dim H^{0}(Y,\sO_{Y}(M))=5.\]
 From this, we obtain $c_{2}(Y).M=40$.

It remains to show the simply connectedness of $Y$. This will imply
$\Pic Y\simeq\mZ$ since we have already shown $\rho(Y)=h^{2}(T_{Y})=1$.
Let $W$ be a general $4$-dimensional complete intersection of $\hcoY$
by members of $|M_{\hcoY}|$ containing $Y$. Then $W$ is a Fano
$4$-fold, and moreover smooth, since $\hcoY$ is smooth outside $G_{\hcoY}$
by Proposition \ref{cla:ZY} and the codimension of $G_{\hcoY}$ in
$\hcoY$ is $5$. Hence we know that $W$ is simply connected. The
simply connectedness of $Y$ follows from the Lefschetz theorem since
$Y$ is an ample divisor on $W$. \end{proof}

In a separate publication \cite{HoTa3}, we will show the derived
equivalence $D^{b}(X)\simeq D^{b}(Y)$ using the properties of the
(dual) Lefschetz collections (\ref{eqn:Lefschetz-D-X}) and (\ref{eqn:Lefschetz-D-Y}).

$\;$

$\;$

$\;$

$\;$\newpage{}

\section{The resolution $\widetilde{\hcoY}\to\hcoY$\label{sec:The-resolution-Y}}

\textit{We will restrict our attention to the case of $n=4$ $(\dim V=5)$.}

\subsection{Conics and planes in $\mathrm{G}(3,V)$ \label{sub:Conics-and-planes}}

Recall the Stein factorization $\Zpq\to\hcoY\to\Hes$. As shown in
Proposition \ref{cla:double2}, the fiber of $\Zpq\to\hcoY$ over
$y\in\hcoY\setminus G_{\hcoY}$ is a smooth conic in $\rG(3,V)$ which
parametrizes planes contained in the corresponding quadric $\Lrho_{\hcoY}(y)=[Q_{y}]\in\Hes$.
Writing this conic by $q_{y}$, we can represent $y\in\hcoY\setminus G_{\hcoY}$
by the pair $([Q_{y}],q_{y})$. Let us note that each conic $q_{y}\subset\rG(3,V)$
determines a conic in $\mP(\wedge^{3}V)$ under the Pl\"ucker embedding,
and in turn, determines a plane $\mP_{q_{y}}^{2}$ in $\mP(\wedge^{3}V)$
which contains the conic. 

If $\rank Q_{y}=4$, the plane $\mP_{q_{y}}^{2}$ recovers the conic
by \[
\mP_{q_{y}}^{2}\cap\rG(3,V)\subset\mP(\wedge^{3}V).\]
This fact can be seen as follows: First note that $Q_{y}$ determines
a quadric $\overline{Q}_{y}$ in $\mP(V/V_{1})$ with $[V_{1}]$ being
the vertex of $Q_{y}$. Note also that $q_{y}$ determines a conic
$\bar{q}_{y}$ in $\rG(2,V/V_{1})$ uniquely and the corresponding
plane $\mP_{\bar{q}_{y}}^{2}\subset\mP(\wedge^{2}(V/V_{1}))$. Since
$\rG(2,V/V_{1})$ is a quadric in $\mP(\wedge^{2}(V/V_{1}))$ and
$\mP_{q_{y}}^{2}\not\subset\rG(2,V/V_{1})$ holds for $\rank Q_{y}=4$,
the intersection with $\mP_{\bar{q}_{y}}^{2}$ recovers the conic
$\bar{q}_{y}$ and also $q_{y}$. 

If $\rank Q_{y}=3$, the plane $\mP_{q_{y}}^{2}$ does not recover
the conic $q_{y}.$ To see this, we recall the fact that there are
two types of planes contained in $\mathrm{G}(3,V)$. The first one
is a plane which is given by \[
{\rm P}_{V_{2}}:=\{[\Pi]\in\mathrm{G}(3,V)\mid V_{2}\subset\Pi\}\cong\mP^{2}\]
 with some $V_{2}$. The second one is \[
{\rm P}_{V_{1}V_{4}}:=\{[\Pi]\in\mathrm{G}(3,V)\mid V_{1}\subset\Pi\subset V_{4}\}\cong\mP^{2}\]
 with some $V_{1}$ and $V_{4}$ satisfying $V_{1}\subset V_{4}$.
We call ${\rm {P}_{V_{2}}}$ \textit{a $\rho$-plane} and ${\rm {P}_{V_{1}V_{4}}}$
\textit{a $\sigma$-plane}. While these are (projective) planes in
$\rG(3,V)$, we may (and will) consider them in $\mP(\wedge^{3}V)$
under the Pl\"ucker embedding $\rG(3,V)\subset\mP(\wedge^{3}V)$.
When $\rank Q_{y}=3$, the conic $q_{y}$ parametrizes planes containing
the vertex $\mP(V_{2})$ of $Q_{y}$, hence the corresponding plane
$\mP_{q_{y}}^{2}$ is a $\rho$-plane and the intersection becomes
$\mP_{q_{y}}^{2}\cap\rG(3,V)=\mP_{q_{y}}^{2}$. 

In \cite[\S 3.1]{IM}, the Hilbert scheme of conics in $\mathrm{G}(3,V)$
has been studied in general. As an element of the Hilbert scheme,
every conic $q$ in $\rG(3,V)$ determines the corresponding conic
in $\mP(\wedge^{3}V)$ and then the corresponding plane $\mP_{q}^{2}$.
We follow \cite{IM} in the following definition:

\begin{defn} 

A conic $q$ is called $\tau$-conic if $\mP_{q}^{2}\not\subset\rG(3,V)$
holds. $q$ is called $\rho$-conic and $\sigma$-conic, respectively,
if the plane $\mP_{q}^{2}$ is a $\rho$-plane and $\sigma$-plane.
\hfill {\rm []} \end{defn} 

$\,$

Clearly the $\rho$-planes ${\rm P}_{V_{2}}(V_{2}\subset V)$ are
parametrized by $\rG(2,V)$ while $\sigma$-planes ${\rm P}_{V_{1}V_{4}}(V_{1}\subset V_{4}\subset V)$
are parametrized by the flag variety $F(1,4,V)\simeq\mP(\Omega_{\mP(V)}^{1})\simeq\mP(\Omega_{\mP(V^{*})}^{1})$.
Since $\rG(2,V)$ and $F(1,4,V)$ parametrize planes in $\mP(\wedge^{3}V)$,
these define subvarieties in $\rG(3,\wedge^{3}V)$, which we denote
by $\overline{\Prt}_{\rho}$ and $\overline{\Prt}_{\sigma}$, respectively.

\subsection{Smooth conics and conics of rank two \label{sub:Smooth-conics-and}}

Denote by $\eS$ the universal subbundle on $\mathrm{G}(3,V)$, and
regard $\mP(\eS)$ as the universal family of the planes in $\mP(V)$.
Consider the natural projection \[
\varphi_{\eS}:\mP(\eS)\to\mP(V).\]
For a conic $q\subset\mathrm{G}(3,V)$, we consider the restriction
$\mP(\eS\vert_{q})$ and its image under $\varphi_{\eS}.$

\begin{prop} \label{prop:Y34}

Smooth $\tau$- and $\rho$-conics correspond bijectively to points
in $\hcoY\setminus G_{\hcoY}$. 

\end{prop} 

\begin{proof} We have seen in the beginning of the preceding subsection
that each points $y\in\hcoY\setminus G_{\hcoY}$ determines a smooth
conic $q_{y}$ in $\rG(3,V)$ (see also Proposition \ref{cla:double2}). 

For the converse, suppose a smooth conic $q\subset\mathrm{G}(3,V)$
is given. Note that the dual bundle $\eS^{*}$ on $\mathrm{G}(3,V)$
restricts as $\eS^{*}|_{q}\simeq\sO(1)_{\mP^{1}}^{\oplus2}\oplus\sO_{\mP^{1}}$,
or $\sO_{\mP^{1}}(2)\oplus\sO_{\mP^{1}}^{\oplus2}$ since $\eS^{*}$
is generated by its global sections and $\deg\eS^{*}|_{q}=c_{1}(\wedge^{3}\eS^{*}|_{q})=2$.
Let $Q$ be the image of $\mP({\eS}|_{q})$ under $\varphi_{\eS}$.
Then there are two possibilities; (i) the degree of $\mP({\eS}|_{q})\to Q$
is two and $Q$ is a $3$-plane, i.e., a quadric of rank 1, or (ii)
the degree of $\mP({\eS}|_{q})\to Q$ is one and $Q$ is a quadric
of rank $4$ or $3$ depending on the splitting type of $\eS^{*}|_{q}$,
i.e., $\sO(1)_{\mP^{1}}^{\oplus2}\oplus\sO_{\mP^{1}}$ or $\sO_{\mP^{1}}(2)\oplus\sO_{\mP^{1}}^{\oplus2}$,
respectively (see Example \ref{ex:conics} below). The case (i) is
excluded by the condition $y=([Q],q)\in\hcoY\setminus G_{\hcoY}.$
Also if $Q$ is a $3$-plane $\mP(V_{4})$, then $q\subset\{[U]\in\mathrm{G}(3,V)\mid U\subset V_{4}\}$
and $\mP_{q}^{2}$ must be a $\sigma$-plane, i.e., $q$ is a $\sigma$-conic
by definition. 

Thus we see that every smooth $\tau$- or $\rho$-conic determines
a point $y=([Q],q)\in\hcoY\setminus G_{\hcoY}$, and vice versa. 

\end{proof}

We present some explicit examples of conics according to their ranks.
It should be noted that the Hilbert scheme of conics admits the natural
$SL(V)$-action, and hence the examples below describe the general
properties of each orbit of the Hilbert scheme under the $SL(V)$-action. 

\begin{expl} \label{ex:conics} 

(Conics of rank 3) Taking a basis ${\bf e}_{1},\cdots,{\bf e}_{5}$
of $V$, consider the subspaces, for example, to be $V_{1}=\langle{\bf e}_{4}\rangle$,
$V_{4}=\langle{\bf e}_{1},{\bf e}_{2},{\bf e}_{3},{\bf e}_{4}\rangle$
and $V_{2}=\langle{\bf e}_{4},{\bf e}_{5}\rangle$. Then typical examples
of $\tau$-, $\rho$-, $\sigma$-conics, respectively, may be given
explicitly in terms of the homogeneous coordinates of $\mathrm{G}(3,V)$;\[
q_{\tau}=\left\{ \left[\begin{array}{c}
s\;\; t\;\;0\;\;0\;\;0\\
0\;\;0\;\; s\;\; t\;\;0\\
0\;\;0\;\;0\;\;0\;\;1\end{array}\right]\right\} ,\; q_{\rho}=\left\{ \left[\begin{array}{c}
s^{2}\; st\;\, t^{2}\;0\;\;0\\
0\;\;0\;\;\,0\;\;1\;\;0\\
0\;\;0\;\;\,0\;\;0\;\;1\end{array}\right]\right\} ,\; q_{\sigma}=\left\{ \left[\begin{array}{c}
s\;\; t\;\;0\;\;0\;\;0\\
0\;\; s\;\; t\;\;0\;\;0\\
0\;\;0\;\;0\;\;1\;\;0\end{array}\right]\right\} ,\]
where $[s,t]\in\mP^{1}$ parameterizes each conic $q$. The $\tau$-conic
above is on a unique plane $\mP_{q_{\tau}}^{2}=\mP(\langle{\bf e}_{135},{\bf e}_{245},{\bf e}_{235}+{\bf e}_{145}\rangle)\subset\mP(\wedge^{3}V)$
and characterized by the conic $\mP_{q_{\tau}}^{2}\cap\mathrm{G}(3,V)=\{p_{135}p_{245}-p_{145}^{2}=0\}$,
where $p_{ijk}$ are the Pl\"ucker coordinates and we have introduced
a notation ${\bf e}_{ijk}:={\bf e}_{i}\wedge{\bf e}_{j}\wedge{\bf e}_{k}$.
The $\rho$-conic above is on a plane ${\rm P}_{V_{2}}=\mP(\langle{\bf e}_{145},{\bf e}_{245},{\bf e}_{345}\rangle)$
with its equation $p_{145}p_{345}-p_{245}^{2}=0$. Similarly, the
$\sigma$-conic above is on the plane ${\rm {\rm P}_{V_{1}V_{4}}\subset\mathrm{G}(3,V)}$.
It is easy to see ${\rm P}{}_{V_{1}V_{4}}=\mP(\langle{\bf e}_{124},{\bf e}_{134},{\bf e}_{234}\rangle)\subset\mP(\wedge^{3}V)$
and the equation $p_{124}p_{234}-p_{134}^{2}=0$ for $q_{\sigma}$.\hfill {\rm []} \end{expl}

$\,$

Let $q$ be a $\tau$-conic of rank $2$. Then $q$ is a pair of intersecting
lines, say, $l_{1}$ and $l_{2}$. We may write $l_{i}=\{[\Pi]\mid V_{2}^{(i)}\subset\Pi\subset V_{4}^{(i)}\}$,
where $\mP(V_{2}^{(i)})\subset\mP(V)$ are lines and $\mP(V_{4}^{(i)})\subset\mP(V)$
are $3$-planes for $i=1,2$. Since $l_{1}\cap l_{2}\not=\emptyset$,
it holds $\dim(V_{2}^{(1)}\cap V_{2}^{(2)})\geq1$. Since $q$ is
not a $\rho$-conic, $V_{2}^{(1)}\not=V_{2}^{(2)}$. Therefore we
see that $\dim(V_{2}^{(1)}\cap V_{2}^{(2)})=1$ and $l_{1}\cap l_{2}=[V_{2}^{(1)}+V_{2}^{(2)}]$.
If $V_{4}^{(1)}=V_{4}^{(2)}$, then $q$ is contained in ${\rm {P}_{V_{2}^{(1)}\cap V_{2}^{(2)},V_{4}^{(1)}}}$
and then $q$ must be a $\sigma$-conic. Therefore $V_{4}^{(1)}\not=V_{4}^{(2)}$.
In summary, $q=l_{1}\cup l_{2}$ is a $\tau$-conic of rank two iff
\[
V_{d}^{(1)}\not=V_{d}^{(2)}\,(d=2,4),\; V_{4}^{(1)}\cap V_{4}^{(2)}=V_{2}^{(1)}+V_{2}^{(2)}.\]
 The other two types of conics of rank two may be described in a similar
way. 

\begin{expl} \label{ex:conics2} 

(Conics of rank 2) We present examples of $\tau$-,$\rho$-, and $\sigma$-conics
of rank 2, respectively. (1) From the above description, an example
of $\tau$-conic of rank 2 may be given by $q_{\tau}=l_{1}\cup l_{2}$
with \[
l_{1}=\{[{\bf e}_{1},{\bf e}_{2},s{\bf e}_{3}+t{\bf e}_{4}\mid[s,t]\in\mP^{1}\},\;\; l_{2}=\{[{\bf e}_{2},{\bf e}_{3},s{\bf e}_{1}+t{\bf e}_{5}\mid[s,t]\in\mP^{1}\}.\]
 It is easy to see that this conic is on a unique plane $\mP_{q_{\tau}}^{2}=\mP(\langle{\bf e}_{123},{\bf e}_{124},{\bf e}_{235}\rangle)\subset\mP(\wedge^{3}V)$.
Then the equation of $q_{\tau}$ can be read as $\mP_{q_{\tau}}\cap\mathrm{G}(3,V)=\{p_{124}p_{235}=0\}$.

\noindent (2) Take $V_{1},V_{2},V_{4}$ as in Example~\ref{ex:conics}.
Then ${\rm P}_{V_{1}V_{4}}=\mP(\langle{\bf e}_{124},{\bf e}_{134},{\bf e}_{234}\rangle)\subset\mathrm{G}(3,V)$
and ${\rm P}_{V_{2}}=\mP(\langle{\bf e}_{145},{\bf e}_{245},{\bf e}_{345}\rangle)\subset\mathrm{G}(3,V)$.
As a $\rho$-conic of rank 2 on the $\rho$-plane ${\rm {P}_{V_{2}}}$,
we may have $q_{\rho}=l_{3}\cup l_{4}$ with \[
l_{3}=\{[s{\bf e}_{1}+t{\bf e}_{2},{\bf e}_{4},{\bf e}_{5}]\mid[s,t]\in\mP^{1}\},\;\; l_{4}=\{[s{\bf e}_{2}+t{\bf e}_{3},{\bf e}_{4},{\bf e}_{5}]\mid[s,t]\in\mP^{1}\}.\]
 The equation of this conic is $p_{145}p_{345}=0$. Similarly, as
an example of $\sigma$-conic of rank 2 on the $\sigma$-plane ${\rm {P}_{V_{1}V_{4}}}$,
we may have $q_{\sigma}=l_{5}\cup l_{6}$ with \[
l_{5}=\{[s{\bf e}_{1}+t{\bf e}_{2},{\bf e}_{3},{\bf e}_{4}\mid[s,t]\in\mP^{1}\},\;\; l_{6}=\{[{\bf e}_{1},s{\bf e}_{2}+t{\bf e}_{3},{\bf e}_{4}\mid[s,t]\in\mP^{1}\}.\]
 The equation of this conic is $p_{124}p_{234}=0$. \hfill [] \end{expl} 

$\;$

\subsection{$\hcoY_{3}$ and $\overline{\hcoY}$}

Consider a smooth conic $q=\left\{ [\xi(s,t)]\in\rG(3,V)\mid[s,t]\in\mP^{1}\right\} ,$
then the planes $\mP(\xi(s,t))\subset\mP(V)$ contain (at least) a
point $[V_{1}]\in\mP(V)$ in common (see Proposition \ref{prop:Y34}
and Example \ref{ex:conics}). Therefore we may assume the form $[\xi(s,t)]=[\xi^{(1)}(s,t),\xi^{(2)}(s,t),v]$
with $V_{1}=\mC v$. Corresponding plane $\mP_{q}^{2}$ in $\mP(\wedge^{3}V)$
is a point $[U]=[u_{1},u_{2},u_{3}]\in\mathrm{G}(3,\wedge^{3}V)$
determined by \[
\xi^{(1)}(s,t)\wedge\xi^{(2)}(s,t)\wedge v=u_{1}\, s^{2}+u_{2}\, st+u_{3}\, t^{2}\quad(u_{i}\in\wedge^{3}V).\]
Note that $u_{i}=\bar{u}_{i}\wedge v$ with some $\bar{u}_{i}\in\wedge^{2}V$,
and $\bar{u}_{i}$ may be considered in $\wedge^{2}(V/V_{1}).$ Then
the point $[U]\in\mathrm{G}(3,\wedge^{3}V)$ may be represented by
$[\bar{U}]=[\bar{u}_{1},\bar{u}_{2},\bar{u}_{3}]\in\mathrm{G}(3,\wedge^{2}(V/V_{1})$).
We write this by $[U]=[\bar{U}\wedge v]$ in short. We also write
by $\bar{\mP}_{q}^{2}$ the plane in $\mP(\wedge^{2}(V/V_{1}))$ which
corresponds to $[\bar{U}]$.

For each $v\in V$, we define a linear map $e_{v}:\wedge^{3}V\to\wedge^{4}V$
by $u\mapsto v\wedge u$. Consider the restriction $e_{v}\vert_{U}$
to $U\subset\wedge^{3}V$, and set $a_{U}=\left\{ v\in V\mid e_{v}\vert_{U}=0\right\} $.
$a_{U}$ is nothing but the annihilator of $U$.

\begin{defn} \label{def:Y3}

We define\[
\hcoY_{3}=\left\{ ([U],[V_{1}])\mid V_{1}\subset a_{U}\right\} \subset\mathrm{G}(3,\wedge^{3}V)\times\mathbb{P}(V),\]
and $\overline{\hcoY}$ to be the projection of $\hcoY_{3}$ to the
first factor of $\mathrm{G}(3,\wedge^{3}V)\times\mP(V)$. We consider
the reduced structure on $\overline{\hcoY}.$ \hfill {\rm []} \end{defn} 

Since $e_{v}\vert_{U}=0\,(V_{1}=\mathbb{C}v)$ implies that $U$ is
the $\mathbb{C}$-span of non-vanishing vectors of the form $\bar{u}_{i}\wedge v\,(i=1,2,3)$
with $\bar{u}_{i}\in\wedge^{2}(V/V_{1})$, the fiber of the natural
projection $\hcoY_{3}\to\mathbb{P}(V)$ over $[V_{1}]\in\mathbb{P}(V)$
can be identified with $\mathrm{G}(3,\wedge^{2}(V/V_{1}))$. Hence
we see that \[
\hcoY_{3}=\mathrm{G}(3,T(-1)^{\wedge2}),\]
and in particular $\hcoY_{3}$ is smooth. In what follows, both $([\bar{U}],[V_{1}])\in\hcoY_{3}$
with $[\bar{U}]\in\mathrm{G}(3,\wedge^{2}(V/V_{1}))$ and $([U],[V_{1}])$
will be used to represent a point on $\hcoY_{3}.$ 

From the arguments above, it is clear that all planes $\mP_{q}^{2}$
corresponding to smooth conics $q$ are contained in $\overline{\hcoY}$.
We also note that each point $([\bar{U}],[V_{1}])$ of $\hcoY_{3}$
determines a conic $\bar{q}$ in $\rG(2,V/V_{1})$ by the intersection
$\mP(\bar{U})\cap\rG(2,V/V_{1})$ in $\mP(\wedge^{2}(V/V_{1}))$ if
$\mP(\bar{U})\not\subset\rG(2,V/V_{1})$. Then the conic $\bar{q}=\left\{ [\bar{\xi}^{(1)}(s,t),\bar{\xi}^{(2)}(s,t)]\in\rG(2,V/V_{1})\mid[s,t]\in\mP^{1}\right\} $
determines a conic $q$ in $\rG(3,V)$ by $q=\left\{ [\xi^{(1)}(s,t),\xi^{(2)}(s,t),v]\in\rG(3,V)\mid[s,t]\in\mP^{1}\right\} $
with $V_{1}=\mC v$. 

\begin{prop}  {\rm (1)} The image $\overline{\hcoY}$ of $\hcoY_{3}$
under the projection is described by \[
\overline{\hcoY}=\left\{ [U]\in\rG(3,\wedge^{3}V)\mid\dim a_{U}\geq1\right\} .\]
{\rm (2)} $\overline{\hcoY}$ is singular along $\overline{\hcoY}_{sing}$
which is given by \[
\left\{ [U]\in\rG(3,\wedge^{3}V)\mid\dim a_{U}=2\right\} =\overline{\Prt}_{\rho}(\simeq\rG(2,V)).\]
{\rm (3)} The morphism $\Lrho_{\hcoY_{3}}:\hcoY_{3}\to\overline{\hcoY}$
is a small resolution with its fiber being $\mP^{1}$ over each point
of $\overline{\hcoY}_{sing}$ and is isomorphic over $\overline{\hcoY}\setminus\overline{\hcoY}_{sing}$. 

\end{prop}

\begin{proof} The claim (1) is immediate from the definition. 

(2) From the definition of the annihilator $a_{U}$, it is clear that
the condition $\dim a_{U}\geq1$ is satisfied only when $[U]\in\rG(3,\wedge^{3}V)$
takes the form \begin{equation}
[U]=[\bar{u}_{1}\wedge v,\,\bar{u}_{2}\wedge v,\,\bar{u}_{3}\wedge v]\label{eq:Uv}\end{equation}
with some non-zero vector $v\in V$ and $\bar{u}_{i}\in\wedge^{2}(V/V_{1})$
with $V_{1}=\mC v$. It is easy to see that the possible values for
$\dim a_{U}$ are 1 or 2, and the case $\dim a_{U}=2$ occurs iff
$[U]$ specializes further to \begin{equation}
[U]=[a\wedge v_{1}\wedge v_{2},\, b\wedge v_{1}\wedge v_{2},\, c\wedge v_{1}\wedge v_{2}].\label{eq:Uvv}\end{equation}

Now let us note the duality $\wedge^{3}V\simeq\wedge^{2}V^{*}$, which
follows from the non-degenerate pairing $\wedge^{3}V\times\wedge^{2}V\to\wedge^{5}V$.
Using this, we consider the following sequence of maps:\[
\ft{S}^{2}(\wedge^{3}V)\simeq\ft{S}^{2}(\wedge^{2}V^{*})\to(\wedge^{2}V^{*})\wedge(\wedge^{2}V^{*})\to\wedge^{4}V^{*}\simeq V.\]
We denote the composition of these maps by $\varphi$, and consider
its restriction $\varphi_{U}:=\varphi\vert_{\ft{S}^{2}U}$: $\ft{S}^{2}U\to V$.
Using the form (\ref{eq:Uv}), we take a basis of $\ft{S}^{2}U$ by
$\left\{ (\bar{u}_{i}\wedge v)\otimes_{s}(\bar{u}_{j}\wedge v)\right\} $,
where $\otimes_{s}$ means the symmetric tensor product. Then the
composite map above may be described by \[
(\bar{u}_{i}\wedge v)\otimes_{s}(\bar{u}_{j}\wedge v)\mapsto\bar{u}_{i}^{*}\otimes_{s}\bar{u}_{j}^{*}\mapsto\bar{u}_{i}^{*}\wedge\bar{u}_{j}^{*},\]
where $\bar{u}_{i}^{*}\in\wedge^{2}(V/V_{1})^{*}$ are determined
by the non-degenerated pairing $\wedge^{2}(V/V_{1})\times\wedge^{2}(V/V_{1})\to\wedge^{4}(V/V_{1})$
and are considered in $\wedge^{2}V^{*}$ via the natural inclusion
$\wedge^{2}(V/V_{1})^{*}$ $\subset\wedge^{2}V^{*}$. We note that
$\bar{u}_{i}^{*}\wedge\bar{u}_{j}^{*}$'s belong to the one dimensional
subspace $\wedge^{4}(V/V_{1})^{*}\subset\wedge^{4}V^{*}$. Hence if
$[U]$ takes the form (\ref{eq:Uv}), then $\rank\varphi_{U}\leq1$.
It is also easy to see the converse. Therefore the condition $\dim a_{U}\geq1$
is equivalent to the condition $\rank\varphi_{U}\leq1$ on the rank,
which we can express by the two-by-two minors of the matrix representing
$\varphi_{U}$ (see Remark below). $\overline{\hcoY}$ is defined
by taking the reduced structure of the condition $\rank\varphi_{U}\leq1$.
We will see in Lemma \ref{lem:ideal-bar-Y} below that the scheme
structure coming from the rank condition contains embedded points
along the singular locus of $\overline{\hcoY}$. Also we will show
explicitly in (\ref{eq:Using-parameter}) that the singularities of
the variety $\overline{\hcoY}$ appear at the locus consisting of
points $[U]$ of the form (\ref{eq:Uvv}) which determine the corresponding
$\rho$-planes ${\rm P}_{V_{2}}$ with $V_{2}=\langle v_{1},v_{2}\rangle$
and vice versa. Hence the singular locus coincides with the set of
$\rho$-planes $\overline{\Prt}_{\rho}\simeq\rG(2,V)$. 

(3) Clearly, the fiber of $\hcoY_{3}\to\overline{\hcoY}_{sing}$ over
each point is $\mP(V_{2})\simeq\mP^{1}$, and $\hcoY_{3}\to\overline{\hcoY}$
is bijective over $\overline{\hcoY}\setminus\overline{\hcoY}_{sing}$. 

\end{proof}

\begin{rem} The condition $\rank\varphi_{U}\leq1$ was suggested
to the authors by the referee. If we write a point $[U]$ by the Pl\"ucker
coordinates $p_{ijk}$ of $\wedge^{3}V$ as $[U]=[p_{i_{1}j_{1}k_{1}}^{(1)},p_{i_{2}j_{2}k_{2}}^{(2)},p_{i_{3}j_{3}k_{3}}^{(3)}]$,
then it is straightforward to obtain the $\dim V\times\dim\ft{S}^{2}U$
matrix which represents $\varphi_{U}$ in terms of these coordinates.
For reader's convenience, we write the the matrix entries explicitly:
\begin{equation}
\sC_{i(ab)}=\sum_{j,j',k,k'}\epsilon^{ijj'kk'}\big(\sum_{l,m,n}\epsilon^{jklmn}p_{lmn}^{(a)}\big)\big(\sum_{l',m',n'}\epsilon^{j'k'l'm'n'}p_{l'm'n'}^{(b)}\big),\label{eq:matrixCiab}\end{equation}
where $\epsilon^{ijklm}$ is the signature function defined by ${\bf e}_{i}^{*}\wedge{\bf e}_{j}^{*}\wedge{\bf e}_{k}^{*}\wedge{\bf e}_{l}^{*}\wedge{\bf e}_{m}^{*}=\epsilon^{ijklm}{\bf e}_{1}^{*}\wedge{\bf e}_{2}^{*}\wedge{\bf e}_{3}^{*}\wedge{\bf e}_{4}^{*}\wedge{\bf e}_{5}^{*}$
with the basis ${\bf e}_{i}^{*}(i=1,..,5)$ of $V^{*}$.\hfill {\rm []}\end{rem}

\begin{prop} \label{prop:barPrho-barPsigma} The set of $\rho$-planes
$\overline{\Prt}_{\rho}=\overline{\hcoY}_{sing}\simeq\rG(2,V)$ and
the set of $\sigma$-planes $\overline{\Prt}_{\sigma}\simeq F(1,4,V)$
in $\rG(3,\wedge^{3}V)$ $($see Subsection $\ref{sub:Conics-and-planes})$
are subvarieties in $\overline{\hcoY}$. They are given by\begin{eqnarray*}
\overline{\Prt}_{\rho} & = & \left\{ \big[\left(V/V_{2}\right)\wedge\left(\wedge^{2}V_{2}\right)\big]\mid[V_{2}]\in\rG(2,V)\right\} \text{ and}\\
\overline{\Prt}_{\sigma} & = & \left\{ \;\big[\wedge^{2}\left(V_{4}/V_{1}\right)\wedge V_{1}\big]\;\mid[V_{1}\subset V_{4}]\in F(1,4,V)\right\} .\end{eqnarray*}

\end{prop} \begin{proof} Claims follow directly from the definitions.
\end{proof}

$\;$

\subsection{The resolution $\widetilde{\hcoY}\to\overline{\hcoY}$ \label{subsection:BlowUp}}

We study the singularity of $\overline{\hcoY}$ along $\overline{\hcoY}_{sing}$,
and find another (small) resolution $\widetilde{\hcoY}$ of $\overline{\hcoY}$.
The two smooth varieties $\hcoY_{3}$ and $\widetilde{\hcoY}$ turns
out to be related by the (anti-)flip with an exceptional set in $\hcoY_{2}$
of $\mP^{1}\times\mP^{5}$ fibration over $\overline{\hcoY}_{sing}$
(see (\ref{eq:smallDiag})). For concreteness, we describe below the
resolution in the local coordinates near a point $[U_{0}]=[{\bf e}_{1}\wedge{\bf e}_{4}\wedge{\bf e}_{5},{\bf e}_{2}\wedge{\bf e}_{4}\wedge{\bf e}_{5},{\bf e}_{3}\wedge{\bf e}_{4}\wedge{\bf e}_{5}]$
in $\overline{\hcoY}_{sing}$, which can also be written by $[\left(\wedge^{2}V_{2}^{0}\right)\wedge\left(V/V_{2}^{0}\right)]$
with $[V_{2}^{0}]=[{\bf e}_{4},{\bf e}_{5}]\in\rG(2,V)$. Let us fix
a basis of $\wedge^{3}V$ by\[
{\bf e}_{123};\,{\bf e}_{234},{\bf e}_{314},{\bf e}_{124};\,{\bf e}_{235},{\bf e}_{315},{\bf e}_{125};\,{\bf e}_{145},{\bf e}_{245},{\bf e}_{345},\]
with ${\bf e}_{ijk}:={\bf e}_{i}\wedge{\bf e}_{j}\wedge{\bf e}_{k}$.
We introduce an affine coordinate of $\rG(3,\wedge^{3}V)$ centered
at $[U_{0}]$ by the $3\times10$ matrix;\[
[U]=\left[\begin{matrix}z_{1} & y_{11} & y_{12} & y_{13} & x_{11} & x_{12} & x_{13} & 1 & 0 & 0\\
z_{2} & y_{21} & y_{22} & y_{23} & x_{21} & x_{22} & x_{23} & 0 & 1 & 0\\
z_{3} & y_{31} & y_{32} & y_{33} & x_{31} & x_{32} & x_{33} & 0 & 0 & 1\end{matrix}\right],\]
where $z_{i}=y_{ij}=x_{ij}=0$ represents $[U_{0}]$. With this coordinate,
it is straightforward to evaluate (\ref{eq:matrixCiab}). We denote
by $\sI_{\varphi_{U}}$ the ideal generated by all $2\times2$ minors
of the matrix $(\sC_{i(ab)})$. Then, evaluating the primary decomposition
of $\sI_{\varphi_{U}}$ by Macaulay2 \cite{M2}, we obtain 

\begin{lem} \label{lem:ideal-bar-Y}The decomposition consists of
two components: $\sI_{\varphi_{U}}=\sI_{\overline{\hcoY}}\cap\sI_{\overline{\hcoY}}^{1}$
with $\sI_{\overline{\hcoY}}$ and $\sI_{\overline{\hcoY}}^{1}$ representing
the reduced part and embedded points, respectively. $\sI_{\overline{\hcoY}}$
is a radical ideal which is generated by \[
z_{1}+\big|\begin{smallmatrix}y_{11} & y_{23}\\
x_{11} & x_{23}\end{smallmatrix}\big\vert-\big\vert\begin{smallmatrix}y_{12} & y_{13}\\
x_{12} & x_{13}\end{smallmatrix}\big\vert,\; z_{2}-\big\vert\begin{smallmatrix}y_{22} & y_{13}\\
x_{22} & x_{13}\end{smallmatrix}\big\vert-\big\vert\begin{smallmatrix}y_{23} & y_{21}\\
x_{23} & x_{21}\end{smallmatrix}\big\vert,\; z_{3}+\big\vert\begin{smallmatrix}y_{33} & y_{12}\\
x_{33} & x_{12}\end{smallmatrix}\big\vert-\vert\begin{smallmatrix}y_{31} & y_{32}\\
x_{31} & x_{32}\end{smallmatrix}\vert\]
and the $2\times2$ minors of \begin{equation}
\left(\begin{matrix}y_{11} & y_{12}+y_{21} & y_{13}+y_{31} & y_{22} & y_{23}+y_{32} & y_{33}\\
x_{11} & x_{12}+x_{21} & x_{13}+x_{31} & x_{22} & x_{23}+x_{32} & x_{33}\end{matrix}\right).\label{eq:matrix-xy-22}\end{equation}
The radical of the component $\sI_{\overline{\hcoY}}^{1}$ is generated
by all the matrix entries of $(\ref{eq:matrix-xy-22})$, i.e., $y_{ii},x_{jj}$,
$y_{ij}+y_{ji},x_{ij}+x_{ji}$ $(1\leq i<j\leq3)$, and $z_{1}-\big|\begin{smallmatrix}y_{12} & y_{13}\\
x_{12} & x_{13}\end{smallmatrix}\big\vert,$ $z_{2}-\big\vert\begin{smallmatrix}y_{23} & y_{21}\\
x_{23} & x_{21}\end{smallmatrix}\big\vert$, $z_{3}-\big\vert\begin{smallmatrix}y_{31} & y_{32}\\
x_{31} & x_{32}\end{smallmatrix}\big\vert.$

\end{lem}

From the form of the ideal generated by the $2\times2$ minors of
(\ref{eq:matrix-xy-22}), we observe that the singular locus of the
variety $V(\sI_{\overline{\hcoY}})$ is exactly along the variety
$V(\sI_{\overline{\hcoY}}^{1})$ of the embedded points. Note that,
due to the natural $SL(V)$ action, these properties are valid for
all other affine coordinates although we worked in a specific affine
coordinate of $\rG(3,\wedge^{3}V)$. $\overline{\hcoY}$ is defined
by the ideal $\sI_{\overline{\hcoY}}$ in each affine coordinate. 

Now, let us note that among $21$ variables $(z_{i},y_{ij},x_{ij})\in\mC^{21}$,
the parameters for the singular locus $\overline{\hcoY}_{sing}=V(\sI_{\overline{\hcoY}}^{1})$
are easily identified in the following form, \begin{equation}
[U]=[(\wedge^{2}V_{2})\wedge(V/V_{2})]=\left[\begin{smallmatrix}a_{2}b_{3}-a_{3}b_{2} & \;0 & \;\; b_{3} & -b_{2} & \;\;0 & -a_{3} & \;\; a_{2} & \;\;1 & \;0 & \;0\\
a_{3}b_{1}-a_{1}b_{3} & -b_{3} & \;\;0 & \;\; b_{1} & \;\; a_{3} & \;\;0 & -a_{1} & \;\;0 & \;1 & \;0\\
a_{1}b_{2}-a_{2}b_{1} & \;\; b_{2} & -b_{1} & \;\;0 & -a_{2} & \;\; a_{1} & \;\;0 & \;\;0 & \;0 & \;1\end{smallmatrix}\right]\label{eq:Using-parameter}\end{equation}
with $[V_{2}]=[{\bf e}_{4}+a_{1}{\bf e}_{1}+a_{2}{\bf e}_{2}+a_{3}{\bf e}_{3},{\bf e}_{5}+b_{1}{\bf e}_{1}+b_{2}{\bf e}_{2}+b_{3}{\bf e}_{3}]$.
In this form, it is clear that $\overline{\hcoY}_{sing}\simeq\rG(2,V)$.
Now, as a normal coordinate to $\overline{\hcoY}_{sing}$ at $[U_{0}]$,
we introduce the following coordinates $z_{i}$ and $y_{ij},x_{ij}\,(i\leq j)$
by \begin{equation}
[U]=\left[\begin{array}{cccccccccc}
z_{1} & y_{11} & y_{12} & y_{13} & x_{11} & x_{12} & x_{13} & 1 & 0 & 0\\
z_{2} & 0 & y_{22} & y_{23} & 0 & x_{22} & x_{23} & 0 & 1 & 0\\
z_{3} & 0 & 0 & y_{33} & 0 & 0 & x_{33} & 0 & 0 & 1\end{array}\right].\label{eq:Uzyx}\end{equation}

\begin{prop} \label{lem:p1p5}

Using the normal coordinates to $\overline{\hcoY}_{sing}$ at $[U_{0}]$
above, $\overline{\hcoY}$ is described by $z_{1}=z_{2}=z_{3}=0$
and the zeros of all $2\times2$ minors of\begin{equation}
\left(\begin{array}{cccccc}
y_{11} & y_{12} & y_{13} & y_{22} & y_{23} & y_{33}\\
x_{11} & x_{12} & x_{13} & x_{22} & x_{23} & x_{33}\end{array}\right),\label{eq:22minors}\end{equation}
 which describes the affine cone of the Segre embedding $\mP^{1}\times\mP^{5}$.
In particular, the variety $\overline{\hcoY}$ is a normal variety.

\end{prop}

\begin{proof} The first claim is immediate from Lemma \ref{lem:ideal-bar-Y}.
Due to the $SL(V)$-action on $\overline{\hcoY}$, all the singularities
along $\overline{\hcoY}_{sing}$ are isomorphic. Hence the second
claim follows. \end{proof}

From the above proposition, we see that the singularity of $\overline{\hcoY}$
can be resolved by introducing the ratio of the two lows or the six
columns of (\ref{eq:22minors}). These ratios introduce $\mP^{1}$
or $\mP^{5}$ along $\overline{\hcoY}_{sing}$ as the exceptional
sets. We can see that the former gives the resolution $\hcoY_{3}\to\overline{\hcoY}$.
The latter resolution gives the (anti-)flip of $\hcoY_{3}.$

\begin{defn} \label{defn:Grho}We define $\widetilde{\hcoY}$ to
be the (small) resolution of $\overline{\hcoY}$ with the exceptional
set being a $\mP^{5}$-bundle over $\overline{\hcoY}_{sing}$. We
denote the exceptional set by $G_{\rho}$.$\hfill[]$

\end{defn}

By definition, the blow-up over the point $[U_{0}]$ is given by the
ratio (representing a point $\mP^{5}$)\begin{equation}
y_{11}:y_{12}:\cdots:y_{33}=x_{11}:x_{12}:\cdots:x_{33}=S_{11}:S_{12}:S_{13}:S_{22}:S_{23}:S_{33}\label{eq:BlowUpCo-xyS}\end{equation}
 for the transversal directions $(y_{ij},x_{ij})_{i\leq j}$ introduced
in (\ref{eq:Uzyx}) with $z_{1}=z_{2}=z_{3}=0.$ Using this, one of
the affine coordinate for the transversal directions may be written
$(y_{11},x_{11},\frac{S_{ij}}{S_{11}})_{i\leq j}$ and similarly for
others. 

If we use the both ratios, the exceptional set becomes $\mP^{1}\times\mP^{5}$
fibered over $\overline{\hcoY}_{sing}$ which is a divisor. We denote
this blow-up by $\hcoY_{2}$ and the exceptional divisor by $F_{\rho}$.
Then $\hcoY_{3}\leftarrow\hcoY_{2}\to\widetilde{\hcoY}$ is the standard
(anti-)flip. Now we summarize the resolutions in the following diagram
(see also Fig.4):

\begin{equation}
\xyYYYYZYH\label{eq:smallDiag}\end{equation}

$\;$

\subsection{Double spin decomposition and the construction of $\Lpi_{\widetilde{\hcoY}}:\widetilde{\hcoY}\to\Hes$}

Consider a point $[U]\in\overline{\hcoY}$. By definition, we can
express $[U]=[\bar{U}\wedge V_{1}]$ with some $V_{1}=\mC v$ and
$[\bar{U]}\in\rG(3,\wedge^{2}(V/V_{1}))$. We describe $\wedge^{3}\bar{U}$
in $\wedge^{3}(\wedge^{2}(V/V_{1}))$. 

Let us note the following irreducible decomposition as $sl(V/V_{1})$-modules
(see \cite[$\S 19.1$]{FH} for example): \begin{equation}
\wedge^{3}(\wedge^{2}(V/V_{1}))=\ft{S}^{2}(V/V_{1})\oplus\ft{S}^{2}(V/V_{1})^{*}.\label{eq:spin}\end{equation}
 We will call this {}``double spin'' decomposition since the components
in the r.h.s. are identified with $V_{2\lambda_{s}}$ and $V_{2\lambda_{\bar{s}}}$
as the $so(\wedge^{2}V/V_{1})(\simeq sl(V/V_{1}))$-modules, where
$\lambda_{s}$ and $\lambda_{\bar{s}}$ represent the spinor and conjugate
spinor weights, respectively (see {[}\textit{loc. cit.}{]}). The projection
to the second factor of (\ref{eq:spin}) defines the following rational
map \[
\bar{\pi}_{2}^{DS}:\mP(\wedge^{3}(\wedge^{2}(V/V_{1}))\dashrightarrow\mP(\ft{S}^{2}(V/V_{1})^{*}).\]
We identify the projective space $\mP(\ft{S}^{2}(V/V_{1})^{*})$ with
the space of quadrics in $\mP(V/V_{1})$. Also, by the natural inclusion
$(V/V_{1})^{*}\hookrightarrow V^{*}$, we consider the following composition
\[
\pi_{2}^{DS}:\mP(\wedge^{3}(\wedge^{2}(V/V_{1}))\dashrightarrow\mP(\ft{S}^{2}(V/V_{1}){}^{*})\to\mP(\ft{S}^{2}V^{*}),\]
and identify its image with the singular quadrics in $\mP(V)$ which
contain $V_{1}$ in the singular locus. We denote the restrictions
of these rational maps to $\rG(3,\wedge^{2}(V/V_{1}))\subset\mP(\wedge^{3}(\wedge^{2}(V/V_{1}))$
by \[
\bar{\varphi}_{V_{1}}=\bar{\pi}_{2}^{DS}\vert_{\rG(3,\wedge^{2}(V/V_{1}))},\;\;\;\;\varphi_{V_{1}}=\pi_{2}^{DS}\vert_{\rG(3,\wedge^{2}(V/V_{1}))}.\]
By definition, these rational maps have the same set of indeterminacy.
Note also that the indeterminacy occurs when the projection to the
second factor is zero in (\ref{eq:spin}).

\begin{lem} \label{lem:indeterminacy}

For a point $[U]\in\rG(3,\wedge^{3}V)$, the following statements
hold:

\noindent {\rm (1)}  If $a_{U}=V_{1}$$(\dim a_{U}=1)$, $[U]$ decomposes
uniquely to $[U]=[\bar{U}\wedge V_{1}]$ and determines a point $\varphi_{V_{1}}([\bar{U}])$. 

\noindent {\rm (2)} If $a_{U}=V_{2}$ $(\dim a_{U}=2)$, $[U]$ decomposes
into $[\bar{U}\wedge V_{1}]=[(V/V_{2})\wedge(V_{2}/V_{1})\wedge V_{1}]$
for each $V_{1}\subset V_{2}$. For any choice of $V_{1}$, $[\wedge^{3}\bar{U}]$
is an indeterminacy point of the rational map $\varphi_{V_{1}}$. 

\end{lem}

\begin{proof} Both the claims follow from the explicit form (\ref{eq:vw-plucker})
of the decomposition (\ref{eq:spin}), where the Pl\"ucker coordinates
of $\wedge^{3}\bar{U}$ are related to symmetric matrices $[v_{ij},w_{kl}]$
representing the decomposition. The indeterminacy of $\varphi_{V_{1}}$
occurs exactly when $w_{kl}=0$. By the property (I.5) in Appendix,
this implies $\rank v\leq1$. Now we may assume $v_{ij}=x_{i}x_{j}$
for some vector $x\in V/V_{1}$. By the action of $GL(V/V_{1})$,
we may further assume $x_{1}=1$ and $x_{i}=0\,(i=2,3,4).$ Then we
obtain the conditions $p_{IJK}=0$ except $p_{{\bf 124}}=\frac{1}{2}$
for the Pl\"ucker coordinates of $\wedge^{3}\bar{U}$, which determine
$[\bar{U}]$ to be $[\bar{{\bf e}}_{1}\wedge\bar{{\bf e}}_{2},\bar{{\bf e}}_{1}\wedge\bar{{\bf e}}_{3},\bar{{\bf e}}_{1}\wedge\bar{{\bf e}}_{4}]$.
This implies $\dim a_{U}=2$. Hence, if $\dim a_{U}=1$, then $\varphi_{V_{1}}([\bar{U}])$
is defined. 

For the claim (2), we represent $[\bar{U}]$ as $[(V/V_{2})\wedge(V_{2}/V_{1})]$
by fixing $V_{1}\subset V_{2}$ arbitrarily. Using the $GL(V/V_{1})$
action, we may assume the form $[\bar{{\bf e}}_{2}\wedge\bar{{\bf e}}_{1},\bar{{\bf e}}_{3}\wedge\bar{{\bf e}}_{1},\bar{{\bf e}}_{4}\wedge\bar{{\bf e}}_{1}]$
for $[\bar{U}]$. Then, from the relation (\ref{eq:vw-plucker}),
we obtain $w_{kl}=0$ and see that $[\wedge^{3}\bar{U}]$ is an indeterminacy
point. 

\end{proof}

From the above lemma, and $\dim a_{U}=1$ or $2$ for $[U]\in\overline{\hcoY}$,
we have the following proposition (and definition):

\begin{prop} 

There is a rational map $\varphi_{DS}:\overline{\hcoY}\dashrightarrow\Hes\subset\mP(\ft{S}^{2}V^{*})$
defined by $\varphi_{DS}([U])=\varphi_{V_{1}}([\bar{U}])$ through
a decomposition $[U]=[\bar{U}\wedge V_{1}]$. 

\end{prop}

\begin{prop} \label{prop:phiDS} The rational map $\varphi_{DS}:\overline{\hcoY}\dashrightarrow\Hes$
extends to a morphism $\widetilde{\varphi}_{DS}:\widetilde{\hcoY}\to\Hes$.

\end{prop}

\begin{proof} We show that the indeterminacy of the rational map
$\pi_{2}^{DS}$ is resolved by the blow-up $\widetilde{\hcoY}\to\overline{\hcoY}.$
Since the indeterminacy occurs at the points $[U]$ with $\dim a_{U}=2$
(Lemma \ref{lem:indeterminacy}), we only need to analyze $\varphi_{DS}$
near a point $[U]\in\overline{\hcoY}_{sing}$. Note that such a point
can be written as $[U]=[(V/V_{2})\wedge\left(\wedge^{2}V_{2}\right)${]}
with $V_{2}=a_{U}$. By using the $GL(V)$ action, we can further
assume the form $[U_{0}]$ analyzed in Subsection \ref{subsection:BlowUp}.
Let us write the affine coordinate (\ref{eq:Uzyx}) in terms of the
ratio (\ref{eq:BlowUpCo-xyS}),\[
[U]=\left[\begin{array}{c}
0\\
0\\
0\end{array}\; y_{11}\left(\begin{array}{ccc}
1 & t_{12} & t_{13}\\
0 & t_{22} & t_{23}\\
0 & 0 & t_{33}\end{array}\right)\; x_{11}\left(\begin{array}{ccc}
1 & t_{12} & t_{13}\\
0 & t_{22} & t_{23}\\
0 & 0 & t_{33}\end{array}\right)\begin{array}{ccc}
1 & 0 & 0\\
0 & 1 & 0\\
0 & 0 & 1\end{array}\right],\]
where the coordinate $(y_{ij},x_{ij})_{i\leq j}$ for the normal direction
to $\overline{\hcoY}_{sing}$ are written by an affine coordinate
$(y_{11},x_{11},t_{ij}=\frac{S_{ij}}{S_{11}})$ of the blow-up. Let
us first consider the case $x_{11}\not=0.$ In this case, we find
a unique decomposition $[U]=[\bar{U}\wedge V_{1}]$ by $V_{1}=\mC v$
with $v=y_{11}{\bf e}_{4}+x_{11}{\bf e}_{5}$. Writing by $\bar{{\bf e}}_{i}$
the image of ${\bf e}_{i}$ in the quotient $V/V_{1}$, we introduce
the basis of $\wedge^{2}V/V_{1}$ by $\bar{{\bf e}}_{23},\bar{{\bf e}}_{31},\bar{{\bf e}}_{12},\bar{{\bf e}}_{14},\bar{{\bf e}}_{24},\bar{{\bf e}}_{34}$
with $\bar{{\bf e}}_{ij}:=\bar{{\bf e}}_{i}\wedge\bar{{\bf e}}_{j}$.
Then, writing $[\bar{U}]\in\rG(3,\wedge^{2}(V/V_{1}))$ explicitly
with this basis, it is straightforward to evaluate $\bar{\varphi}_{V_{1}}([\bar{U}])$
as \[
\bar{\varphi}_{V_{1}}([\bar{U}])=\left[x_{11}\left(\begin{array}{cccc}
2 & t_{12} & t_{13} & x_{11}t_{23}\\
t_{12} & 2t_{22} & t_{23} & x_{11}D\\
t_{13} & t_{23} & 2t_{33} & x_{11}t_{12}t_{33}\\
x_{11}t_{23} & x_{11}D & x_{11}t_{12}t_{33} & 2x_{11}^{2}t_{22}t_{33}\end{array}\right)\right],\]
where we set $D=t_{12}t_{23}-t_{13}t_{22}$. Note that the $4\times4$
matrix above is written with respect to the basis $\bar{{\bf e}}_{i}(i=1,2,3,4)$
of the quotient $V/V_{1}$. For the case $y_{11}\not=0$, we obtain
the same expression but $x_{11}$ being replaced by $-y_{11}$ with
the basis $\bar{{\bf e}}_{i}(i=1,2,3,5)$. These two cases are unified
when we write the quadric $\varphi_{V_{1}}([\bar{U}])\in\Hes\subset\mP(\ft{S}^{2}V^{*})$,
by dividing the expressions by $x_{11}$ and $-y_{11}$, respectively,
as \[
\varphi_{V_{1}}([\bar{U}])=\left[\left(\begin{matrix}2 & t_{12} & t_{13} & x_{11}t_{23} & -y_{11}t_{23}\\
t_{12} & 2t_{22} & t_{23} & x_{11}D & -y_{11}D\\
t_{13} & t_{23} & 2t_{33} & x_{11}t_{12}t_{33} & -y_{11}t_{12}t_{33}\\
x_{11}t_{23} & x_{11}D & x_{11}t_{12}t_{33} & 2x_{11}^{2}t_{22}t_{33} & -2x_{11}y_{11}t_{22}t_{33}\\
-y_{11}t_{23} & -y_{11}D & -y_{11}t_{12}t_{33} & -2x_{11}y_{11}t_{22}t_{33} & 2y_{11}^{2}t_{22}t_{33}\end{matrix}\right)\right].\]

The calculations are parallel for other affine coordinates of the
blow-up, and we see that the rational map $\varphi_{DS}$ is extended
to a morphism $\tilde{\varphi}_{DS}$ so that the point $[S_{ij}]$
of the exceptional set $\mP^{5}$ is mapped to the following $3\times3$
matrix \begin{equation}
\tilde{\varphi}_{DS}([S_{ij}])=\left[\begin{array}{ccc}
2S_{11} & S_{12} & S_{13}\\
S_{12} & 2S_{22} & S_{23}\\
S_{13} & S_{23} & 2S_{33}\end{array}\right]\in\mP(\ft{S}^{2}(V/V_{2})^{*}).\label{eq:indeterminacySij}\end{equation}

\end{proof}

$\;$

\subsection{The morphism $\widetilde{\hcoY}\to\hcoY$\label{sub:tildeY-Y}}

Here we prove the following proposition toward the end of this subsection:

\begin{prop} \label{prop:factorVarphiDS}There exists a morphism
$\Lrho_{\widetilde{\hcoY}}:\widetilde{\hcoY}\to\hcoY$ which factors
the morphism $\tilde{\varphi}_{DS}:\widetilde{\hcoY}\to\Hes$ as $\tilde{\varphi}_{DS}=\Lrho_{\hcoY}\circ\Lrho_{\widetilde{\hcoY}}$. 

\end{prop}

We recall that the morphism $\tilde{\varphi}_{DS}$ is defined by
extending the rational map $\varphi_{DS}:\overline{\hcoY}\dashrightarrow\Hes$
through the blow-up $\widetilde{\hcoY}\to\overline{\hcoY}$. Since
the value of $\tilde{\varphi}_{DS}$ over the exceptional set is described
by (\ref{eq:indeterminacySij}), studying the morphism $\varphi_{DS}$
over $\overline{\hcoY}\setminus\overline{\hcoY}_{sing}$ suffices
to describe $\tilde{\varphi}_{DS}$. In particular, we have $\tilde{\varphi}_{DS}^{-1}([Q])=\overline{\varphi_{DS}^{-1}([Q])}$
if $\varphi_{DS}^{-1}([Q])\not=\emptyset$, where the closure is taken
in $\widetilde{\hcoY}$. 

By our definition of the rational map $\varphi_{DS}$, we note the
following relation \begin{equation}
\varphi_{DS}^{-1}=\bigcup_{V_{1}\subset V}\mbox{\ensuremath{\varphi}}_{V_{1}}^{-1},\label{eq:inversePhiV1}\end{equation}
where, although it is implicit, the inverse image $\varphi_{V_{1}}^{-1}$
should be considered with the wedge product $V_{1}\wedge:$$\rG(3,\wedge^{3}(V/V_{1}))\hookrightarrow\rG(3,\wedge^{3}V)$.
Also we note that $\varphi_{V_{1}}^{-1}$ above can be replaced by
$\bar{\varphi}_{V_{1}}^{-1}$ due to the canonical embedding $\mP(\ft{S}^{2}(V/V_{1})^{*})\hookrightarrow\mP(\ft{S}^{2}V^{*}).$ 

\textcolor{black}{We will study the fiber over each point $[Q]\in\Hes$
according to the rank of quadric $Q$. In the arguments below, we
denote by $\Ker Q$ the cone of the singular locus of $Q$, which
is a linear subspace in $V$ of dimension $5-\rank Q$. }

\begin{lem} \label{lem:inverse-rk4and3} {\rm (1)} For quadrics
$Q$ of rank $4$, $\tilde{\varphi}_{DS}^{-1}([Q])$ consists of two
points. 

\noindent {\rm (2)} If $\rank Q=3$, then $\tilde{\varphi}_{DS}^{-1}([Q])$
is one point in the fiber of $\widetilde{\hcoY}\to\overline{\hcoY}$
over $[U]=[(V/V_{2})\wedge(\wedge^{2}V_{2})]$ with $V_{2}=\Ker Q$. 

\end{lem}

\begin{proof}(1) We use (\ref{eq:inversePhiV1}) to determine the
inverse image. If $\rank Q$=4, then $\varphi_{V_{1}}^{-1}([Q])$
is non-empty only for $V_{1}=\Ker Q$. It is useful to see $\bar{\varphi}_{V_{1}}^{-1}([w])$
instead of $\varphi_{V_{1}}^{-1}([Q])$ by expressing the quadric
$[Q]\in\mP(\ft{S}^{2}V^{*})$ with the corresponding $\mbox{4\ensuremath{\times}4}$
matrix $[w]\in\mP(\ft{S}^{2}(V/V_{1})^{*})$. The inverse image $\bar{\varphi}_{V_{1}}^{-1}([w])$
can then be determined from the Pl\"ucker relation (\ref{eq:Ivw})
for $\rG(3,\wedge^{2}(V/V_{1}))$ in terms of the {}``double spin
coordinates'' $[v,w]$. It is immediate from (I.2) in Appendix (and
$\rank w$=4) to see that $\bar{\varphi}_{V_{1}}^{-1}([w])$ consists
of two points $[v,w]$ satisfying $v.w=\pm\sqrt{\det w}\id_{4}$.
Then we have the claim for $\tilde{\varphi}_{DS}^{-1}([Q])=\overline{\varphi_{DS}^{-1}([Q])}$. 

(2) As above, we consider $[w]$ which corresponds to a quadric $[Q]$
by choosing $V_{1}\subset V_{2}$ ($V_{2}=\Ker Q$). Also, we may
assume $V_{2}=\langle{\bf e}_{4},{\bf e}_{5}\rangle$ using suitable
$GL(V)$ actions. From (I.3) in Appendix, and the assumption $\rank w=3$,
we see that $[w]$ cannot be the image of $\bar{\varphi}_{V_{1}}$
for any choice of $V_{1}$. However, it is clear that there exists
exactly one point in the exceptional set $\mP^{5}$ which corresponds
to the rank three matrix $w$ (see (\ref{eq:indeterminacySij}) and
the equation above it). \end{proof}

\vspace{0.1cm}

\[
\xyFigCoYs\]

\vspace{0.2cm}
 \begin{fcaption} 

\item \textbf{Fig.2. The fibers of the morphism $\tilde{\varphi}_{DS}:\widetilde{\hcoY}\to\Hes$.}
$\Hes^{i}$ represent the loci of symmetric matrices of rank $i$.
Note that $\Hes^{k}$ may be identified with $G_{\hcoY}^{k}$ for
$k=1,2$. Each fiber is interpreted in terms of the geometry of conics
in Sections \ref{sub:Conics-and-planes} and \ref{sub:Smooth-conics-and}.
\end{fcaption}

$\;$

Let us now describe $\tilde{\varphi}_{DS}^{-1}([Q])$ for $\rank Q=2$.
Set $V_{3}=\Ker Q$. Then the quadric considered in $\mP(V/V_{3})$
defines two points $[V_{4}^{(i)}/V_{3}](i=1,2)$. Note that $V_{4}^{(1)}\cap V_{4}^{(2)}=V_{3}$
since the two points are distinct. With these data, we define a subset
$\Gamma(V_{4}^{(1)},V_{4}^{(2)},V_{3})$ in $\overline{\hcoY}$ by\[
\Gamma(V_{4}^{(1)},V_{4}^{(2)},V_{3})=\left\{ [l_{V_{2}^{(1)}V_{4}^{(1)}}\cup l_{V_{2}^{(2)}V_{4}^{(2)}}]\,\vert\, V_{2}^{(1)},V_{2}^{(2)}\subset V_{3},V_{2}^{(1)}\not=V_{2}^{(2)}\right\} ,\]
where $l_{V_{2}V_{4}}$ is a line in $\rG(3,V)$ defined by $l_{V_{2}V_{4}}=\left\{ [\wedge^{3}\mC^{3}]\vert V_{2}\subset\mC^{3}\subset V_{4}\right\} $(cf.
the lines $l_{i}$ described in Subsection \ref{sub:Smooth-conics-and})
and $[l\cup l']$ represents the projective plane containing $l\cup l'$
with $l\not=l'$ and $l\cap l'\not\not=\emptyset$. By the condition
$V_{2}^{(1)}\not=V_{2}^{(2)}$, the intersection $V_{2}^{(1)}\cap V_{2}^{(2)}=V_{1}$
determines a point in $V_{3}$. Then we can decompose the set as $\Gamma(V_{4}^{(1)},V_{4}^{(2)},V_{3})=\underset{V_{1}\subset V_{3}}{\cup}\Gamma_{V_{1}}(V_{4}^{(1)},V_{4}^{(2)},V_{3})$
with the obvious notation. 

\begin{lem} \label{lem:inverse-rk2}With the above definitions for
quadrics of rank $2$, we have \[
\tilde{\varphi}_{DS}^{-1}([Q])=\overline{\Gamma(V_{4}^{(1)},V_{4}^{(2)},V_{3})}\simeq\mP(V_{3}^{*})\times\mP(V_{3}^{*}).\]

\end{lem}

\begin{proof}Since the subspaces $V_{2}^{(i)}\subset V_{3}$ are
described by points in $\mP(V_{3}^{*})$ for each, it is clear that
the closure in $\widetilde{\hcoY}$ of $\Gamma(V_{4}^{(1)},V_{4}^{(2)},V_{3})$
is $\mP(V_{3}^{*})\times\mP(V_{3}^{*})$. Similarly we see that the
closure of $\Gamma_{V_{1}}(V_{4}^{(1)},V_{4}^{(2)},V_{3})$ is $\mP((V_{3}/V_{1})^{*})\times\mP((V_{3}/V_{1})^{*})\simeq\mP^{1}\times\mP^{1}$.
We now show $\varphi_{V_{1}}^{-1}([Q])=\Gamma_{V_{1}}(V_{4}^{(1)},V_{4}^{(2)},V_{3})$.
The inclusion $\varphi_{V_{1}}^{-1}([Q])\supset\Gamma_{V_{1}}(V_{4}^{(1)},V_{4}^{(2)},V_{3})$
can be verified by evaluating $\varphi_{V_{1}}$ explicitly, for example,
with $V_{1}=\langle{\bf e}_{1}\rangle,V_{2}^{(i)}=\langle{\bf e}_{1},{\bf e}_{i+1}\rangle$
and $V_{4}^{(i)}=\langle{\bf e}_{1},{\bf e}_{2},{\bf e}_{3},{\bf e}_{3+i}\rangle$.
To show the equality, we consider as before the matrix $[w]$ which
represents the quadric $Q$ in $\mP(V/V_{1})$ and describe $\bar{\varphi}_{V_{1}}^{-1}([w])$
using the Pl\"ucker relations (\ref{eq:Ivw}) in terms of $[v,w]$.
Changing the coordinate of $V/V_{1}$ suitably, we may assume that
$[w]$ is given in the form $w_{0}=\left(\begin{smallmatrix}\begin{smallmatrix}0 & 1\\
1 & 0\end{smallmatrix} & O_{2}\\
O_{2} & O_{2}\end{smallmatrix}\right)$ with $O_{2}$ being the $2\times2$ zero matrix. Then by the properties
(I.4) and (I.2), we obtain $v=\left(\begin{smallmatrix}O_{2} & O_{2}\\
O_{2} & \begin{smallmatrix}v_{11} & v_{12}\\
v_{12} & v_{22}\end{smallmatrix}\end{smallmatrix}\right)$ . Now substituting $[v,w]=[v,tw_{0}](t\not=0)$ into the equation
in the first line of (\ref{eq:Ivw}), we have \[
v_{11}v_{22}-v_{12}^{2}+t^{2}=0\;\;(t\not=0).\]
From this, we obtain $\overline{\bar{\varphi}_{V_{1}}^{-1}([w])}\simeq\mP^{1}\times\mP^{1}$
for the closure. Since both sides of $\varphi_{V_{1}}^{-1}([Q])\supset\Gamma_{V_{1}}(V_{4}^{(1)},V_{4}^{(2)},V_{3})$
have the same closure in $\widetilde{\hcoY}$, they must coincide.
Now the claim follows since $\varphi_{DS}^{-1}([Q])=\underset{V_{1}\subset V_{3}}{\cup}\varphi_{V_{1}}^{-1}([Q])$
and $\tilde{\varphi}_{DS}^{-1}([Q])=\overline{\varphi_{DS}^{-1}([Q])}$.\end{proof}

Quadrics of rank 1 are determined by specifying $V_{4}=\Ker Q$ or
the corresponding points in $\mP(V^{*})$. We now describe $\tilde{\varphi}_{DS}^{-1}(V_{4}):=\tilde{\varphi}_{DS}^{-1}([Q])$.
This inverse image is determined by careful analysis of the loci representing
$\rho$-planes and $\sigma$-planes described in Proposition \ref{prop:barPrho-barPsigma}.

Let us consider a triple $V_{1}\subset V_{2}\subset V_{4},$ and one
dimensional subspaces $(V/V_{4})\wedge(\wedge^{2}V_{2})$ and $\wedge^{2}(V_{4}/V_{2})\wedge V_{1}$
in the quotient space $\wedge^{3}V$ mod $V_{4}\wedge(\wedge^{2}V_{2})$.
We introduce an affine two plane in the quotient space by\[
\mC^{2}(V_{1},V_{2},V_{4})=(V/V_{4})\wedge(\wedge^{2}V_{2})\oplus\wedge^{2}(V_{4}/V_{2})\wedge V_{1}.\]
 Using this, we define the following subset in $\overline{\hcoY}$:\begin{align*}
\Gamma(V_{4})= & \bigg\{[(V_{4}/V_{2})\wedge(\wedge^{2}V_{2}),\xi]\bigg|\begin{array}{l}
\xi\in\mC^{2}(V_{1},V_{2},V_{4})\\
\mC\xi\not=(V/V_{4})\wedge(\wedge^{2}V_{2})\end{array},V_{1}\subset V_{2}\subset V_{4}\bigg\},\end{align*}
where it should be noted that $[(V_{4}/V_{2})\wedge(\wedge^{2}V_{2}),\xi]$
defines a (projective) plane in $\mP(\wedge^{3}V$) and also a point
in $\overline{\hcoY}_{sing}$ if $\mC\xi=(V/V_{4})\wedge(\wedge^{2}V_{2})$.
We decompose this set into $\Gamma(V_{4})=\underset{V_{1}\subset V_{4}}{\cup}\Gamma(V_{1},V_{4})$
with fixing $V_{1}\subset V_{4}$, and similarly $\Gamma(V_{4})=\underset{V_{2}\subset V_{4}}{\cup}\Gamma(V_{2},V_{4})$
with fixing $V_{2}\subset V_{4}.$

\begin{lem}\label{lem:inverse-rk1} $(1)\;$$\tilde{\varphi}_{DS}^{-1}([V_{4}])=\overline{\Gamma(V_{4})}$. 

\noindent $(2)$ There is a birational morphism $\mP(\sO_{\rG(2,V_{4})}\oplus\sU_{\rG(2,V_{4})}^{*}(1))\to\overline{\Gamma(V_{4})}$
which contracts a divisor $E_{\sigma}=\mP(\sU_{\rG(2,V_{4})}^{*}(1))$
to $\mP(V_{4})$, where $\sU_{\rG(2,V_{4})}$ is the universal subbundle
of the Grassmann $\rG(2,V_{4})$.\end{lem}

\begin{proof}

(1) As in the preceding lemma, it suffices to show the equality $\bar{\varphi}_{V_{1}}^{-1}([w])=\Gamma(V_{1},V_{4})$.
The inclusion $\varphi_{V_{1}}^{-1}([V_{4}])\supset\Gamma(V_{1},V_{4})$
is easily verified by taking $V_{4}=\langle{\bf e}_{1},{\bf e}_{2},{\bf e}_{3},{\bf e}_{4}\rangle$
and $V_{1}=\langle{\bf e}_{1}\rangle$, for example. To show the equality,
we represent the quadric $Q$ by the corresponding matrix $[w]\in\mP(\ft{S}^{2}(V/V_{1})^{*})$
and determine $\bar{\varphi}_{V_{1}}^{-1}([w])$. By assumption, we
have $\rank w=1$ and hence $[w]$ can be written as $[w_{kl}]=[a_{k}a_{l}]$
with some non-zero vector $a\in\mP((V/V_{1})^{*})$. Then from (I.5)
in Appendix, we see that $\rank v\leq1$. Writing $v_{ij}=x_{i}x_{j}$
with $x\in V/V_{1}$ and also solving (\ref{eq:Ivw}) we obtain\[
\bar{\varphi}_{V_{1}}^{-1}([w])=\left\{ [x_{i}x_{j},ta_{k}a_{l}]\mid a.x=0,t\not=0\right\} .\]
 From this, we see that the closure of $\bar{\varphi}_{V_{1}}^{-1}([w])=\varphi_{V_{1}}^{-1}([V_{4}])$
in $\widetilde{\hcoY}$ is the cone over $v_{2}(\mP^{2})\simeq\mP^{2}$
from the vertex $[0,a_{k}a_{l}]\in\mP(\ft{S}^{2}(V/V_{1})\oplus\ft{S}^{2}(V/V_{1})^{*})$,
which is isomorphic to $\mP(1^{3},2)$. 

The same cone arises from the closure $\Gamma(V_{1},V_{4}).$ To see
this, let us write a point $[U]=[(V_{4}/V_{2})\wedge(\wedge^{2}V_{2}),\xi]$
of $\Gamma(V_{1},V_{4})$ with $\xi\in\mC^{2}(V_{1},V_{2},V_{4})$
and $[V_{2}]\in F(V_{1},2,V_{4})$$\simeq\mP(V_{4}/V_{1})$. Here
we observe that, when $\mC\xi=\wedge^{2}(V_{4}/V_{2})\wedge V_{1}$
in $\mC^{2}(V_{1},V_{2},V_{4})$, the corresponding point $[U]=[(V_{4}/V_{2})\wedge(\wedge^{2}V_{2}),\xi]$
is reduced to $[U_{v}]:=[\wedge^{2}(V_{4}/V_{1})\wedge V_{1}]$ which
is constant for any choice of $[V_{2}]$. Otherwise $[U]$ varies
when $[V_{2}]$ moves in $\mP(V_{4}/V_{1})$. This defines a cone
over $\mP^{2}\simeq\mP(V_{4}/V_{1})$ in $\overline{\Gamma(V_{1},V_{4})}$
from the vertex $[U_{v}]$. Combined with the inclusion $\varphi_{V_{1}}^{-1}([V_{4}])\supset\Gamma(V_{1},V_{4})$,
we see that the equality $\bar{\varphi}_{V_{1}}^{-1}([w])=\Gamma(V_{1},V_{4})$
must hold.

(2) We note that the elements of $\Gamma(V_{2},V_{4})$ are parametrized
by the lines $\mC\xi$ contained in the union $\underset{V_{1}\subset V_{2}}{\cup}\mC^{2}(V_{1},V_{2},V_{4})$
which simplifies to $(V/V_{4})\wedge(\wedge^{2}V_{2})\oplus\wedge^{2}(V_{4}/V_{2})\wedge V_{2}$.
Namely, the set $\Gamma(V_{2},V_{4})$ is parametrized by the projective
plane $\mP((V/V_{4})\wedge(\wedge^{2}V_{2})\oplus\wedge^{2}(V_{4}/V_{2})\wedge V_{2})$
. Now moving $[V_{2}]$ in the Grassmann $\rG(2,V_{4})$, we obtain
the following projective bundle which parametrizes $\Gamma(V_{4}$):
\[
\mP(\wedge^{2}\sU_{\rG(2,V_{4})}\oplus(\wedge^{2}\sU_{\rG(2,V_{4})}^{*})\otimes\sU_{\rG(2,V_{2})})\simeq\mP(\sO_{\rG(2,V_{2})}\oplus\sU_{\rG(2,V_{2})}^{*}(1)),\]
where we use $\wedge^{2}\sU_{\rG(2,V_{4})}^{*}\simeq\sO_{\rG(2,V_{4})}(1)$
and $\sU_{\rG(2,V_{4})}(1)\simeq\sU_{\rG(2,V_{4})}^{*}$ for the universal
bundle $\sU_{\rG(2,V_{4})}$. Since $\sU_{\rG(2,V_{2})}^{*}\vert_{[V_{2}]}\simeq\wedge^{2}(V_{4}/V_{2})\wedge V_{2}=\left\{ \wedge^{2}(V_{4}/V_{2})\wedge V_{1}\vert V_{1}\subset V_{2}\right\} $
and $[(V_{4}/V_{2})\wedge(\wedge^{2}V_{2}),\xi]=[\wedge^{2}(V_{4}/V_{1})\wedge V_{1}]$
holds for all $\mC\xi\subset\sU_{\rG(2,V_{2})}^{*}\vert_{[V_{2}]}$,
we see that the divisor $\mP(\sU_{\rG(2,V_{2})}^{*}(1))$ collapses
to $\left\{ [\wedge^{2}(V_{4}/V_{1})\wedge V_{1}]\vert V_{1}\subset V_{4}\right\} \simeq\mP(V_{4})$
in $\overline{\Gamma(V_{4})}$ as claimed. \end{proof}

$\;$

\begin{rem} $\tilde{\varphi}_{DS}^{-1}([V_{4}])$ is the weighted
Grassmann ${\rm {w}\mathrm{G}(2,5)}$ as in \cite[Example 2.5]{CR},
which is defined in $\mP(1^{6},2^{4})$ by the equations \[
\left(\begin{smallmatrix}0 & x_{ij}\\
-x_{ij} & 0\end{smallmatrix}\right)\left(\begin{smallmatrix}y_{1}\\
:\\
y_{4}\end{smallmatrix}\right)=\left(\begin{smallmatrix}0\\
:\\
0\end{smallmatrix}\right),\;\; x_{12}x_{34}-x_{13}x_{24}+x_{14}x_{23}=0,\]
 where $[x_{ij},y_{k}]=[x_{12},x_{13},x_{14},x_{23},x_{24},x_{34},y_{1},..,y_{4}]$
is the (weighted) homogeneous coordinate of $\mP(1^{6},2^{4})$. This
${\rm {w}\mathrm{G}(2,5)}$ has singularities along $\{x_{ij}=0\}\simeq\mP^{3}$
of type $\frac{1}{2}(1,1,1)$. Over the complement of this singular
locus, we see ${\rm {w}\mathrm{G}(2,5)}$ has a $\mC^{2}$-fibration
over $\mathrm{G}(2,4)$. In Fig.3, we have depicted schematically
the fiber $\tilde{\varphi}_{DS}^{-1}([V_{4}])=\Lrho_{\widetilde{\hcoY}}^{-1}([V_{4}])$
over $[V_{4}]\in G_{\hcoY}^{1}$. \hfill []\end{rem}

\vspace{0.2cm}

\[
\xyFigwG\]

\vspace{0.2cm}
 \begin{fcaption} 

\item \textbf{Fig.3. The fiber of $\tilde{\varphi}_{DS}$ over $[V_{4}]$.}
Fibers are interpreted in terms of the geometry of conics in Sections
\ref{sub:Conics-and-planes} and \ref{sub:Smooth-conics-and}. The
contraction of $E_{\sigma}$ may be understood that the $\sigma$-conics
(double lines) are reduced $\sigma$-planes on which they lie. \end{fcaption}

$\;$

\noindent\textit{\textcolor{black}{Proof of Proposition \ref{prop:factorVarphiDS}}}\textbf{.}
It suffices to show that each fiber of $\tilde{\varphi}_{DS}$ over
quadrics of rank $\leq3$ consists only one connected component, and
two components over quadrics of tank 4 (see Proposition \ref{cla:double}).
We have seen these properties in the preceding three lemmas. \hfill $\square$

$\;$

\subsection{More properties of $\Lrho_{\widetilde{\hcoY}}:\widetilde{\hcoY}\to\hcoY$}

We show that the contraction ${\Lrho}_{\widetilde{\hcoY}}\colon\widetilde{\hcoY}\to{\hcoY}$
is a $K_{\widetilde{\hcoY}}$-negative extremal divisorial contraction. 

\begin{prop} \label{cla:FY} The exceptional locus of ${\Lrho}_{\widetilde{\hcoY}}$
is an $\mathrm{SL}(V)$-invariant prime divisor, which will be denoted
by $F_{\widetilde{\hcoY}}$. The image of $F_{\widetilde{\hcoY}}$
under $\Lrho_{\widetilde{\hcoY}}:\widetilde{\hcoY}\to\hcoY$ is $G_{\hcoY}$.
\end{prop}

\begin{proof} By construction, ${\Lrho}_{\widetilde{\hcoY}}$ is
$\mathrm{SL}(V)$-equivariant, and also $\widetilde{\hcoY}$ is smooth
and $\rho(\widetilde{\hcoY})=2$. By Proposition \ref{cla:ZY}, $\hcoY$
is a $\mQ$-factorial Gorenstein Fano variety with Picard number one.
Therefore ${\Lrho}_{\widetilde{\hcoY}}$ is neither a composite of
non-trivial morphisms nor a small contraction. Hence we have the first
assertion.

By Lemmas \ref{lem:inverse-rk4and3}, \ref{lem:inverse-rk2} and \ref{lem:inverse-rk1}
(see also Fig.2), it is immediate to see that $\Lrho_{\widetilde{\hcoY}}(F_{\widetilde{\hcoY}})=G_{\hcoY}$.
\end{proof}

\begin{prop} \label{cla:F} $K_{\widetilde{\hcoY}}=\Lrho_{\widetilde{\hcoY}}^{\;*}K_{\hcoY}+2F_{\widetilde{\hcoY}}$.
In particular, $\Lrho_{\widetilde{\hcoY}}$ is a $K_{\widetilde{\hcoY}}$-negative
divisorial extremal contraction and $\hcoY$ has only terminal singularities
with $\Sing\hcoY=G_{\hcoY}$. 

\end{prop} 

\begin{proof} $F_{\widetilde{\hcoY}}$ is generically a $\mP^{2}\times\mP^{2}$-fibration
over $G_{\hcoY}$(see Lemma \ref{lem:inverse-rk2}). Let $r$ be a
line in a ruling of the generic fiber $\Gamma\simeq\mP^{2}\times\mP^{2}$
of $F_{\widetilde{\hcoY}}$. It is enough to show $K_{\widetilde{\hcoY}}\cdot r=-2$.
Indeed, assume this, then by the adjunction formula\[
K_{\Gamma}=K_{F_{\widetilde{\hcoY}}}\vert_{\Gamma}=(K_{\widetilde{\hcoY}}+F_{\widetilde{\hcoY}})\vert_{\Gamma},\]
it holds $(K_{\widetilde{\hcoY}}+F_{\widetilde{\hcoY}})\cdot r=-3$,
and hence $F_{\widetilde{\hcoY}}\cdot r=-1$. Set $K_{\widetilde{\hcoY}}=\Lrho_{\widetilde{\hcoY}}^{\;*}K_{\hcoY}+aF_{\widetilde{\hcoY}}$
with unknown $a$. Then it holds that $K_{\widetilde{\hcoY}}\cdot r=aF_{\widetilde{\hcoY}}\cdot r$
and we obtain $a=2$ as claimed. 

Now we show that $K_{\widetilde{\hcoY}}\cdot r=-2$. Let us choose
$r$ so that it intersects with the diagonal of $\Gamma$. Let $r'$
be the strict transform on ${\hcoY_{2}}$ of $r$ and $r''$ the image
of $r'$ on $\hcoY_{3}$. Since $\hcoY_{2}\to\widetilde{\hcoY}$ is
the blow-up along $G_{\rho}$, we have \[
K_{\hcoY_{2}}\cdot r'=K_{\widetilde{\hcoY}}\cdot r+1.\]
 Moreover, since $K_{\hcoY_{2}}=\Lrho_{\hcoY_{2}}^{\;*}K_{\hcoY_{3}}+5F_{\rho}$
(cf. (\ref{eq:adj})), we have \[
K_{\hcoY_{2}}\cdot r'=K_{\hcoY_{3}}\cdot r''+5.\]
 Therefore it suffices to show $K_{\hcoY_{3}}\cdot r''=-6$. The strict
transform of $\Gamma$ on $\hcoY_{2}$ has a natural $\mP^{1}\times\mP^{1}$-fibration
over $\mP^{2}$ since it is the blow-up of $\Gamma$ along the diagonal.
Its fiber is described in the proof of Lemma \ref{lem:inverse-rk2}
and then $r''$ is a line in a fiber of $\hcoY_{3}\to\mP(V)$. Therefore
we have $K_{\hcoY_{3}}\cdot r''=-6$.

Finally, $\hcoY$ has only terminal singularity along $G_{\hcoY}$
since the minimal discrepancy for a smooth point of $\hcoY$ is equal
to $\dim\hcoY-1$. \end{proof}

\vspace{0.5cm}
 \newpage{}

\section{Sheaves $\tilde{\sS}_{L}^{*}$, $\tilde{\sQ}$, $\tilde{\sT}$ on
$\widetilde{\hcoY}$ and their properties \label{sec:SheavesDef}}

Here we introduce locally free sheaves $\tilde{\sS}_{L}^{*}$, $\tilde{\sQ}$,
$\tilde{\sT}$ on $\widetilde{\hcoY}$, which will play central roles
in our construction of the Lefschetz collection. For this, it will
be helpful to have the following picture of the birational geometry
of $\widetilde{\hcoY}$ (Fig.4).

\[
\FigYsDisplay\]

\begin{fcaption} 

\item \textbf{Fig.4. Birational geometries of $\hcoY$.} $F_{\widetilde{\hcoY}}$
in $\widetilde{\hcoY}$ represents the prime divisor parameterizing
reducible conics on $\mathrm{G}(3,V)$. Note that the diagram from
$\mP(V)$ to the endpoint $\hcoY$ defines a so-called Sarkisov link.
\end{fcaption}

$\:$

\textit{We will use the following convention without mentioning at
each time}: \begin{myitem} 

\item[$L_{\Sigma}$:] the pull back on a variety $\Sigma$ of $\sO(1)$
if there is a morphism $\Sigma\to\mP(V)$. 

\item[$M_{\Sigma}$:] the pull back on a variety $\Sigma$ of $\sO_{\Hes}(1)$
if there is a morphism $\Sigma\to\Hes$. \end{myitem}

\vspace{0.5cm}

\subsection{Locally free sheaves $\widetilde{\sS}_{L}^{*}$, $\widetilde{\sQ}$
on $\widetilde{\hcoY}$}

Consider the following universal sequence of the Grassmann bundle
$\hcoY_{3}=$ $\rG(3,\wedge^{2}T(-1))$ over $\mP(V)$ (cf. \cite[p.434]{Ful}):
\begin{equation}
0\to\sS\to\Lpi_{\hcoY_{3}}^{\;*}(T(-1)^{\wedge2})\to\sQ\to0,\label{eq:univ}\end{equation}
where $\sS$ is the relative universal subbundle of rank three and
$\sQ$ is the relative universal quotient bundle of rank three. Similarly,
we denote by $\bar{\sS}$ the universal subbundle of rank three of
the Grassmann $\rG(2,\wedge^{3}V)$. Then 

\begin{prop} $\sS^{*}(L_{\hcoY_{3}})=\Lrho_{\hcoY_{3}}^{\;*}\bar{\sS}^{*}$.
\end{prop}

\begin{proof} We may write a point of $\hcoY_{3}$ by $y=([\bar{U}],[V_{1}])$
(see the description right after Definition \ref{def:Y3}), which
is mapped to $[U]=[\bar{U}\wedge V_{1}]\in\overline{\hcoY}$. Note
that $V_{1}=\Lpi_{\hcoY_{3}}^{\;*}\sO_{\mP(V)}(-1)\vert_{y}=-L_{\hcoY_{3}}\vert_{y}$.
Therefore $\sS(-L_{\hcoY_{3}})=\Lrho_{\hcoY_{3}}^{\;*}\bar{\sS}$
holds for the universal subbundles. 

\end{proof}

Now we have the following proposition (and definition):

\begin{prop}\label{def:SQT1} There exist locally free sheaves $\widetilde{\sS}_{L}^{*}$
and $\widetilde{\sQ}$ on $\widetilde{\hcoY}$ which satisfy \[
\Lrho_{\hcoY_{2}}^{\;*}\sS^{*}(L_{\hcoY_{2}})=\tLrho_{\hcoY_{2}}^{\;*}\widetilde{\sS}_{L}^{*}\text{ and }\Lrho_{\hcoY_{2}}^{\;*}\sQ=\tLrho_{\hcoY_{2}}^{\;*}\widetilde{\sQ}.\]

\end{prop}

\begin{proof} We define $\widetilde{\sS}_{L}^{*}$ to be the pullback
of $\bar{\sS}^{*}$ to $\widetilde{\hcoY}$, then the first claim
is immediate by the commutativity of the morphisms in Fig.4. To see
the existence of $\widetilde{\sQ}$, consider the universal sequence
(\ref{eq:univ}) on $\hcoY_{3}$. Let $[V_{1},V_{2}]$ be a point
on the exceptional locus $\Prt_{\rho}=F(1,2,V)\to\rG(2,V)$ of the
blow-up $\hcoY_{3}\to\overline{\hcoY}$. Since $\sS\vert_{[V_{1},V_{2}]}=(V/V_{2})\wedge(V_{2}/V_{1})$
(see Proposition \ref{prop:barPrho-barPsigma}), we have $0\to(V/V_{2})\wedge(V_{2}/V_{1})\to\wedge^{2}(V/V_{1})\to\sQ\vert_{[V_{1},V_{2}]}\to0$.
Hence we have $\sQ\vert_{[V_{1},V_{2}]}\simeq\wedge^{2}(V/V_{2})$,
which implies $\sQ\vert_{\gamma}\simeq\sO^{\oplus3}$ for a fiber
$\gamma=\mP^{1}$ of $F(1,2,V)\to\rG(2,V)$. The last property ensures
the existence of a locally free sheaf $\bar{\sQ}$ on $\overline{\hcoY}$
such that $\sQ=\Lrho_{\hcoY_{3}}^{\;*}\bar{\sQ}$. From the commutativity
of the diagram in Fig.4, the pull-back $\widetilde{\sQ}$ of $\bar{\sQ}$
to $\widetilde{\hcoY}$ has the claimed property. \end{proof}

$\;$

\subsection{Locally free sheaf $\widetilde{\sT}$ on $\widetilde{\hcoY}$}

Let us focus on the local geometry of the blow-up $\hcoY_{3}\to\overline{\hcoY}$
which is described by $\Prt_{\rho}=F(1,2,V)\to\rG(2,V)$. We denote
the universal sub-bundles of the partial flag variety $F(1,2,V)$
by $\sR_{1}\subset\sR_{2}\subset\sR_{V}$, where we set $\sR_{V}:=V\otimes\sO_{F(1,2,V)}$
and $\rank\sR_{k}=k$. There is an exact sequence \begin{equation}
0\to\sR_{2}/\sR_{1}\to\sR_{V}/\sR_{1}\to\sR_{V}/\sR_{2}\to0.\label{eq:exact-sequence-Rv}\end{equation}
It is sometimes useful to identify $F(1,2,V)$ with the projective
bundle $\mP(T(-1))$ over $\mP(V)$. Under this identification, the
exact sequence above is nothing but the relative Euler sequence of
the projective bundle. For example, we have $\sR_{2}/\sR_{1}=\sO_{\mP(T(-1))}(-1)$
and $\sR_{V}/\sR_{1}=\pi_{F}^{*}T(-1)$ with $\pi_{F}:\mP(T(-1))\to\mP(V)$.

\begin{prop}

Let $\Lpi_{\hcoY_{2}}=\Lpi_{\hcoY_{3}}\circ\Lrho_{\hcoY_{2}}$ be
the composition of $\Lrho_{\hcoY_{2}}:\hcoY_{2}\to\hcoY_{3}$ and
$\Lpi_{\hcoY_{3}}:\hcoY_{3}\to\mP(V)$. Also denote by $\Lrho_{\hcoY_{2}}\vert_{F_{\rho}}$
the restriction of the morphism $\Lrho_{\hcoY_{2}}$ to the exceptional
divisor $F_{\rho}\subset\hcoY_{2}$. Then

\noindent {\rm (1)} There is a surjective morphism $\Lpi_{\hcoY_{2}}^{\;*}\Omega(1)\to(\Lrho_{\hcoY_{2}}\vert_{F_{\rho}})^{*}\big(\sR_{2}/\sR_{1}\big)^{*}$.

\noindent {\rm (2)} The kernel of the surjective morphism \[
\sT_{2}^{*}=\Ker\left\{ \Lpi_{\hcoY_{2}}^{\;*}\Omega(1)\to(\Lrho_{\hcoY_{2}}\vert_{F_{\rho}})^{*}\big(\sR_{2}/\sR_{1}\big)^{*}\right\} \]
\;\;\;\; is a locally free sheaf on $\hcoY_{2}$.

\end{prop}

\begin{proof}From the exact sequence (\ref{eq:exact-sequence-Rv}),
we have a surjection $(\sR_{V}/\sR_{1})^{*}\to(\sR_{2}/\sR_{1})^{*}\to0$,
and hence $(\Lrho_{\hcoY_{2}}\vert_{F_{\rho}})^{*}(\sR_{V}/\sR_{1})^{*}\overset{q}{\to}(\Lrho_{\hcoY_{2}}\vert_{F_{\rho}})^{*}(\sR_{2}/\sR_{1})^{*}\to0$.
Here we note that $(\Lrho_{\hcoY_{2}}\vert_{F_{\rho}})^{*}(\sR_{V}/\sR_{1})^{*}=i^{*}\circ\Lpi_{\hcoY_{2}}^{\;*}\Omega(1)$
with $i:F_{\rho}\hookrightarrow\hcoY_{2}$. We now obtain the exact
sequence \begin{equation}
0\to\sT_{2}^{*}\to\Lpi_{\hcoY_{2}}^{\;*}\Omega(1)\overset{\,_{q\circ i^{*}}}{\longrightarrow}(\Lrho_{\hcoY_{2}}\vert_{F_{\rho}})^{*}(\sR_{2}/\sR_{1})^{*}\to0,\label{eq:exact-seq-T2}\end{equation}
with the surjection claimed in (1) and $\sT_{2}^{*}$ defined in (2).
By taking $\sE xt(-,\sO_{\hcoY_{2}})$ of this sequence, we see that
$\sT_{2}^{*}$ is a locally free sheaf on $\hcoY_{2}$ (see \cite[III, Ex 6.6]{Ha}).
\end{proof}

\begin{lem}\label{lem:T2-sum} $\sT_{2}^{*}\vert_{\delta}=\sO_{\delta}^{\oplus4}$
for each fiber $\delta\simeq\mP^{1}$ of $F_{\rho}\to G_{\rho}$.

\end{lem}

\begin{proof}Each fiber $\delta$ of $F_{\rho}\to G_{\rho}$ projects
isomorphically to a fiber of $F(1,2,V)\to\rG(2,V)$, and further to
a line $\mP^{1}$ in $\mP(V)$. Therefore $\Lpi_{\hcoY_{2}}^{\;*}\Omega(1)\vert_{\delta}\simeq\sO_{\mP^{1}}^{\oplus3}\oplus\sO_{\mP^{1}}(-1)$
and also $(\Lrho_{\hcoY_{2}}\vert_{F_{\rho}})^{*}(\sR_{2}/\sR_{1})^{*}\vert_{\delta}\simeq\sO_{\mP^{1}}(-1)$.
By restricting (\ref{eq:exact-seq-T2}), we obtain \[
\sT_{2}^{*}\vert_{\delta}\to\sO_{\mP^{1}}^{\oplus3}\oplus\sO_{\mP^{1}}(-1)\to\sO_{\mP^{1}}(-1)\to0,\]
which shows that there is a surjection $\sT_{2}^{*}\vert_{\delta}\to\sO_{\mP^{1}}^{\oplus3}$
with its kernel being an invertible sheaf $\sL$. Note that $\det\sT_{2}^{*}\simeq\sO_{\hcoY_{2}}(-L_{\hcoY_{2}}-F_{\rho})$
from (\ref{eq:exact-seq-T2}) and $(\sR_{2}/\sR_{1})^{*}=\sO_{\mP(T(-1))}(1)$.
Now, since $L_{\hcoY_{2}}\vert_{\delta}=\sO_{\delta}(1)$ by definition
and also $F_{\rho}\cdot\delta=-1$, we have $\det\,\sT_{2}^{*}\vert_{\delta}\simeq\sO_{\delta}$
and $\sL\simeq\sO_{\delta}$. Hence $\sT_{2}^{*}\vert_{\delta}\simeq\sO_{\delta}^{\oplus4}$.
\end{proof}

Define $\sT_{2}:=(\sT_{2}^{*})^{*}$. From the above lemma, we have
the following proposition (and definition):

\begin{prop} \label{def:tidelT} There exists a locally free sheaf
$\widetilde{\sT}$ on $\widetilde{\hcoY}$ such that \[
\sT_{2}=\tLrho_{\hcoY_{2}}^{\;*}\widetilde{\sT}.\]

\end{prop}

The following exact sequence will be used in our later calculations:

\begin{prop} \begin{equation}
0\to\Lpi_{\hcoY_{2}}^{\;*}T(-1)\to\sT_{2}\to(\Lrho_{\hcoY_{2}}|_{F_{\rho}})^{*}\big(\sR_{2}/\sR_{1}\big)({F_{\rho}}|_{F_{\rho}})\to0.\label{eq:Eb}\end{equation}
 \end{prop}

\begin{proof} By definition of $\sT_{2}^{*}$, we have $0\to\sT_{2}^{*}\to\Lpi_{\hcoY_{2}}^{\;*}\Omega(1)\to(\Lrho_{\hcoY_{2}}\vert_{F_{\rho}})^{*}\big(\sR_{2}/\sR_{1}\big)^{*}\to0.$
Now, by taking $\mathcal{H}om(-,\sO_{\hcoY_{2}})$, we obtain: \[
0\to\Lpi_{\hcoY_{2}}^{\;*}T(-1)\to\sT_{2}\to\sE xt_{\sO_{\hcoY_{2}}}^{1}((\Lrho_{\hcoY_{2}}|_{F_{\rho}})^{*}\big(\sR_{2}/\sR_{1}\big)^{*},\sO_{\hcoY_{2}})\to0.\]
 The claim follows from the isomorphism: \[
\sE xt_{\sO_{\hcoY_{2}}}^{1}\big((\Lrho_{\hcoY_{2}}|_{F_{\rho}})^{*}\big(\sR_{2}/\sR_{1}\big)^{*},\sO_{\hcoY_{2}}\big)\simeq(\Lrho_{\hcoY_{2}}|_{F_{\rho}})^{*}\big(\sR_{2}/\sR_{1}\big)(F_{\rho}|_{F_{\rho}}),\]
which we derive by the spectral sequence \[
\begin{aligned}\sE xt_{\sO_{\hcoY_{2}}}^{i}\big((\Lrho_{\hcoY_{2}}|_{F_{\rho}})^{*}\big(\sR_{2}/\sR_{1}\big)^{*}, & \,\sE xt_{\sO_{\hcoY_{2}}}^{j}(\sO_{F_{\rho}},\sO_{\hcoY_{2}})\big)\\
 & \Rightarrow\sE xt_{\sO_{\hcoY_{2}}}^{i+j}((\Lrho_{\hcoY_{2}}|_{F_{\rho}})^{*}\big(\sR_{2}/\sR_{1}\big)^{*},\,\sO_{\hcoY_{2}}),\end{aligned}
\]
 and also $\sE xt_{\sO_{\hcoY_{2}}}^{1}(\sO_{F_{\rho}},\sO_{\hcoY_{2}})\simeq\omega_{F_{\rho}}\otimes\omega_{\hcoY_{2}}^{-1}\simeq\sO_{F_{\rho}}({F_{\rho}}|_{F_{\rho}})$,
$\sE xt_{\sO_{\hcoY_{2}}}^{j}(\sO_{F_{\rho}},\sO_{\hcoY_{2}})=0$
if $j\not=1$. \end{proof}

\vspace{0.5cm}

\subsection{Properties of $\sS^{*},\sQ$ restricted on $\Prt_{\rho}$ and $\Prt_{\sigma}$}

As in the last subsection, we identify $\Prt_{\rho}=F(1,2,V)$ with
the projective bundle $\mP(T(-1))$ with $\pi_{F}:\mP(T(-1))\to\mP(V)$.
We introduce two divisors on $\Prt_{\rho}$;\[
H_{\Prt_{\rho}}=\sO_{\mP(T(-1))}(1)\text{\;\;{and}}\;\; L_{\Prt_{\rho}}:=\pi_{F}^{*}\sO(1).\]

\begin{prop} \label{cla:mislead} $\sQ|_{\Prt_{\rho}}\simeq\sS^{*}(L_{\hcoY_{3}})|_{\Prt_{\rho}}$
and $\sQ|_{\Prt_{\sigma}}\simeq\sS^{*}(L_{\hcoY_{3}})|_{\Prt_{\sigma}}$. 

\end{prop}

\begin{proof} 

Take a point $[V_{1},V_{2}]\in\Prt_{\rho}$, then the universal sequence
(\ref{eq:univ}) restricts to \[
0\to(V/V_{2})\wedge(V_{2}/V_{1})\to\wedge^{2}(V/V_{1})\to\wedge^{2}(V/V_{2})\to0,\]
where $\sS\vert_{[V_{1},V_{2}]}=(V/V_{2})\wedge(V_{2}/V_{1})$ and
$\sQ\vert_{[V_{1},V_{2}]}=\wedge^{2}(V/V_{2})$. By the wedge product,
we have a natural map $\sS\vert_{[V_{1},V_{2}]}\times\sQ\vert_{[V_{1},V_{2}]}\to\wedge^{4}(V/V_{1})$.
In fact, it is easy to see that this defines a non-degenerate form.
Hence, noting $\wedge^{4}(V/V_{1})=L_{\hcoY_{3}}\vert_{[V_{1},V_{4}]}$,
we obtain the claimed isomorphism $\sQ|_{\Prt_{\rho}}\simeq\sS^{*}(L_{\hcoY_{3}})|_{\Prt_{\rho}}$. 

The second isomorphism follows from a similar argument starting with
the restrictions $\sS\vert_{[V_{1},V_{4}]}=\wedge^{2}(V_{4}/V_{1})$
(see Proposition \ref{prop:barPrho-barPsigma}) and $\sQ\vert_{[V_{1},V_{4}]}=(V/V_{4})\wedge(V_{4}/V_{1})$.\end{proof}

\begin{prop} \label{cla:det} $\;$ 

\noindent {\rm (1)} $\sQ|_{\Prt_{\rho}}=\wedge^{2}(\sR_{V}/\sR_{2})\simeq(\sR_{V}/\sR_{2})^{*}(H_{\Prt_{\rho}}+L_{\Prt_{\rho}})$. 

\noindent {\rm (2)} $\det\sQ|_{\Prt_{\rho}}=2(H_{\Prt_{\rho}}+L_{\Prt_{\rho}})$. 

\end{prop}

\begin{proof}(1) Since $(\sR_{V}/\sR_{2})\vert_{[V_{1},V_{2}]}=V/V_{2}$
and $\sQ\vert_{[V_{1},V_{2}]}=\wedge^{2}(V/V_{2})$, we have $\sQ\vert_{\Prt_{\rho}}=\wedge^{2}(\sR_{V}/\sR_{2})$.
By taking the determinant of (\ref{eq:exact-sequence-Rv}), we have
\[
\wedge^{3}(\sR_{V}/\sR_{2})\simeq\wedge^{4}(\sR_{V}/\sR_{1})\otimes(\sR_{2}/\sR_{1})^{*}\simeq L_{\Prt_{\rho}}\otimes\sO_{\mP(T(-1))}(1),\]
where we use the relations $\sR_{V}/\sR_{1}=\pi_{F}^{*}T(-1)$ and
$\sR_{2}/\sR_{1}=\sO_{\mP(T(-1))}(-1)$ with $\pi_{F}:\mP(T(-1))\to\mP(V)$.
Noting the non-degenerate pairing $\sR_{V}/\sR_{2}\times\wedge^{2}(\sR_{V}/\sR_{2})\to\wedge^{3}(\sR_{V}/\sR_{2})$,
we obtain the isomorphism $\sQ\vert_{\Prt_{\rho}}=\wedge^{2}(\sR_{V}/\sR_{2})\simeq(\sR_{V}/\sR_{2})^{*}(H_{\Prt_{\rho}}+L_{\Prt_{\rho}}).$
The claim (2) follows from $\det\sQ\vert_{\Prt_{}}=\wedge^{3}(\wedge^{2}(\sR_{V}/\sR_{2}))\simeq(\wedge^{3}(\sR_{V}/\sR_{2}))^{\otimes2}$.
\end{proof}

\vspace{0.5cm}

\subsection{Divisors on $\hcoY_{3}$ and $\hcoY_{2}$ }

Recall the universal sequence of the Grassmann bundle $\hcoY_{3}=\rG(3,T(-1)^{\wedge2})$;

\[
0\to\sS\to\Lpi_{\hcoY_{3}}^{\;*}(T(-1)^{\wedge2})\to\sQ\to0.\]
Taking the determinant and using $\wedge^{6}(T(-1)^{\wedge^{2}})=(\wedge^{4}T(-1))^{\otimes3}=\sO(3)$,
we have \begin{equation}
\det\sQ=\det\sS^{*}+3L_{\hcoY_{3}}=\det\{\sS^{*}(L_{\hcoY_{3}})\}.\label{eq:detdiff}\end{equation}
 Also, since $T_{\hcoY_{3}/\mP(V)}=\sS^{*}\otimes\sQ$ (see \cite[p.435]{Ful}),
we have \begin{equation}
K_{\hcoY_{3}}=-\det(\sQ\otimes\sS^{*})+5L_{\hcoY_{3}}=-3(\det\sQ+\det\sS^{*})-5L_{\hcoY_{3}}=-6\det\sQ+4L_{\hcoY_{3}},\label{eq:KY3}\end{equation}
 where we note $\rank\sS=\rank\sQ=3$ and we use (\ref{eq:detdiff})
in the last equality.

We now investigate several relations among basic divisors on $\hcoY_{2}$.

\vspace{0.2cm}

Note that \begin{equation}
K_{\hcoY_{2}}=\Lrho_{\hcoY_{2}}^{\;*}K_{\hcoY_{3}}+5F_{\rho}\label{eq:adj}\end{equation}
 since $\Lrho_{\hcoY_{2}}$ is the blow-up along a smooth subvariety
of codimension $6$. By (\ref{eq:KY3}), we have \begin{equation}
K_{\hcoY_{2}}=-6\Lrho_{\hcoY_{2}}^{\;*}\det\sQ+4L_{\hcoY_{2}}+5F_{\rho}.\label{eq:canY2}\end{equation}

\begin{prop} \label{cla:M} The pull-back $M_{\hcoY_{2}}$ of $\sO_{\Hes}(1)$
is given by \[
M_{\hcoY_{2}}=\Lrho_{\hcoY_{2}}^{\;*}(\det\sQ)-L_{\hcoY_{2}}-F_{\rho}.\]
 \end{prop}

\begin{proof} Recall the definition of $\tilde{\varphi}_{DS}$ (Proposition
\ref{prop:phiDS}) and the second relation of (\ref{eq:vw-general-fromula}).
We note that the Pl\"ucker coordinates of $G(3,\wedge^{2}T(-1))$
are sections of $\det\,\sS^{*}$ and also $\wedge^{4}T(-1)=\sO(1)$
in (\ref{eq:vw-general-fromula}) (cf. Lemma \ref{cla:dec} below).
$\tilde{\varphi}_{DS}$ is defined by removing the zero in $w_{kl}$
along the indeterminacy set $\Prt_{\rho}$ of $\varphi_{V_{1}}$,
which is blown-up to the exceptional divisor $F_{\rho}$ under $\Lrho_{\hcoY_{2}}$.
Hence, using commutativity of the diagram in Fig.4, we have $\tLrho_{\hcoY_{2}}^{\;*}\circ\tilde{\varphi}_{DS}^{*}\sO_{\Hes}(1)=\Lrho_{\hcoY_{2}}^{\;*}(\det\,\sS^{*}+2L_{\hcoY_{3}})-F_{\rho}.$
The claim follows from (\ref{eq:detdiff}) and the definition $M_{\hcoY_{2}}:=\tLrho_{\hcoY_{2}}^{\;*}\circ\tilde{\varphi}_{DS}^{*}\sO_{\Hes}(1).$\end{proof}

\begin{lem} \label{cla:dec} \begin{equation}
\wedge^{3}(T(-1)^{\wedge2})=\bigl(\ft{S}^{2}(T(-1))\otimes\sO(1)\bigr)\oplus\bigl(\ft{S}^{2}(T(-1))^{*}\otimes\sO(2)\bigr).\qquad\quad\label{eq:relspin}\end{equation}
 \end{lem}

\begin{proof} Globalizing the decomposition (\ref{eq:spin}), we
have $\wedge^{3}(T(-1)^{\wedge2})=\ft{S}^{2}(T(-1))\otimes\sO(a)\oplus\ft{S}^{2}(T(-1))^{*}\otimes\sO(b)$
with some integers $a$ and $b$. To determine $a$ and $b$, we restricts
this equality to a line $l\subset\mP(V)$. Noting $T(-1)|_{l}\simeq\sO_{l}^{\oplus3}\oplus\sO_{l}(1)$,
we see that $a=1$ and $b=2$. \end{proof}

\vspace{3cm}

\section{Birational models of $F_{\widetilde{\hcoY}}$ and flattening of $F_{\widetilde{\hcoY}}\to G_{\hcoY}$}

\label{subsubsection:Global}

In this section, we study the birational geometry of the divisor $F_{\widetilde{\hcoY}}\subset\widetilde{\hcoY}$
in details, and obtain its explicit description $F_{\widetilde{\hcoY}}=\widehat{F}/\mZ_{2}$
in Proposition~\ref{cla:F'}. In particular, we obtain a natural
flattening of the fibration $F_{\widetilde{\hcoY}}\to G_{\hcoY}$
in Proposition~\ref{cla:union}, which is required for our cohomology
calculations of the sheaves $\widetilde{\sS}_{L}^{*}$, $\widetilde{\sQ}$,
$\widetilde{\sT}$ on $\widetilde{\hcoY}$ in Section~\ref{section:comp}.
The properties summarized in Lemma~\ref{cla:A} and Lemma~\ref{cla:ST}
will be used in our proof of Lemma~\ref{cla:key}.

\subsection{Birational models $F^{(k)}/\mZ_{2}$ of $F_{\widetilde{\hcoY}}$ }

\label{subsection:Bir}

Let us see that $F_{\widetilde{\hcoY}}$ is birationally equivalent
to the $\mZ_{2}$-quotient of the following $\mZ_{2}$-variety: \begin{equation}
\begin{aligned}{F}^{(1)} & :=\{([V_{2}^{(1)}],[V_{2}^{(2)}];[V_{4}^{(1)}],[V_{4}^{(2)}])\mid\\
 & \hskip3cmV_{2}^{(1)},V_{2}^{(2)}\subset V_{4}^{(1)}\cap V_{4}^{(2)},\dim V_{2}^{(1)}\cap V_{2}^{(2)}\geq1\}\\
 & \;\;\subset\mathrm{G}(2,V)\times\mathrm{G}(2,V)\times\mP(V^{*})\times\mP(V^{*}),\end{aligned}
\label{eq:FsY}\end{equation}
where $\mZ_{2}$ acts by the simultaneous exchanges $V_{2}^{(1)}\leftrightarrow V_{2}^{(2)}$
and $V_{4}^{(1)}\leftrightarrow V_{4}^{(2)}$. There is a natural
morphism ${F}^{(1)}/\mZ_{2}\to\ft{S}^{2}\mP(V^{*})$, and its fiber
over a point $([V_{4}^{(1)}],[V_{4}^{(2)}])$ with $V_{4}^{(1)}\not=V_{4}^{(2)}$
is isomorphic to $\mP(V_{3}^{*})\times\mP(V_{3}^{*})$ with $V_{3}=V_{4}^{(1)}\cap V_{4}^{(2)}$.
Therefore ${F}^{(1)}/\mZ_{2}\to\ft{S}^{2}\mP(V^{*})$ is birationally
equivalent to $F_{\widetilde{\hcoY}}\to G_{\hcoY}$ by Lemma \ref{lem:inverse-rk2}.

It turns out that ${F}^{(1)}/\mZ_{2}$ is not isomorphic to ${F}_{\widetilde{\hcoY}}$
but we can reconstruct $F_{\widetilde{\hcoY}}$ from ${F}^{(1)}/\mZ_{2}$
explicitly. Our reconstruction of $F_{\widetilde{\hcoY}}$ proceeds
as follows; we first construct the (anti-)flip $F^{(2)}\dashrightarrow F^{(4)}$
of a resolution ${F}^{(2)}\to{F}^{(1)}$ (Lemma \ref{cla:flip});
then we describe the strict transform $D^{(4)}$ on $F^{(4)}$ of
the inverse image $D^{(1)}$ in ${F}^{(1)}$ of the diagonal of $\widehat{G}=\mP(V^{*})\times\mP(V^{*})$
(Lemma \ref{cla:div}) and then consider a contraction of $D^{(4)}\subset F^{(4)}$
to obtain $F^{(4)}\to\hat{F}$; finally we divide $\hat{F}\to\hat{G}$
by the naturally induced $\mZ_{2}$-action to obtain $F_{\widetilde{\hcoY}}\to G_{\hcoY}$
(Proposition \ref{cla:F'}). The entire picture is summarized in the
diagram (\ref{eq:house}).

\vspace{0.3cm}
 Now let us first construct a small resolution of ${F}^{(1)}$. 

\begin{lem} \label{cla:smallres} Let \begin{equation}
\begin{aligned}{F}^{(2)} & :=\{([V_{2}^{(1)}],[V_{2}^{(2)}];[V_{3}];[V_{4}^{(1)}],[V_{4}^{(2)}])\mid V_{2}^{(1)},V_{2}^{(2)}\subset V_{3}\subset V_{4}^{(1)}\cap V_{4}^{(2)}\}\\
 & \;\;\subset\mathrm{G}(2,V)\times\mathrm{G}(2,V)\times\mathrm{G}(3,V)\times\mP(V^{*})\times\mP(V^{*})\end{aligned}
\label{eq:FsY'}\end{equation}
 and \begin{equation}
\widehat{G}':=\{([V_{3}];[V_{4}^{(1)}],[V_{4}^{(2)}])\mid V_{3}\subset V_{4}^{(1)}\cap V_{4}^{(2)}\}\subset\mathrm{G}(3,V)\times\mP(V^{*})\times\mP(V^{*}).\label{eqn:GY}\end{equation}
 Then, $1)$ $\widehat{G}'$ is the blow-up of $\widehat{G}:=\mP(V^{*})\times\mP(V^{*})$
along the diagonal, and ${F}^{(2)}$ has a $\mP^{2}\times\mP^{2}$-fibration
structure over $\widehat{G}'$. In particular, ${F}^{(2)}$ is smooth.
$2)$ Moreover, the natural morphism ${F}^{(2)}\to{F}^{(1)}$ is a
small resolution. \end{lem}

\begin{proof} The first part is almost obvious. We only show that
${F}^{(2)}\to{F}^{(1)}$ is a small resolution. Note that, since $V_{3}=V_{2}^{(1)}+V_{2}^{(2)}$
holds for $F^{(2)}$ when $V_{2}^{(1)}\not=V_{2}^{(2)}$, or $V_{3}=V_{4}^{(1)}\cap V_{4}^{(2)}$
holds for $F^{(2)}$ when $V_{4}^{(1)}\not=V_{4}^{(2)}$, the morphism
${F}^{(2)}\to{F}^{(1)}$ is isomorphic outside the diagonal set \begin{equation}
\Delta_{{F}^{(1)}}:=\{([V_{2}],[V_{2}];[V_{4}],[V_{4}])\mid V_{2}\subset V_{4}\}\simeq\mathrm{F}(2,4,V)\subset{F}^{(1)}.\label{eq:DelF}\end{equation}
 The fiber of ${F}^{(2)}\to{F}^{(1)}$ over a point $([V_{2}],[V_{2}];[V_{4}],[V_{4}])\in\Delta_{{F}^{(1)}}$
is \[
\{([V_{2}],[V_{2}];[V_{3}];[V_{4}],[V_{4}])\mid[V_{3}]\in\mathrm{G}(3,V),V_{2}\subset V_{3}\subset V_{4}\}\simeq\mP^{1}.\]
 Therefore the dimension of the exceptional set of ${F}^{(2)}\to{F}^{(1)}$
is equal to $\dim\Delta_{{F}^{(1)}}+1=9$, hence ${F}^{(2)}\to{F}^{(1)}$
is small. \end{proof}

\begin{lem} \label{cla:flip} 

{\rm 1)} $\sN_{\gamma/{F}^{(2)}}\simeq\sO_{\mP^{1}}(-1)^{\oplus3}\oplus\sO_{\mP^{1}}^{\oplus8}$
for non-trivial fibers $\gamma\simeq\mP^{1}$ of ${F}^{(2)}\to{F}^{(1)}$.
{\rm 2)} There exists another small resolution ${F}^{(4)}\to{F}^{(1)}$
which is isomorphic outside $\Delta_{{F}^{(1)}}$ and whose nontrivial
fiber is isomorphic to $\mP^{2}$. The variety ${F}^{(4)}$ is constructed
by taking the blow-up ${F}^{(3)}\to{F}^{(2)}$ along the exceptional
locus of ${F}^{(2)}\to{F}^{(1)}$ and contracting the exceptional
divisor of this blow-up in the other direction. \end{lem}

\begin{rem} The small resolution ${F}^{(4)}\to{F}^{(1)}$ as in the
above statement can be considered locally a family of the small resolution
of a $4$-dimensional singularity, which is studied in \cite{Ka}.
\hfill{}{[}{]} \end{rem}

\begin{proof} The main part of our proof is to determine the singularities
of ${F}^{(1)}$. To describe it, we set \[
Z:=\{([V_{2}^{(1)}],[V_{4}^{(1)}])\mid V_{2}^{(1)}\subset V_{4}^{(1)}\}\subset\mathrm{G}(2,V)\times\mP(V^{*})\]
 and consider the projection ${F}^{(1)}\to Z$. Let $F_{z}^{(1)}$
be the fiber of this projection over a point $z=([V_{2}^{(1)}],[V_{4}^{(1)}])\in Z$.
To describe the fiber $F_{z}^{(1)}$, we choose a basis $\{{\bf e}_{1},\dots,{\bf e}_{5}\}$
of $V$ such that $V_{2}^{(1)}=\langle{\bf e}_{1},{\bf e}_{2}\rangle$
and $V_{4}^{(1)}=\langle{\bf e}_{1},\dots,{\bf e}_{4}\rangle=(0,0,0,0,1)^{\perp}$.
Then the conditions $V_{2}^{(1)},V_{2}^{(2)}\subset V_{4}^{(1)}\cap V_{4}^{(2)}$
and $\dim(V_{2}^{(1)}\cap V_{2}^{(2)})\geq1$ for the points $([V_{2}^{(2)}],[V_{4}^{(2)}])\in\rG(2,V_{4}^{(1)})\times\mP(V^{*})$
on the fiber are easily analyzed to obtain \begin{equation}
F_{z}^{(1)}=\left\{ \left(\left[\begin{matrix}a_{1} & a_{2} & a_{3} & a_{4}\\
b_{1} & b_{2} & b_{3} & b_{4}\end{matrix}\right],[0,0,c_{3},c_{4},c_{5}]\right)\bigg|\;\rank\left(\begin{matrix}a_{3} & a_{4}\\
b_{3} & b_{4}\\
c_{4} & -c_{3}\end{matrix}\right)\leq1\right\} .\label{eq:fiberFz}\end{equation}
From this, we see that $F_{z}^{(1)}$ is singular only at the origin
$o$ of the affine chart with $\left(\begin{smallmatrix}a_{1} & a_{2}\\
b_{1} & b_{2}\end{smallmatrix}\right)\not=0$ and $c_{5}\not=0$, i.e., $([V_{2}^{(2)}],[V_{4}^{(2)}])=([V_{2}^{(1)}],[V_{4}^{(1)}])$,
and the singularity is isomorphic to the vertex of the cone over the
Segre variety $\mP^{1}\times\mP^{2}$.

We may consider ${F}^{(1)}\to Z$ locally as an equisingular family
of the cone over the Segre variety $\mP^{1}\times\mP^{2}$. It is
well-known that the cone $C$ over the Segre variety $\mP^{1}\times\mP^{2}$
has exactly two small resolutions $p_{1}\colon C_{1}\to C$ and $p_{2}\colon C_{2}\to C$,
where the exceptional locus $E_{1}$ of $p_{1}$ is a copy of $\mP^{2}$
with $\sN_{E_{1}/C_{1}}\simeq\sO_{\mP^{2}}(-1)^{\oplus2}$, and the
exceptional locus $E_{2}$ of $p_{2}$ is a copy of $\mP^{1}$ with
$\sN_{E_{2}/C_{2}}\simeq\sO_{\mP^{1}}(-1)^{\oplus3}$. We can conclude
that ${F}^{(2)}\to{F}^{(1)}$ is locally a family of $p_{2}\colon C_{2}\to C$,
and then we have the assertion (cf. \cite{Ka}). ${F}^{(4)}\to{F}^{(1)}$
is nothing but locally a family of $p_{1}\colon C_{1}\to C$. \end{proof}

\subsection{Divisors $D^{(k)}$ in $F^{(k)}$}

Now, let $\Delta_{\mP}=\rG(4,V)$ be the diagonal of $\widehat{G}=\mP(V^{*})\times\mP(V^{*})$
and ${D}^{(1)}$ the inverse image of $\Delta_{\mP}$ by the natural
morphism ${F}^{(1)}\to\mP(V^{*})\times\mP(V^{*})$, namely, \[
{D}^{(1)}:=\{([V_{2}^{(1)}],[V_{2}^{(2)}];[V_{4}],[V_{4}])\mid V_{2}^{(1)},V_{2}^{(2)}\subset V_{4},\dim V_{2}^{(1)}\cap V_{2}^{(2)}\geq1\}\subset{F}^{(1)}.\]
 Let \begin{eqnarray*}
{D}^{(2)}:=\{([V_{2}^{(1)}],[V_{2}^{(2)}];[V_{3}];[V_{4}],[V_{4}])\mid V_{2}^{(1)},V_{2}^{(2)}\subset V_{3}\subset V_{4}\}\subset{F}^{(2)}.\end{eqnarray*}
The natural morphism ${D}^{(2)}\to{D}^{(1)}$ over $\Delta_{\mP}$
is the restriction of the morphism ${F}^{(2)}\to{F}^{(1)}$ in Lemma
\ref{cla:smallres}. ${D}^{(2)}$ has a $\mP^{2}\times\mP^{2}$-bundle
structure over the flag variety $\mathrm{F}(3,4,V)$. In particular
${D}^{(2)}$ is a smooth variety. Hence ${D}^{(1)}$ is a prime divisor
on ${F}^{(1)}$ and ${D}^{(2)}$ is its strict transform on ${F}^{(2)}$.

\begin{lem} \label{cla:D_4} Set \begin{equation}
\begin{aligned}{D}^{(4)} & :=\{([V_{1}];[V_{2}^{(1)}],[V_{2}^{(2)}];[V_{4}],[V_{4}])\mid V_{2}^{(1)},V_{2}^{(2)}\subset V_{4},V_{1}\subset V_{2}^{(1)}\cap V_{2}^{(2)}\}\\
 & \;\;\subset\mP(V)\times\mathrm{G}(2,V)\times\mathrm{G}(2,V)\times\mP(V^{*}).\end{aligned}
\label{eqn:FsY''}\end{equation}
 Then ${D}^{(4)}$ is the strict transform on ${F}^{(4)}$ of ${D}^{(2)}$.
Moreover, the restriction ${D}^{(2)}\dashrightarrow{D}^{(4)}$ of
the $($anti-$)$flip ${F}^{(2)}\dashrightarrow{F}^{(4)}$ is a family
of Atiyah flops. Noting ${D}^{(2)}$ $($resp. ${D}^{(4)})$ has a
natural $\mP^{2}\times\mP^{2}$-fibration structure over $\mathrm{F}(3,4,V)$
$($resp. $\mathrm{F}(1,4,V))$, we obtain the following commutative
diagram$:$ \begin{equation}
\begin{matrix}\xymatrix{{D}^{(2)}\ar[d]_{\;_{\text{\ensuremath{\mP^{2}\times\mP^{2}}-fib.}}}\ar@{-->}[rr]^{\;_{\text{Atiyah flop}}}\ar[dr] &  & {D}^{(4)}\ar[d]^{\;_{\text{\ensuremath{\mP^{2}\times\mP^{2}}-fib.}}}\ar[dl]\\
\mathrm{F}(3,4,V)\ar[dr] & {D}^{(1)}\ar[d] & \mathrm{F}(1,4,V)\ar[dl]\\
 & \Delta_{\mP}.}
\end{matrix}\label{eq:smallhouse}\end{equation}
 \end{lem}

\begin{proof} Let $D_{[V_{4}]}^{(1)}$ be the fiber of $D^{(1)}\to\Delta_{\mP}$
over the diagonal point $([V_{4}],[V_{4}])$. Then, from the definition,
we have \[
D_{[V_{4}]}^{(1)}=\left\{ ([V_{2}^{(1)}],[V_{2}^{(2)}])\mid V_{2}^{(1)},V_{2}^{(2)}\subset V_{4},\dim(V_{2}^{(1)}\cap V_{2}^{(2)})\geq1\right\} \]
which is singular along $[V_{2}^{(1)}]=[V_{2}^{(2)}]$. As in the
proof of Lemma \ref{cla:flip}, we fix $V_{2}^{(1)}=\langle\bm{e}_{1},\bm{e}_{2}\rangle$
and $V_{4}=\langle\bm{e}_{1},\bm{e}_{2},\bm{e}_{3},\bm{e}_{4}\rangle$.
Then the natural restriction $D_{[V_{4}]}^{(1)}\vert_{[V_{2}^{(1)}]}$
can be described (by (\ref{eq:fiberFz})) as\[
D_{[V_{4}]}^{(1)}\vert_{[V_{2}^{(1)}]}=\left\{ \bigg([V_{2}^{(1)}],\left[\begin{matrix}a_{1} & a_{2} & a_{3} & a_{4}\\
b_{1} & b_{2} & b_{3} & b_{4}\end{matrix}\right]\bigg)\bigg|\,\rank\left(\begin{matrix}a_{3} & a_{4}\\
b_{3} & b_{4}\end{matrix}\right)\leq1\right\} .\]
In this form, it is clear that the two small resolutions $F^{(2)}\to F^{(1)}$
and $F^{(4)}\to F^{(1)}$ given in Lemma \ref{cla:flip} restricts
to $D^{(2)}\to D^{(1)}$ and $D^{(4)'}\to D^{(1)}$, respectively.
We show the equality $D^{(4)'}$ to $D^{(4)}$. 

Since the equality holds over the smooth locus of $D^{(1)}$, it suffices
to see the correspondence between the exceptional set over the diagonal
set $[V_{2}^{(1)}]=[V_{2}^{(2)}]$. To see this, we fix $[V_{2}^{(1)}]=\langle\bm{e}_{1},\bm{e}_{2}\rangle$
and consider $D_{[V_{4}]}^{(1)}\vert_{[V_{2}^{(1)}]}$ as above. Then
the exceptional set of $D^{(4)'}$ over $D_{[V_{4}]}^{(1)}\vert_{[V_{2}^{(1)}]}$
consists of points $[V_{1}]$ such that $V_{1}=V_{2}^{(1)}\cap\tilde{V}_{2}^{(2)}(s,t)=\mC(t,-s,0,0)$
with $[\tilde{V}_{2}^{(2)}(s,t)]=\left[\begin{smallmatrix}1 & 0 & a_{3} & a_{4}\\
0 & 1 & b_{3} & b_{4}\end{smallmatrix}\right]$ and \[
a_{3}:b_{3}=a_{4}:b_{4}=s:t\;\in\mP^{1}.\]
This exactly describes $\mP^{1}$ over the diagonal set of $D^{(4)}$. 

Other statements follow directly from the definitions. \end{proof}

\begin{rem} From the proof of the above lemma, we see that the strict
transform in $F^{(3)}$ (see Lemma \ref{cla:flip}) of the divisor
$D^{(2)}\subset F^{(2)}$ may be described by \begin{equation}
D^{(3)}:=\left\{ ([V_{1}];[V_{2}^{(1)}],[V_{2}^{(2)}];[V_{3}];[V_{4}],[V_{4}])\bigg|\begin{matrix}V_{1}\subset V_{2}^{(1)}\cap V_{2}^{(2)}\;\;\;\\
V_{2}^{(1)},V_{2}^{(2)}\subset V_{3}\subset V_{4}\end{matrix}\right\} .\label{eq:divisorD3}\end{equation}
Then the natural projections $D^{(2)}\leftarrow D^{(3)}\rightarrow D^{(4)}$
describe the Atiyah flop. $\hfill${[}{]}

\end{rem}

$\;$

\subsection{A divisorial contraction $F^{(4)}\to\hat{F}$ }

We can now complete our construction of $F_{\widetilde{\hcoY}}=\hat{F}/\mZ_{2}$.
We start with the following contraction of the divisor $D^{(4)}\subset F^{(4)}.$

\begin{lem} \label{cla:div} Let ${F}^{(4)}$ be as in Lemma $\ref{cla:flip}$.
Then there exists a divisorial contraction ${F}^{(4)}\to\widehat{F}$
which contracts the strict transform ${D}^{(4)}$ of ${D}^{(1)}$
to the locus isomorphic to the flag variety $\mathrm{F}(1,4,V)$$($see
$(\ref{eq:smallhouse}))$. The discrepancy of ${D}^{(4)}$ is two.
\end{lem}

\begin{proof} Let $\Delta'_{\mP}$ be the inverse image in $\widehat{G}'$
of $\Delta_{\mP}$. Note that $\Delta'_{\mP}\simeq\mathrm{F}(3,4,V)$.
Let $\Gamma$ be a fiber of ${D}^{(2)}\to\Delta'_{\mP}$, where we
recall $\Gamma\simeq\mP^{2}\times\mP^{2}$. Then, by the proof of
Lemma \ref{cla:D_4}, $\Gamma$ intersects the flopping locus along
the diagonal transversally. Take a line $r\subset\mP^{2}\times\mP^{2}$
which is contained in a fiber of a projection $\Gamma\to\mP^{2}$
and intersects the flopping locus. Then its strict transform $r'$
on ${D}^{(4)}$ is contracted by the morphism ${D}^{(4)}\to\mathrm{F}(1,4,V)$.
Since ${F}^{(2)}\to\widehat{G}'$ is a $\mP^{2}\times\mP^{2}$-fibration
and ${D}^{(2)}$ is the pull-back of $\Delta'_{\mP}$, we see that
$K_{{F}^{(2)}}\cdot r=-3$ and ${D}^{(2)}\cdot r=0$. By the standard
calculations of the changes of the intersection numbers by the flip,
we have $K_{{F}^{(4)}}\cdot r'=-3+1=-2$ and ${D}^{(4)}\cdot r'=0-1=-1$.
These equalities of the intersection numbers still hold for any line
in a ruling of a fiber of ${D}^{(4)}\to\mathrm{F}(1,4,V)$.

We show that $-K_{F^{(4)}}+2D^{(4)}$ is relatively nef over $\widehat{G}$.
Let $\gamma$ be a curve on $F^{(4)}$ which is contracted to a point
$t$ on $\widehat{G}$. If $t\not\in\Delta_{\mP}$, then $(-K_{F^{(4)}}+2D^{(4)})\cdot\gamma>0$
since $D^{(4)}\cap\gamma=\emptyset$ and $F^{(4)}\to\widehat{G}$
is a $\mP^{2}\times\mP^{2}$ fibration outside $\Delta_{\mP}$. If
$t\in\Delta_{\mP}$ and $\gamma$ is an exceptional curve of $F^{(4)}\to F^{(1)}$,
then $(-K_{F^{(4)}}+2D^{(4)})\cdot\gamma>0$ since $-K_{F^{(4)}}\cdot\gamma>0$
and $D^{(4)}\cdot\gamma>0$. In the remaining cases, $t\in\Delta_{\mP}$
and $\gamma\subset D^{(4)}$. Therefore we have only to consider the
relative nefness of $(-K_{F^{(4)}}+2D^{(4)})|_{D^{(4)}}$ over $\Delta_{\mP}$.
Now we take for $\gamma$ any line in a ruling of a fiber of ${D}^{(4)}\to\mathrm{F}(1,4,V)$.
As we see above, $(-K_{F^{(4)}}+2D^{(4)})\cdot\gamma=0$. Therefore
$(-K_{F^{(4)}}+2D^{(4)})|_{D^{(4)}}$ is the pull-back of some divisor
$D_{F}$ on $\mathrm{F}(1,4,V)$. It suffices to show $D_{F}$ is
relatively nef over $\Delta_{\mP}$, which is true since an exceptional
curve of $D^{(4)}\to D^{(1)}$ is positive for $(-K_{F^{(4)}}+2D^{(4)})|_{D^{(4)}}$
as above and is mapped to a curve on a fiber of $\mathrm{F}(1,4,V)\to\Delta_{\mP}$.
Therefore $-K_{F^{(4)}}+2D^{(4)}$ is relatively nef over $\widehat{G}$.

Finally, from the above argument, we see that $(-K_{F^{(4)}}+2D^{(4)})^{\perp}\cap\overline{\mathrm{NE}}(F^{(4)}/\widehat{G})$
is generated by the numerical class of the curves on fibers of $D^{(4)}\to\mathrm{F}(1,4,V)$.
In particular, $(-K_{F^{(4)}}+2D^{(4)})^{\perp}\cap\overline{\mathrm{NE}}(F^{(4)}/\widehat{G})\subset(K_{F^{(4)}})^{<0}$.
Therefore, by Mori theory, there exists a contraction associated to
this extremal face, which is nothing but the divisorial contraction
contracting $D^{(4)}$ to $\mathrm{F}(1,4,V)$.

By the equalities $K_{{F}^{(4)}}\cdot r'=-2$ and ${D}^{(4)}\cdot r'=-1$,
we see that the discrepancy of ${D}^{(4)}$ is two. \end{proof}

Recall the $\mZ_{2}$-action on $F^{(1)}$ described in (\ref{eq:FsY}).
Since all the morphisms constructed to obtain $\widehat{F}$ from
${F}^{(1)}$ are $\mZ_{2}$-equivariant, the variety $\widehat{F}$
also has a naturally induced $\mZ_{2}$-action. We also note that
\[
{G}'_{{\hcoY}}:=\widehat{G}'/\mZ_{2}\simeq\Hilb^{2}\mP(V^{*}).\]
Now we have 

\begin{prop} \label{cla:F'} The $\Lrho_{\widetilde{\hcoY}}$-exceptional
divisor $F_{\widetilde{\hcoY}}$ is isomorphic to $\widehat{F}/\mZ_{2}$.
\end{prop} 

\begin{proof} We compare the morphisms $a\colon F_{\widetilde{\hcoY}}\to G_{{\hcoY}}$
and $b\colon\widehat{F}/\mZ_{2}\to G_{{\hcoY}}$. By \cite[Lemma 5.5]{Tk}
for example, it suffices to check the following properties hold for
them: 
\begin{itemize}
\item Both $F_{\widetilde{\hcoY}}$ and $\widehat{F}/\mZ_{2}$ are normal. 
\item the morphisms $a$ and $b$ are isomorphic to each other in codimension
one. 
\item $-K_{F_{\widetilde{\hcoY}}}$ and $-K_{\widehat{F}/\mZ_{2}}$ are
$\mQ$-Cartier. 
\item $-K_{F_{\widetilde{\hcoY}}}$ is $a$-ample and $-K_{\widehat{F}/\mZ_{2}}$
is $b$-ample. 
\end{itemize}
The variety $F_{\widetilde{\hcoY}}$ is normal. Indeed, it satisfies
the $S_{2}$ condition since it is a Cartier divisor on a smooth variety.
Moreover, it satisfies the $R_{1}$ condition since it is a $\mP^{2}\times\mP^{2}$-fibration
outside the locus of codimension two by Lemmas \ref{lem:inverse-rk2}
and \ref{lem:inverse-rk1}. We see also that the variety $\widehat{F}/\mZ_{2}$
is normal by its explicit construction as above.

The morphisms $a$ and $b$ are isomorphic outside $G_{\hcoY}^{1}$
by Lemma \ref{lem:inverse-rk2} and the construction of ${F}^{(1)}/\mZ_{2}$.
Moreover, the inverse images of $G_{\hcoY}^{1}$ by the morphism $a$
has codimension two in $F_{\widetilde{\hcoY}}$ by Lemma \ref{lem:inverse-rk1}
(and Remark after it), and the inverse images of $G_{\hcoY}^{1}$
by the morphism $b$ has codimension two in $\widehat{F}/\mZ_{2}$
by the construction of $\widehat{F}/\mZ_{2}$. Therefore the morphisms
$a$ and $b$ are isomorphic to each other in codimension one.

The divisor $-K_{F_{\widetilde{\hcoY}}}$ is $\mQ$-Cartier since
${F_{\widetilde{\hcoY}}}$ is a divisor on the smooth variety $\widetilde{\hcoY}$.
Since the relative Picard number $\rho(\widetilde{\hcoY}/\hcoY)$
is one and $a$ is generically a $\mP^{2}\times\mP^{2}$-fibration,
we see that $-K_{F_{\widetilde{\hcoY}}}$ is $a$-ample.

We see that similar facts hold for the morphism $b$. Also we see
that $-K_{\widehat{F}/\mZ_{2}}$ is $\mQ$-Cartier. Indeed, by Lemma
\ref{cla:div}, the discrepancy of ${D}^{(4)}$ is two. Then, by the
Kawamata-Shokurov base point free theorem, $-K_{{F}^{(4)}}-2{D}^{(4)}$
is the pull-back of a Cartier divisor on $\widehat{F}$, which turns
out to be the anti-canonical divisor $-K_{\widehat{F}}$. Thus $-K_{\widehat{F}/\mZ_{2}}$
is $\mQ$-Cartier.

To show that $-K_{\widehat{F}/\mZ_{2}}$ is $b$-ample, it suffices
to see the relative Picard number $\rho((\widehat{F}/\mZ_{2})/G_{{\hcoY}})$
is one because $b$ is generically a $\mP^{2}\times\mP^{2}$-fibration.
Let us note that the relative Picard number $\rho({F}^{(2)}/\widehat{G}')$
is two since ${F}^{(2)}\to\widehat{G}'$ is a $\mP^{2}\times\mP^{2}$-fibration
and it is easy to see that it is a composite of two $\mP^{2}$-fibrations.
Moreover we have $\rho^{\mZ_{2}}({F}^{(2)}/\widehat{G}')=1$ since
the two rulings of a fiber $\mP^{2}\times\mP^{2}$ of ${F}^{(2)}\to\widehat{G}'$
are exchanged by the $\mZ_{2}$-action. Therefore $\rho^{\mZ_{2}}({F}^{(2)})=3$
since $\rho^{\mZ_{2}}(\widehat{G}')=2$. It holds that $\rho^{\mZ_{2}}({F}^{(4)})=3$
since the flip preserves the Picard number and it is $\mZ_{2}$-equivariant.
Since a divisorial contraction decreases the Picard number at least
by one, we have $\rho^{\mZ_{2}}(\widehat{F})\leq2$. Now we see that
$\rho((\widehat{F}/\mZ_{2})/G_{{\hcoY}})$ is one since $\rho(G_{{\hcoY}})=1$
and the morphism $\widehat{F}/\mZ_{2}\to G_{{\hcoY}}$ is non-trivial.
Therefore we conclude that $-K_{\widehat{F}/\mZ_{2}}$ is $b$-ample.
\end{proof}

$\;$ 

In summary, we have obtained the following diagram:\begin{equation}
\begin{matrix}\xymatrix{ & {F}^{(3)}\ar[dl]\ar[dr]\\
{F}^{(2)}\ar[ddr]_{\;_{\text{\ensuremath{\mP^{2}\times\mP^{2}}-fib.}}}\ar@{-->}[rr]^{\;_{\text{(anti-)flip(Lem.\ref{cla:flip})}}}\ar[dr] &  & {F}^{(4)}\ar[d]^{\;_{\text{div. cont.(Lem.\ref{cla:div})}}}\ar[dl]\\
 & {F}^{(1)}\ar[dr] & \widehat{F}\ar[d]\ar[r]_{\;_{\text{\ensuremath{\mZ_{2}}-quot.}}} & F_{\widetilde{\hcoY}}\ar[d]^{\Lrho_{\widetilde{\hcoY}}|_{F_{\widetilde{\hcoY}}}}\\
 & \widehat{G}'\ar[r]_{_{\text{{diag.blow up}}}} & \widehat{G}\ar[r]^{\;_{\text{\ensuremath{\mZ_{2}}-quot.}}} & {G}_{{\hcoY}}.}
\end{matrix}\label{eq:house}\end{equation}

$\;$

\subsection{Flattening of $F_{\widetilde{\hcoY}}\to G_{\hcoY}$}

Here we describe the fibers of ${F}^{(3)}\to\widehat{G}'$ in the
diagram (\ref{eq:house}) and show the flatness of the morphism $F^{(3)}\to\widehat{G}'$.

\begin{prop} \label{cla:union} $\;$

\begin{myitem2} \item[\rm (1)] The fiber of ${F}^{(3)}\to\widehat{G}'$
over a point $([V_{3}];[V_{4}^{(1)}],[V_{4}^{(2)}])$ with $V_{4}^{(1)}\not=V_{4}^{(2)}$
is $\mP(V_{3}^{*})\times\mP(V_{3}^{*})$. 

\item[\rm (2)] The fiber of ${F}^{(3)}\to\widehat{G}'$ over a point
$([V_{3}];[V_{4}],[V_{4}])$ is the union of the following two $4$-dimensional
varieties $A$ and $B:$ 

\begin{itemize} \item$A\simeq\mP(\sO_{\mP(V_{3}^{*})}\oplus T_{\mP(V_{3}^{*})})$,
which is isomorphic to the restriction of $\mP(\sO_{\mathrm{G}(2,V_{4})}\oplus\sU_{\rG(2,V_{4})}^{*}(1))$
over $\mP(V_{3}^{*})=\mathrm{G}(2,V_{3})\subset\mathrm{G}(2,V_{4})$. 

\item$B$ is the blow-up of $\mP(V_{3}^{*})\times\mP(V_{3}^{*})$
along the diagonal $\Delta_{V_{3}^{*}}$. It is endowed with a morphism
$p_{B}\colon B\to\mP(V_{3})$ induced from the rational map \begin{eqnarray*}
\mP(V_{3}^{*})\times\mP(V_{3}^{*})\setminus\Delta_{V_{3}^{*}} & \to & \mP(V_{3})\\
([V_{2}^{(1)}],[V_{2}^{(2)}]) & \mapsto & [V_{2}^{(1)}\cap V_{2}^{(2)}],\end{eqnarray*}
 and is a $\mP^{1}\times\mP^{1}$-bundle over $\mP(V_{3})$. \end{itemize}

\noindent In particular the morphism ${F}^{(3)}\to\widehat{G}'$
is flat. Moreover, the intersection $E_{AB}:=A\cap B$ is $\mP(T_{\mP(V_{3}^{*})})$
in $A$, which is the restriction of $E_{\sigma}=\mP(\sU_{\rG(2,V_{4})}^{*}(1))$
$($cf. Lemma $\ref{lem:inverse-rk1}$$)$ over $\rG(2,V_{3})$,
and also $E_{AB}$ is the exceptional divisor of $B\to\mP(V_{3}^{*})\times\mP(V_{3}^{*})$
in $B$. 

\end{myitem2}\end{prop}

\begin{proof} Part (1) follows from the construction of ${F}^{(2)}\to\widehat{G}'$.

We show Part (2). The fiber of ${F}^{(2)}\to\widehat{G}'$ over a
point $([V_{3}];[V_{4}],[V_{4}])$ is $\mP(V_{3}^{*})\times\mP(V_{3}^{*})$.
The intersection of the fiber $\mP(V_{3}^{*})\times\mP(V_{3}^{*})$
with the exceptional locus of ${F}^{(2)}\to{F}^{(1)}$ is \[
\{([V_{2}],[V_{2}];[V_{3}];[V_{4}],[V_{4}])\mid V_{2}\subset V_{3}\}\simeq\mP^{2},\]
 which is nothing but the diagonal of $\mP(V_{3}^{*})\times\mP(V_{3}^{*})$.
Therefore we have $B$ as an irreducible component of the fiber of
${F}^{(3)}\to\widehat{G}'$ over the point $([V_{3}];[V_{4}],[V_{4}])$.

Another component $A$ is a $\mP^{2}$-bundle over the diagonal of
$\mP(V_{3}^{*})\times\mP(V_{3}^{*})$ since the exceptional divisor
of ${F}^{(3)}\to{F}^{(2)}$ is a $\mP^{2}$-bundle over the exceptional
locus of ${F}^{(2)}\to{F}^{(1)}$. As described in Lemma \ref{cla:flip},
it is introduced by the blow-up with respect to the three normal coordinates
of the singular locus of $F^{(1)}.$ Hence $A$ is given by the projective
bundle $\mP(\sO_{\Delta_{V_{3}^{*}}}\oplus\sN_{\Delta_{V_{3}^{*}}})$
over the diagonal. Now we note that $\sN_{\Delta_{V_{3}^{*}}}\cong T_{\mP(V_{3}^{*})}$
for the normal bundle. Also we note that the image of $\Delta_{V_{3}^{*}}$
under $F^{(2)}\to F^{(1)}$ is identified with $\mathrm{G}(2,V_{3})=\mP(V_{3}^{*})$
in $\mathrm{G}(2,V_{4})$ and the isomorphism $\sU_{\rG(2,V_{4})}^{*}|_{\mP(V_{3}^{*})}\simeq T_{\mP(V_{3}^{*})}(-1)$
holds. Therefore, we see that the image of $A$ under $F^{(3)}\to F^{(4)}$
can be described by the restriction of $\mP(\sO_{\mathrm{G}(2,V_{4})}\oplus\sU_{V_{4}}^{*}(1))$
over $\mP(2,V_{3})$. 

The properties of the intersection $E_{AB}$ follow from the above
descriptions. \end{proof}

\begin{rem} Recall the description (\ref{eq:divisorD3}) of the divisor
$D^{(3)}$ in $F^{(3)}$. The component $B$ above is nothing but
the restriction $D^{(3)}\vert_{([V_{3}];[V_{4}],[V_{4}])}$ of $D^{(3)}$
over the point $([V_{3}];[V_{4}],[V_{4}])$. Also the component $A$
can be identified in Lemma \ref{lem:inverse-rk1} (2) and also in
Fig. 3. \hfill{[}{]}

\end{rem}

\vspace{0.2cm}
 By Proposition~\ref{cla:union} and the results in the previous
subsection summarized in (\ref{eq:house}), we obtain the flattening
of $F_{\widetilde{\hcoY}}\to G_{\hcoY}$, \begin{equation}
\begin{matrix}\xymatrix{{F}^{(3)}\ar[r]\ar[d] & F_{\widetilde{\hcoY}}\ar[d]\\
\widehat{G}'\ar[r] & G_{\hcoY}.}
\end{matrix}\label{eq:flatten}\end{equation}
 The flatness of the morphism ${F}^{(3)}\to\widehat{G}'$ is crucial
for our proof of Lemma \ref{cla:key}. By this property, we can reduce
computations of cohomology groups on ${F}_{\widetilde{\hcoY}}$ to
those on ${F}^{(3)}$ and then those on special fibers of ${F}^{(3)}\to\widehat{G}'$.

$\;$

\subsection{The pull-backs of $\widetilde{\sS}_{L}^{*}$, $\widetilde{\sQ}$
and $\widetilde{\sT}$ on $A$ and $B$}

In this subsection, we consider the situation of Proposition~\ref{cla:union}
(2) fixing $V_{3}$ and $V_{4}$. We describe the pull-backs of the
divisor $F_{\widetilde{\hcoY}}$ and the sheaves $\widetilde{\sS}_{L}^{*}$,
$\widetilde{\sQ}$,\, $\widetilde{\sT}$ on $A$ and $B$ in the
fiber of $F^{(3)}\to\hat{G}'$.

\begin{lem} \label{cla:A} Denote by $A_{\widetilde{\hcoY}}$ the
image of $A$ on $\widetilde{\hcoY}$, by $A_{\hcoY_{2}}$ the strict
transform on $\hcoY_{2}$ of $A_{\widetilde{\hcoY}}$, and by $A_{\hcoY_{3}}$
the image on $\hcoY_{3}$ of $A_{\hcoY_{2}}$. Then, 

\begin{myitem2} 

\item[\rm(1)]

$A\to A_{\widetilde{\hcoY}}$ is the contraction of $E_{AB}\simeq\mP(T_{\mP(V_{3}^{*})})$
to $\mP(V_{3})$. $\mP(V_{3})$ is given by the image of $B$ by the
morphism $p_{B}$ in Proposition $\ref{cla:union}$ $(2)$, and is
equal to the singular locus of $A_{\widetilde{\hcoY}}$ $($see Fig.$3)$,

\item[\rm(2)] $A_{\widetilde{\hcoY}}=\overline{\Gamma(V_{3},V_{4})}$
with $\Gamma(V_{3},V_{4}):=\cup_{V_{2}\subset V_{3}}\Gamma(V_{2},V_{4})$,

\item[\rm(3)] $A_{\hcoY_{2}}\to A_{\widetilde{\hcoY}}$ is the
blow-up along the image on $A_{\widetilde{\hcoY}}$ of the section
$s_{A}$ of $A\to\mP(V_{3}^{*})$ associated to the injection $\sO_{\mP(V_{3}^{*})}\hookrightarrow\sO_{\mP(V_{3}^{*})}\oplus T_{\mP(V_{3}^{*})}$,

\item[\rm(4)]let $\widehat{A}\to A$ be the blow-up of $A$ along
the section $s_{A}$. Then there exists a natural morphism $\widehat{A}\to A_{\hcoY_{2}}$,
which is the blow-up of $A_{\hcoY_{2}}$ along its singular locus,
and

\item[\rm(5)] $A_{\hcoY_{3}}\simeq A_{\hcoY_{2}}$. Moreover,
$\Prt_{\rho}|_{A_{\hcoY_{3}}}$ is isomorphic to $\mP(T_{\mP(V_{3})}(-1))$
and $\Lrho_{\hcoY_{3}}|_{A_{\hcoY_{3}}}\colon A_{\hcoY_{3}}\to\mP(V_{3})$
is a quadric cone fibration.

\[
\xymatrix{ &  & \widehat{A}\ar[dl]\ar[dr]\\
A_{\hcoY_{3}}\ar@{<-}[r]^{\;_{\text{isom.}}}\ar[d]_{\Lpi_{A_{\hcoY_{3}}}} & A_{\hcoY_{2}}\ar[dr] &  & A\ar[dl]\ar[d]\\
\mP(V_{3}) &  & A_{\widetilde{\hcoY}} & \mP(V_{3}^{*}).}
\]

\end{myitem2}\end{lem}

\begin{proof} The claims (1)--(3) follow from Lemmas \ref{lem:inverse-rk1},
\ref{cla:union}. The claim (4) is almost obvious.

Now we prove (5). By the proof of Lemma \ref{lem:inverse-rk1}, the
fiber of $\Lrho_{\hcoY_{3}}\colon A_{\hcoY_{3}}\to\mP(V_{3})$ over
$[V_{1}]\in\mP(V_{3})$ is $\overline{\Gamma(V_{1},V_{3})}$, which
is defined similarly to $\overline{\Gamma(V_{1},V_{4})}$. Noting
that $\overline{\Gamma(V_{1},V_{4})}$ is the cone over $v_{2}(\mP(V_{4}/V_{1}))$,
we see that $\overline{\Gamma(V_{1},V_{3})}$ is the cone over $v_{2}(\mP(V_{3}/V_{1}))$,
namely, $\overline{\Gamma(V_{1},V_{3})}$ is the quadric cone. Therefore,
$\Lrho_{\hcoY_{3}}\colon A_{\hcoY_{3}}\to\mP(V_{3})$ is a quadric
cone fibration. In particular, $\rho(A_{\hcoY_{3}})=2$. On the other
hand, we have $\rho(A_{{\hcoY}_{2}})=2$ since $\rho(A_{\widetilde{\hcoY}})=1$
and $A_{{\hcoY}_{2}}\to A_{\widetilde{\hcoY}}$ is a simple blow-up.
Thus $A_{{\hcoY}_{2}}\to A_{\hcoY_{3}}$ must be an isomorphism since
it is birational.

Finally we show $\Prt_{\rho}|_{A_{\hcoY_{3}}}\simeq\mP(T_{\mP(V_{3})}(-1))$.
Note that $\Prt_{\rho}|_{A_{\hcoY_{3}}}$ is isomorphic to the exceptional
divisor $G$ of $\widehat{A}\to A$, which we determine from now on.
Let $\sI_{s_{A}}$ be the ideal sheaf of the section $s_{A}$ in $A$.
Note that $\sO_{\mP(\sO_{\mP(V_{3}^{*})}\oplus T_{\mP(V_{3}^{*})})}(1)|_{s_{A}}=\sO_{s_{A}}$.
Tensoring $0\to\sI_{s_{A}}\to\sO_{A}\to\sO_{s_{A}}\to0$ with $\sO_{\mP(\sO_{\mP(V_{3}^{*})}\oplus T_{\mP(V_{3}^{*})})}(1)$
and pushing forward to $\mP(V_{3}^{*})$, we see that $\sI_{s_{A}}/\sI_{s_{A}}^{2}\simeq\Omega_{\mP(V_{3}^{*})}$.
Therefore $G$ is isomorphic to $\mP(T_{\mP(V_{3}^{*})})$. It is
well-known that $\mP(T_{\mP(V_{3}^{*})})$ is isomorphic to the incident
variety $\{([V_{1}],[V_{2}])\mid V_{1}\subset V_{2}\}\subset\mP(V_{3})\times\mP(V_{3}^{*})$,
which is also isomorphic to $\mP(T_{\mP(V_{3})}(-1))$.\end{proof}

For a locally free sheaf $\sE$ on $\widetilde{\hcoY}$, we denote
by $\sE_{A}$ and $\sE_{B}$ its pull-backs on $A$ and $B$, respectively
unless stated otherwise. Denote by $H_{A}$ the pull back on $A$
of $\sO_{\mP(V_{3}^{*})}(1)$, and $F_{A}$ and $F_{B}$ the pull-backs
of (the line bundle) $F_{\widetilde{\hcoY}}$ to $A$ and $B$, respectively.

\begin{lem} \label{cla:ST} 

\begin{myitem2}\item[\rm(1)]

$F_{A}\sim-(E_{AB}+2H_{A})$, $(\widetilde{\sS}_{L}^{*})_{A}\simeq\widetilde{\sQ}_{A}\simeq\sO_{A}\oplus\sV$,
and $\widetilde{\sT}_{A}\simeq\sO_{A}^{\oplus2}\oplus\sV$, where
$\sV$ is a locally free sheaf obtained as a unique nonsplit extension
\[
0\to\sO_{A}(H_{A}+E_{AB})\to\sV\to\sO_{A}(H_{A})\to0.\]

\item[\rm(2)] $\sO_{B}(F_{B})\simeq p_{B}^{*}\sO_{\mP(V_{3})}(-1)$,
$(\widetilde{\sS}_{L}^{*})_{B}\simeq\widetilde{\sQ}_{B}\simeq\sO_{B}\oplus p_{B}^{*}T_{\mP(V_{3})}(-1)$,
and $\widetilde{\sT}_{B}\simeq\sO_{B}^{\oplus2}\oplus p_{B}^{*}T_{\mP(V_{3})}(-1)$,
where $p_{B}\colon B\to\mP(V_{3})$ is as in Proposition $\ref{cla:union}$
$(2)$. 

\end{myitem2}\end{lem}

\begin{proof} $\empty$

\textbf{Step 1. $\det(\widetilde{\sS}_{L}^{*})_{A}=\det\widetilde{\sQ}_{A}=E_{AB}+2H_{A}$.}\\
 By (\ref{eq:detdiff}), we have only to determine $\det\widetilde{\sQ}_{A}$.
Let $L_{\widehat{A}}$ be the pull-back of $\sO_{\mP(V_{3})}(1)$,
and $G$ the exceptional divisor for $\widehat{A}\to A$. Note that
$G$ is the pull-back of the exceptional divisor $F_{\rho}$ of $\hcoY_{2}\to\widetilde{\hcoY}$.
By Proposition \ref{cla:M}, we have $\det\widetilde{\sQ}_{A}=G+L_{\widehat{A}}$
since $M_{\widetilde{\hcoY}}$ is trivial on a fiber of $F_{\widetilde{\hcoY}}\to G_{\hcoY}$.
Therefore it suffices to show $\widehat{E}_{AB}+2H_{\widehat{A}}-G=L_{\widehat{A}}$,
where $H_{\widehat{A}}$ and $\widehat{E}_{AB}$ are the pull-backs
on $\widehat{A}$ of $H_{A}$ and $E_{AB}$, respectively. Note that
we can write $a\widehat{E}_{AB}+bH_{\widehat{A}}-cG=L_{\widehat{A}}$
with some $a,b,c\in\mZ$. Since $\widehat{E}_{AB}\cap G=\emptyset$,
we have $a\widehat{E}_{AB}+bH_{\widehat{A}}|_{\widehat{E}_{AB}}=L_{\widehat{A}}$.
Since $\widehat{E}_{AB}\simeq E_{AB}\simeq\mP(T_{\mP(V_{3}^{*})})$,
we may consider $\widehat{E}_{AB}$ is the incident variety $\{[V_{1}],[V_{2}])\mid V_{1}\subset V_{2}\}\subset\mP(V_{3})\times\mP(V_{3}^{*})$.
Also, since $E_{AB}$ is the tautological divisor with respect to
$\sO_{\mP(V_{3}^{*})}\oplus T_{\mP(V_{3}^{*})}$, $\widehat{E}_{AB}|_{\widehat{E}_{AB}}={E}_{AB}|_{{E}_{AB}}$
is the tautological divisor with respect to $T_{\mP(V_{3}^{*})}$.
Hence $\widehat{E}_{AB}|_{\widehat{E}_{AB}}$ is the restriction of
a $(1,-2)$-divisor of $\mP(V_{3})\times\mP(V_{3}^{*})$. This is
equivalent to that $a=1$ and $b=2$. Now we note the equality $a\widehat{E}_{AB}|_{G}+bH_{\widehat{A}}|_{G}-cG|_{G}=bH_{\widehat{A}}|_{G}-cG|_{G}=L_{\widehat{A}}|_{G}$.
Recall that the conormal bundle of $s_{A}$ in $A$ is $\Omega_{\mP(V_{3}^{*})}$
as in the proof of Lemma \ref{cla:A} (5). Therefore $-G|_{G}$ is
the tautological divisor with respect to $T_{\mP(V_{3}^{*})}$, which
is the restriction of a $(1,-2)$-divisor of $\mP(V_{3})\times\mP(V_{3}^{*})$.
This is equivalent to that $b=2$ and $c=1$. 

\textbf{Step 2. $F_{A}=-(E_{AB}+2H_{A})$.}\\
 By Proposition \ref{cla:M}, we have $L_{\widetilde{\hcoY}}=\det\widetilde{\sQ}-M_{\widetilde{\hcoY}}$,
where $L_{\widetilde{\hcoY}}$ is the image of $L_{\widetilde{\hcoY}_{2}}$
on $\widetilde{\hcoY}$. Therefore, by (\ref{eq:canY2}), we have
$K_{\widetilde{\hcoY}}=-6\det\widetilde{\sQ}+4L_{\widetilde{\hcoY}}=-2\det\widetilde{\sQ}-4M_{\widetilde{\hcoY}}$.
Further, by Proposition \ref{cla:F} (2), we have $-2\det\widetilde{\sQ}-4M_{\widetilde{\hcoY}}=-10M_{\widetilde{\hcoY}}+2F_{\widetilde{\hcoY}}$.
Therefore, since the pull-back of $M_{\widetilde{\hcoY}}$ on $A$
is trivial, we have $F_{A}=-\det\widetilde{\sQ}_{A}$. Consequently,
we obtain $F_{A}=-(E_{AB}+2H_{A})$ by the equality in Step 1. 

\textbf{Step 3. $(\widetilde{\sS}_{L}^{*})_{A}\simeq\widetilde{\sQ}_{A}\simeq\sO_{A}\oplus\sV$.}\\
 We investigate the restriction of the universal exact sequence
(\ref{eq:univ}) on $A_{\hcoY_{3}}$. Let $\sS_{A_{\hcoY_{3}}}$ and
$\sQ_{A_{\hcoY_{3}}}$ be the restrictions of $\sS$ and $\sQ$, respectively.
Then we obtain \begin{equation}
0\to\sS_{A_{\hcoY_{3}}}\to\Lpi_{A_{\hcoY_{3}}}^{\;*}(T(-1)^{\wedge2}|_{\mP(V_{3})})\to\sQ_{A_{\hcoY_{3}}}\to0.\label{eq:resuniv}\end{equation}
Note the following isomorphisms, \begin{eqnarray}
\wedge^{2}(T(-1)|_{\mP(V_{3})})\simeq\wedge^{2}(T_{\mP(V_{3})}(-1)\oplus V/V_{3}\otimes\sO_{\mP(V_{3})})\simeq\label{eq:resT-1}\\
\sO_{\mP(V_{3})}(1)\oplus V/V_{3}\otimes T_{\mP(V_{3})}(-1)\oplus\wedge^{2}(V/V_{3})\otimes\sO_{\mP(V_{3})}.\nonumber \end{eqnarray}
 Let $([\bar{U}],[V_{1}])$ be a point in $A_{\hcoY_{3}}\subset\hcoY_{3}$
with $[U]=[\bar{U}\wedge V_{1}]$ (see Definition \ref{def:Y3}).
Since the morphism $A_{\hcoY_{3}}\simeq A_{\hcoY_{2}}\to A_{\widetilde{\hcoY}}$
is given by $([\bar{U}],[V_{1}])\mapsto[U]$ and $A_{\widetilde{\hcoY}}=\overline{\Gamma(V_{3},V_{4})}$,
we can assume the following form of $[\bar{U}]$ (see Lemma \ref{lem:inverse-rk1});\[
[\bar{U}]=[(V_{4}/V_{2})\wedge(V_{2}/V_{1}),a(V/V_{4})\wedge(V_{2}/V_{1})+b\wedge^{2}(V_{4}/V_{2})\wedge V_{1}],\]
with $V_{1}\subset V_{2}\subset V_{3}$. For simplicity, we write
$[\bar{U}]\in A_{\hcoY_{3}}$ with $[V_{1}]$ being implicit. 

By definition, we have $\sS_{A_{\hcoY_{3}}}\vert_{[\bar{U}]}=\bar{U}$
for the fiber of $\sS_{A_{\hcoY_{3}}}$ over $[\bar{U}]\in A_{\hcoY_{3}}$.
Similarly, the fiber of the pull-back $L_{A_{\hcoY_{3}}}$ of $\sO_{\mP(V_{3})}(1)\simeq\wedge^{2}T_{\mP(V_{3})}(-1)$
is given by $L_{A_{\hcoY_{3}}}\vert_{[\bar{U}]}=\wedge^{2}(V_{3}/V_{1})\simeq(V_{3}/V_{2})\wedge(V_{2}/V_{1})$.
Hence we see that $\sS_{A_{\hcoY_{3}}}$ contains $L_{A_{\hcoY_{3}}}$
as a direct summand. Let us write $\sS_{A_{\hcoY_{3}}}=\sS'_{A_{\hcoY_{3}}}\oplus L_{A_{\hcoY_{3}}}$
with a locally free sheaf $\sS'_{A_{\hcoY_{3}}}$ of rank two on $A_{\hcoY_{3}}$.
We note that $\sS_{A_{\hcoY_{3}}}$ is contained in $L_{A_{\hcoY_{3}}}\oplus(V/V_{3}\otimes\Lpi_{A_{\hcoY_{3}}}^{\;*}T_{\mP(V_{3})}(-1)$)
since this does not contain (the pull-back on $A_{\hcoY_{3}}$ of)
the factor $\wedge^{2}(V/V_{3})\otimes\sO_{\mP(V_{3})}$ in (\ref{eq:resT-1}).
Therefore, we have the following exact sequence from (\ref{eq:resuniv}):
\[
0\to\sS'{}_{A_{\hcoY_{3}}}\to V/V_{3}\otimes\Lpi_{A_{\hcoY_{3}}}^{\;*}T_{\mP(V_{3})}(-1)\to\sQ'_{A_{\hcoY_{3}}}\to0,\]
 where $\sQ_{A_{\hcoY_{3}}}\simeq\sQ'_{A_{\hcoY_{3}}}\oplus\sO_{A_{\hcoY_{3}}}$. 

Consider the pull-back $\sS'_{\hat{A}}$ on $\widehat{A}$ of $\sS'_{A_{\hcoY_{3}}}$.
This contains a subbundle of rank one whose fiber at a point over
$[\bar{U}]$ is isomorphic to $V_{4}/V_{3}\otimes V_{2}/V_{1}$. Since
$V_{3}$ and $V_{4}$ are fixed, the vector space $V_{4}/V_{3}$ is
the fiber of the trivial bundle on $\widehat{A}$. Also, since $V_{1}$
is the fiber of $-L_{\widehat{A}}$ and $V_{2}$ is the fiber of the
pull-back of $\Omega_{\mP(V_{3}^{*})}(1)$, we see that $V_{2}/V_{1}$
is the fiber of $\sO_{\widehat{A}}(-H_{\widehat{A}}+L_{\widehat{A}})$
by taking the determinants. Therefore $\sS'_{\hat{A}}(-L_{\widehat{A}})$
is presented as an extension, \[
0\to\sO(-H_{\widehat{A}})\to\sS'_{\hat{A}}(-L_{\widehat{A}})\to\sO(-H_{\widehat{A}}-\widehat{E}_{AB})\to0,\]
 where the quotient is determined by taking determinants. Since $\sO(-H_{\widehat{A}})$,
$\sS'_{\hat{A}}(-L_{\widehat{A}})$, and $\sO(-H_{\widehat{A}}-\widehat{E}_{AB})$
are the pull-backs of locally free sheaves on $A$, we have \[
0\to\sO_{A}(-H_{{A}})\to(\widetilde{\sS}'_{L})_{{A}}\to\sO_{A}(-H_{{A}}-E_{AB})\to0,\]
 where $(\widetilde{\sS}'_{L})_{{A}}$ is the sheaf which represents
$\sS'_{\hat{A}}(-L_{\widehat{A}})$ by the pull-back. The sequence
does not split since $H_{{A}}$ is not a locally free sheaf on $A_{\widetilde{\hcoY}}$
while $(\widetilde{\sS}'_{L})_{{A}}$ is. Since\begin{align*}
 & \Ext^{1}(\sO_{A}(-H_{{A}}-E_{AB}),\sO_{A}(-H_{{A}}))\simeq\\
 & H^{1}(A,\sO_{A}(E_{AB}))\simeq H^{1}(\mP(V_{3}^{*}),\sO_{\mP(V_{3}^{*})}\oplus\Omega_{\mP(V_{3}^{*})}^{1})\simeq\mC\end{align*}
 we see that $(\widetilde{\sS}_{L}^{'*})_{A}\simeq\sV$, and hence
$(\widetilde{\sS}_{L}^{*})_{A}\simeq\sV\oplus\sO_{A}$, with a locally
free sheaf $\sV$ as described in (1).

Let ${\sQ}'_{\widehat{A}}$ be the pull-back on $\widehat{A}$ of
${\sQ}'_{{A_{\hcoY_{3}}}}$. Taking a basis of $V$ so that $V_{1}=\langle\bm{e}_{5}\rangle$,$V_{2}=\langle\bm{e}_{1},\bm{e}_{5}\rangle$,$V_{3}=\langle\bm{e}_{1},\bm{e}_{2},\bm{e}_{5}\rangle$
and $V_{4}=\langle\bm{e}_{1},\bm{e}_{2},\bm{e}_{3},\bm{e}_{5}\rangle$,
we see that there is a surjective map from ${\sQ}'_{\widehat{A}}$
to the invertible sheaf whose fiber at a point over $[\bar{U}]$ is
isomorphic to $V_{3}/V_{2}\otimes V/V_{4}$. We identify this invertible
sheaf with $\sO_{\widehat{A}}(H_{\widehat{A}})$. Therefore $\sQ'_{\hat{A}}$
is presented as an extension: \[
0\to\sO(H_{\widehat{A}}+\widehat{E}_{AB})\to{\sQ}'_{\widehat{A}}\to\sO(H_{\widehat{A}})\to0,\]
 where the kernel is determined by taking determinants. Therefore
we see that ${\sQ}'_{\widehat{A}}$ is also the pull-back of $\sV$
and $\widetilde{\sQ}_{A}\simeq\sV\oplus\sO_{A}$ as we have determined
$(\widetilde{\sS}_{L}^{*})_{{A}}$.

\textbf{Step 4. $\widetilde{\sT}_{A}\simeq\sO_{A}^{\oplus2}\oplus\sV$.}\\
 By Lemma \ref{cla:A} (3), $\Prt_{\rho}|_{A_{\hcoY_{3}}}\simeq\mP(T_{\mP(V_{3})}(-1))$
and this lifts to $\hcoY_{2}$ isomorphically. Therefore, restricting
(\ref{eq:exact-seq-T2}) to ${A_{{\hcoY}_{2}}}$, we obtain \begin{equation}
0\to\sT_{A_{\hcoY_{2}}}^{*}\to\Lpi_{{A_{\hcoY_{2}}}}^{\;*}(\Omega_{\mP(V_{3})}^{1}(1)\oplus\sO_{\mP(V_{3})}^{\oplus2})\to\sO_{\mP(T_{\mP(V_{3})}(-1))}(1)\to0,\label{eq:E_a'res}\end{equation}
 where we set $\sT_{A_{\hcoY_{2}}}^{*}=\sT_{2}^{*}\vert_{A_{\hcoY_{2}}},\Lpi_{A_{\hcoY_{2}}}=\Lpi_{\hcoY_{2}}\vert_{A_{\hcoY_{2}}}$
and used $\sR_{2}/\sR_{1}=\sO_{\mP(T(-1))}(-1)$. Since $H^{0}(\sO_{\mP(T_{\mP(V_{3})}(-1))}(1))=H^{0}(\Omega_{\mP(V_{3})}^{1}(1))=0$,
we have \[
\sT_{A_{\hcoY_{2}}}^{*}\simeq\sO_{A_{\hcoY_{2}}}^{\oplus2}\oplus\sV',\]
 where $\sV'$ is the kernel of the map \begin{equation}
\Lpi_{{A_{\hcoY_{2}}}}^{\;*}\Omega_{\mP(V_{3})}^{1}(1)\to\sO_{\mP(T_{\mP(V_{3})}(-1))}(1).\label{eq:V'}\end{equation}
Let us consider a point $[\bar{U}]\in A_{\hcoY_{3}}$. We note that
the fiber of $\Lpi_{{A_{\hcoY_{3}}}}^{\;*}\Omega_{\mP(V_{3})}^{1}(1)$
at $[\bar{U}]\in A_{\hcoY_{3}}$ is isomorphic to $(V_{3}/V_{1})^{*}$.
In a similar way to the arguments after (\ref{eq:resT-1}), the vector
space $(V_{3}/V_{2})^{*}$ can be considered to be the fiber of $\sO_{\widehat{A}}(-H_{\widehat{A}})$
at a point over $[V_{2}]\in\mP(V_{3})$. Therefore we have a natural
injection \[
\sO_{\widehat{A}}(-H_{\widehat{A}})\hookrightarrow\Lpi_{\widehat{A}}^{\;*}\Omega_{\mP(V_{3})}^{1}(1),\]
 where the cokernel $\sK_{1}$ is an invertible sheaf and $\Lpi_{\widehat{A}}$
is the naturally induced map $\widehat{A}\to\mP(V_{3})$. By taking
the determinants, we see that $\sK_{1}=\sO_{\widehat{A}}(H_{\widehat{A}}-L_{\widehat{A}})$.
We show that the composite morphism $\sO_{\widehat{A}}(-H_{\widehat{A}})\to\sO_{\mP(T_{\mP(V_{3})}(-1))}(1)$
of the injection above with the pull-back of (\ref{eq:V'}) is zero.
Here we note that $\mP(T_{\mP(V_{3})}(-1))$ lifts isomorphically
on $\widehat{A}$. Indeed, $H_{\mP(\Omega_{\mP(V_{3})}^{1}(1))}$
is nothing but the pull-back of $H_{A}$ by the map $\widehat{A}\to A$.
Since $H_{\mP(T_{\mP(V_{3})}(-1))}=H_{\mP(\Omega_{\mP(V_{3})}^{1}(1))}-L$,
where $L$ is the pull-back of $\sO_{\mP(V_{3})}(1)$ to $\mP(T_{\mP(V_{3})}(-1))$,
we have only to show $H^{0}(2H_{\mP(\Omega_{\mP(V_{3})}^{1}(1))}-L)=0$,
which follows from the Bott theorem \ref{thm:Bott} by noting $H^{0}(2H_{\mP(\Omega_{\mP(V_{3})}^{1}(1))}-L)\simeq H^{0}(\ft{S}^{2}(T_{\mP(V_{3})}(-1))\otimes\sO_{{A_{\hcoY_{2}}}}(-1))$.
Therefore we have an injection $\sO_{\widehat{A}}(-H_{\widehat{A}})\hookrightarrow\sV'_{\widehat{A}}$,
where $\sV'_{\widehat{A}}$ is the pull-back of $\sV'$ on $\widehat{A}$
and the cokernel $\sK_{2}$ has the following expression as an extension:
\[
0\to\sK_{2}\to\sO_{\widehat{A}}(H_{\widehat{A}}-L_{\widehat{A}})\to H_{\mP(T_{\mP(V_{3})}(-1))}\to0.\]
 Taking $\sE xt^{\bullet}(-,\sO_{\widehat{A}})$ of this exact sequence,
we see that $\sK_{2}$ is also an invertible sheaf by \cite[III, Ex 6.6]{Ha}.
By taking the determinants, we see that $\sK_{2}=\sO_{\widehat{A}}(H_{\widehat{A}}-L_{\widehat{A}}-F_{\rho}|_{\widehat{A}})$,
where $F_{\rho}|_{\widehat{A}}$ is the pull-back of $F_{\rho}$.
Since $M_{\hcoY_{2}}|_{A_{\hcoY_{2}}}=0$, we have $\det\sQ_{\widehat{A}}=L_{\widehat{A}}+F_{\rho}|_{\widehat{A}}$
by Proposition \ref{cla:M}, where $\sQ_{\widehat{A}}$ is the pull-back
of $\sQ$. Therefore, $\sK_{2}=\sO_{\widehat{A}}(H_{\widehat{A}}-\det\sQ_{\widehat{A}})=\sO_{\widehat{A}}(-H_{\widehat{A}}-\widehat{E}_{AB})$,
where the second equality follows from Step 1. Therefore $\sV'_{\widehat{A}}$
fits into the following expression as an extension: \[
0\to\sO_{\widehat{A}}(-H_{\widehat{A}})\to\sV'_{\widehat{A}}\to\sO_{\widehat{A}}(-H_{\widehat{A}}-\widehat{E}_{AB})\to0.\]
 Consequently, we have $\widetilde{\sT}_{A}\simeq\sO_{A}^{\oplus2}\oplus\sV$
as we have determined $(\widetilde{\sS}_{L}^{*})_{A}$.

\textbf{Step 5. $F_{B}$, $(\widetilde{\sS}_{L}^{*})_{B}$, $\widetilde{\sQ}_{B}$
and $\widetilde{\sT}_{B}$.}\\
 By Lemma \ref{cla:A} (1), the image of $B$ on $F_{\widetilde{\hcoY}}$
is the $\mP(V_{3})$ contained in $A_{\widetilde{\hcoY}}$. Therefore,
$F_{B}$, $(\widetilde{\sS}_{L}^{*})_{B}$, $\widetilde{\sQ}_{B}$
and $\widetilde{\sT}_{B}$, respectively, are the pull-backs of the
restrictions of $F_{\widetilde{\hcoY}}$, $\widetilde{\sQ}$, and
$\widetilde{\sT}$ to $\mP(V_{3})$. Since $F_{A}|_{E_{AB}}\simeq-(E_{AB}+2H_{A})|_{E_{AB}}$
by Step 2, and this is the pull-back of $\sO_{\mP(V_{3})}(-1)$, we
have $F_{B}=p_{B}^{*}\sO_{\mP(V_{3})}(-1)$. Also, since $\widetilde{\sT}_{A}\simeq\sO_{A}\oplus(\widetilde{\sS}_{L}^{*})_{A}\simeq\sO_{A}\oplus\widetilde{\sQ}_{A}$
as above, we have $\widetilde{\sT}_{B}\simeq\sO_{B}\oplus(\widetilde{\sS}_{L}^{*})_{B}\simeq\sO_{B}\oplus\widetilde{\sQ}_{B}$.
Thus we have only to determine $\widetilde{\sT}_{B}$. Recall that
$\mP(V_{3})$ in $A_{\widetilde{\hcoY}}$ consists of the points of
the form $[U]=[\wedge^{2}(V_{4}/V_{1})\wedge V_{1}]\in\Prt_{\sigma}$
with $V_{1}\subset V_{3}$. Therefore $\mP(V_{3})$ is disjoint from
$G_{\rho}$ (Definition \ref{defn:Grho}, Fig. 4), and then, by (\ref{eq:Eb}),
we have $\widetilde{\sT}_{B}\simeq p_{B}^{*}(T(-1)|_{\mP(V_{3})})\simeq\sO_{B}^{\oplus2}\oplus p_{B}^{*}(T_{\mP(V_{3})}(-1))$.
\end{proof}


\newpage{}

\section{The Lefschetz collection in $\sD^{b}(\widetilde{\hcoY})$}

\label{section:comp}

Using the sheaves $\widetilde{\sS}_{L}^{*},\widetilde{\sQ},\,\widetilde{\sT}$
and divisors introduced in Section \ref{sec:SheavesDef}, we describe
a Lefschetz collection in $\sD^{b}(\widetilde{\hcoY})$, which shows
an interesting duality between the (dual) Lefschetz collection obtained
in Theorem \ref{thm:Gvan1}. 

$\,$

\subsection{Results}

Our results on the sheaves $\widetilde{\sS}_{L}^{*},\widetilde{\sQ},\,\widetilde{\sT}$
and $\sO_{\widetilde{\hcoY}}$ are summarized in the following theorem. 

\begin{thm} \label{thm:Gvan} $\;$

\begin{myitem2}\item[\rm(1)] Let \[
(\sE_{3},\sE_{2},\sE_{1a},\sE_{1b})=((\widetilde{\sS}_{L}^{*})^{*},\widetilde{\sT}^{*},\sO_{\widetilde{\hcoY}},\widetilde{\sQ}^{*}(M_{\widetilde{\hcoY}}))\]
 be an ordered collection of sheaves on $\widetilde{\hcoY}$. Then
$(\widetilde{\sB}_{i})_{1\leq i\leq4}:=(\sE_{3},\sE_{2},\sE_{1a},\sE_{1b})$
is a strongly exceptional collection, namely, it satisfies \[
H^{\bullet}(\widetilde{\sB}_{i}^{*}\otimes\widetilde{\sB}_{j})=0\text{ for }1\leq i,j\leq4\text{ and }\bullet>0\]
 and $H^{0}(\widetilde{\sB}_{i}^{*}\otimes\widetilde{\sB}_{j})=0\;(i>j)$,
$H^{0}(\widetilde{\sB}_{i}^{*}\otimes\widetilde{\sB}_{i})=\mC\;(1\leq i\leq4)$. 

\item[\rm (2)] For $i<j$, $H^{0}(\widetilde{\sB}_{i}^{*}\otimes\widetilde{\sB}_{j})$
are given by \[
\begin{aligned} & H^{0}(\sE_{3}^{*}\otimes\sE_{2})\simeq V,\, H^{0}(\sE_{3}^{*}\otimes\sE_{1a})\simeq\wedge^{2}V,\, H^{0}(\sE_{3}^{*}\otimes\sE_{1b})\simeq\ft{S}^{2}V,\\
 & H^{0}(\sE_{2}^{*}\otimes\sE_{1a})\simeq V,\, H^{0}(\sE_{2}^{*}\otimes\sE_{1b})\simeq V,\;\; H^{0}(\sE_{1a}^{*}\otimes\sE_{1b})\simeq0,\end{aligned}
\]
 and may be summarized in the following diagram: \begin{equation}
\xyQuiverII\label{eq:quiver2}\end{equation}

\item[\rm (3)] Set \[
\sD_{\widetilde{\hcoY}}:=\langle\sE_{3},\sE_{2},\sE_{1a},\sE_{1b}\rangle\subset\sD^{b}(\widetilde{\hcoY}).\]
 Then \[
\sD_{\widetilde{\hcoY}},\sD_{\widetilde{\hcoY}}(1),\dots,\sD_{\widetilde{\hcoY}}(9)\]
 is a Lefschetz collection, namely, for $1\leq i,j\leq4$ and $\bullet>0$
it holds that \[
H^{\bullet}(\widetilde{\sB}_{i}^{*}\otimes\widetilde{\sB}_{j}(-t))=0\;\;(1\leq t\leq9),\]
 where $(-t)$ represents the twist by the sheaf $\sO_{\widetilde{\hcoY}}(-tM_{\widetilde{\hcoY}})$. 

\end{myitem2}\end{thm}

\vspace{0.5cm}
 The rest of this section is devoted to our proof of Theorem \ref{thm:Gvan},
where we compute the cohomology groups $H^{\bullet}(\widetilde{\sB}_{i}^{*}\otimes\widetilde{\sB}_{j}(-t))$
$(0\leq t\leq9)$. Our strategy is to reduce the computations of cohomology
groups on $\widetilde{\hcoY}$ to those on $\hcoY_{3}$ and use the
Bott Theorem \ref{thm:Bott} for the $\mathrm{G}(3,6)$-bundle $\hcoY_{3}\to\mP(V)$.
This idea works for small values of $t$ as we formulate in Proposition
\ref{cla:cohY2-Y3} below. The computations will be completed in the
next subsection. 

\vspace{0.2cm}
 \begin{lem} $K_{\widetilde{\hcoY}}=-10M_{\widetilde{\hcoY}}+2F_{\widetilde{\hcoY}}$
\end{lem} \begin{proof} We have $K_{\hcoY}=-10M_{\hcoY}$ by Proposition
\ref{cla:ZY}. Then from Proposition \ref{cla:F} (2), $K_{\widetilde{\hcoY}}=\Lrho_{\widetilde{\hcoY}}^{\;*}K_{\hcoY}+2F_{\widetilde{\hcoY}}=-10M_{\widetilde{\hcoY}}+2F_{\widetilde{\hcoY}}$.
\end{proof}

\vspace{0.3cm}
 Let us introduce \[
\widetilde{\sC}_{ij}=\widetilde{\sB}_{i}^{*}\otimes\widetilde{\sB}_{j}\;\;(1\leq i,j\leq4)\]
 for the ordered collection $(\widetilde{\sB}_{i})_{1\leq i\leq4}$.

\vspace{0.2cm}
 \begin{lem} \label{cla:key} For $\widetilde{\sC}=\widetilde{\sC}_{ij}$
as above, it holds \[
H^{\bullet}(\widetilde{\hcoY},\widetilde{\sC}(-t))\simeq H^{13-\bullet}(\widetilde{\hcoY},\widetilde{\sC}^{*}(t-10))\]
 for any integer $t$. \end{lem}

\begin{proof} Note that each of the cohomology groups $H^{\bullet}(\widetilde{\hcoY},\widetilde{\sC}(-t))$
is Serre dual to \[
H^{13-\bullet}(\widetilde{\hcoY},\widetilde{\sC}^{*}((t-10)M_{\widetilde{\hcoY}}+2F_{\widetilde{\hcoY}})).\]
 Considering the exact sequence \[
0\to\widetilde{\sC}^{*}((t-10)M_{\widetilde{\hcoY}}+(i-1)F_{\widetilde{\hcoY}})\to\widetilde{\sC}^{*}((t-10)M_{\widetilde{\hcoY}}+iF_{\widetilde{\hcoY}})\to\widetilde{\sC}^{*}((t-10)M_{\widetilde{\hcoY}}+iF_{\widetilde{\hcoY}})|_{F_{\widetilde{\hcoY}}}\to0,\]
 it suffices to show \begin{equation}
H^{13-\bullet}(F_{\widetilde{\hcoY}},\sA_{i})=0\text{\text{ for }}\ensuremath{i=1,2}\text{\text{ and any}}\;\ensuremath{\bullet},\label{eq:vanF}\end{equation}
 where we set \[
\sA_{i}:=\widetilde{\sC}^{*}((t-10)M_{\widetilde{\hcoY}}+iF_{\widetilde{\hcoY}})|_{F_{\widetilde{\hcoY}}}.\]
 Recall the diagrams (\ref{eq:house}) and (\ref{eq:flatten}). Since
the morphism $\widehat{F}\to F_{\widetilde{\hcoY}}$ is finite, and
$\widehat{F}$ has only rational singularities by its construction,
the vanishing follows from the vanishings of the cohomology groups
of the pull-backs of $\sA_{i}$ on ${F}^{(3)}$. Since the morphism
${F}^{(3)}\to\widehat{G}'$ is flat by Proposition~\ref{cla:union},
we have only to show the vanishing along its fibers. By the upper
semi-continuity of cohomology groups on fibers, it suffices to prove
the vanishing on fibers over points of the diagonal subset of $\widehat{G}'$.
By Proposition~\ref{cla:union}, such fibers are of the form $A\cup B$.
Note that the restriction of the pull-back of $M_{\widetilde{\hcoY}}$
to the fibers is trivial. Therefore, by Lemma \ref{cla:ST}, it suffices
to show the vanishing of the following cohomology groups: \[
H^{\bullet}(A\cup B,\sC_{A\cup B}^{*}(iF_{A\cup B}))\,(i=1,2),\]
 where $\sC_{A\cup B}$ and $F_{A\cup B}$ are the pull-backs of $\widetilde{\sC}$
and $F_{\widetilde{\hcoY}}$ to $A\cup B$, respectively. Tensoring
$\sC_{A\cup B}^{*}(iF_{A\cup B})$ with the Mayor-Vietris sequence,
we have \begin{equation}
0\to\sC_{A\cup B}^{*}(iF_{A\cup B})\to\sC_{A}^{*}(iF_{A})\oplus\sC_{B}^{*}(iF_{B})\to\sC_{A\cap B}^{*}(iF_{A\cap B})\to0,\label{eq:MV}\end{equation}
 where $\sC_{A}$, $\sC_{B}$, and $\sC_{A\cap B}$ are the restrictions
of $\sC_{A\cup B}$ to $A$, $B$ and $A\cap B$ respectively, and
$F_{A\cap B}$ is the restriction of $F_{A\cup B}$ to $A\cap B$.
By Lemma \ref{cla:ST} (1), we can easily show the vanishing of $H^{\bullet}(A,\sC_{A}^{*}(iF_{A}))$.
Moreover, by Lemma \ref{cla:ST} (2), we easily see that the restriction
maps $H^{\bullet}(B,\sC_{B}^{*}(iF_{B}))\to H^{\bullet}(A\cap B,\sC_{A\cap B}^{*}(iF_{A\cap B}))$
are isomorphisms. Therefore we have the vanishing of $H^{\bullet}(A\cup B,\sC_{A\cup B}^{*}(iF_{A\cup B}))$.
\end{proof}

\vspace{0.3cm}
 \begin{rem} Form the above lemma, we conjecture that $\sD_{\widetilde{\hcoY}},\sD_{\widetilde{\hcoY}}(1),\dots,\sD_{\widetilde{\hcoY}}(9)$
generates a strongly crepant categorical resolution (cf.~\cite{Lef}).
\hfill{}{[}{]} \end{rem}

\vspace{0.3cm}
 Let us introduce a sequence of sheaves on $\hcoY_{2}$ \[
(\sB_{i})_{1\leq i\leq4}=(\Lrho_{\hcoY_{2}}^{\;*}\sS(-L_{\hcoY_{2}}),\sT_{2}^{*},\sO_{\hcoY_{2}},\Lrho_{\hcoY_{2}}^{\;*}\sQ^{*}),\]
 and define $\sC_{ij}=\sB_{i}^{*}\otimes\sB_{j}$ for $1\leq i,j\leq4$.
Note that $\sB_{i}$ have its corresponding form to $\widetilde{\sB}_{i}$
for $i=1,2,3$, but $\sB_{4}$ is defined by removing the twist of
$M_{\widetilde{\hcoY}}$ from $\widetilde{\sB}_{4}$.

Let $\widetilde{\sD}$ be a locally free sheaf on $\widetilde{\hcoY}$
and $\sD:=\tLrho_{\hcoY_{2}}^{\;*}\widetilde{\sD}$. Since $\tLrho_{\hcoY_{2}}\colon\hcoY_{2}\to\widetilde{\hcoY}$
is a blow-up of a smooth variety, it holds that \begin{equation}
H^{\bullet}(\widetilde{\hcoY},\widetilde{\sD}(-t))\simeq H^{\bullet}(\hcoY_{2},{\sD}(-t)),\label{eq:rhs}\end{equation}
 where $(-t)$ on the right hand side represents the twist by $\sO_{\widetilde{\hcoY_{2}}}(-tM_{\widetilde{\hcoY_{2}}})$.
By Propositions \ref{def:SQT1}, \ref{def:tidelT} and Lemma \ref{cla:key},
for our cohomology calculations, it suffices to know \begin{equation}
\begin{aligned} & H^{\bullet}(\widetilde{\hcoY},\widetilde{\sC}_{i4}(-t))=H^{\bullet}(\hcoY_{2},{\sC}_{i4}(-t+1))\;\;(t=0,1,...,6)\\
 & H^{\bullet}(\widetilde{\hcoY},\widetilde{\sC}_{4j}(-t))=H^{\bullet}(\hcoY_{2},{\sC}_{4j}(-t-1))\;\;(t=0,1,...,4)\end{aligned}
\label{eq:cohY2-1}\end{equation}
 for $1\leq i,j\leq3$ and \begin{equation}
H^{\bullet}(\widetilde{\hcoY},\widetilde{\sC}_{ij}(-t))=H^{\bullet}(\hcoY_{2},{\sC}_{ij}(-t))\;\;(t=0,1,...,5)\label{eq:cohY2-2}\end{equation}
 for $1\leq i,j\leq3$ or $i=j=4$.

\vspace{0.3cm}
 \begin{lem} \label{cla:T} For the computations of $(\ref{eq:cohY2-1})$
and $(\ref{eq:cohY2-2})$, we may replace $\sT_{2}$ by $\Lpi_{\hcoY_{2}}^{\;*}T(-1)$
except in one case $\sC_{24}(1)=\sT_{2}\otimes\Lrho_{\hcoY_{2}}^{\;*}\sQ^{*}(M_{\hcoY_{2}})$.
From now on we may use the following relations: \[
\begin{aligned} & \sC_{42}(-t-1)=(\Lrho_{\hcoY_{2}}^{\;*}\sQ\otimes\Lpi_{\hcoY_{2}}^{\;*}\Omega(1))(-t-1),\;\\
 & \sC_{24}(-t+1)=(\Lpi_{\hcoY_{2}}^{\;*}T(-1)\otimes\Lrho_{\hcoY_{2}}^{\;*}\sQ^{*})(-t+1)\;\;\;(t\not=0),\;\\
 & \sC_{i2}(-t)=(\sB_{i}^{*}\otimes\Lpi_{\hcoY_{2}}^{\;*}\Omega(1))(-t)\;\;\;(i=1,3),\\
 & \sC_{2j}(-t)=(\Lpi_{\hcoY_{2}}^{\;*}T(-1)\otimes\sB_{j})(-t)\;\;\;(j=1,3),\\
 & \sC_{22}(-t)=(\Lpi_{\hcoY_{2}}^{\;*}T(-1)\otimes\Lpi_{\hcoY_{2}}^{\;*}\Omega(1))(-t).\end{aligned}
\]
 \end{lem}

\begin{proof} First we consider $\sC_{i2}(-t)$. By the exact sequence
(\ref{eq:exact-seq-T2}), we have \[
0\to\sT_{2}^{*}\otimes{\sD}(-t)\to\Lpi_{\hcoY_{2}}^{*}\Omega(1)\otimes{\sD}(-t)\to(\Lrho_{\hcoY_{2}}|_{F_{\rho}})^{*}(\sR_{2}/\sR_{1})^{*}\otimes{\sD}(-t)|_{F_{\rho}}\to0\]
 for a locally free sheaf ${\sD}$ on $\hcoY_{2}$. For ${\sD}=\sO_{\hcoY_{2}}$,
$\Lrho_{\hcoY_{2}}^{\;*}\sQ$, $\Lrho_{\hcoY_{2}}^{\;*}\sS^{*}(L_{\hcoY_{2}})$
or $\sT_{2}$, it holds that \[
\text{\ensuremath{H^{\bullet}(F_{\rho},(\Lrho_{\hcoY_{2}}|_{F_{\rho}})^{*}(\sR_{2}/\sR_{1})^{*}\otimes{\sD}(-t)|_{F_{\rho}})=0} for any \ensuremath{t}}\]
 by the Leray spectral sequence for $\tLrho_{\hcoY_{2}}|_{F_{\rho}}\colon F_{\rho}\to G_{\rho}$
since $\tLrho_{\hcoY_{2}}|_{F_{\rho}}$ is a $\mP^{1}$-bundle and
the restriction of $(\Lrho_{\hcoY_{2}}|_{F_{\rho}})^{*}(\sR_{2}/\sR_{1})^{*}\otimes\sD(-t)|_{F_{\rho}}$
to its fiber is a direct sum of $\sO_{\mP^{1}}(-1)$ (cf. the proof
of Lemma \ref{lem:T2-sum}). Therefore we have \[
H^{\bullet}({\hcoY_{2}},\sT_{2}^{*}\otimes{\sD}(-t))\simeq H^{\bullet}({\hcoY_{2}},\Lpi_{\hcoY_{2}}^{\;*}\Omega(1)\otimes{\sD}(-t))\]
 for ${\sD}=\sO_{\hcoY_{2}}$, $\Lrho_{\hcoY_{2}}^{\;*}\sQ$, $\Lrho_{\hcoY_{2}}^{\;*}\sS^{*}(L_{\hcoY_{2}})$
or $\sT_{2}$ and for any $t$.

Next we consider $\sC_{2j}(-t)$ except $\sC_{24}(1)$. By the exact
sequence (\ref{eq:Eb}), we have \[
0\to\Lpi_{\hcoY_{2}}^{\;*}T(-1)\otimes{\sD}(-t)\to\sT_{2}\otimes{\sD}(-t)\to(\Lrho_{\hcoY_{2}}|_{F_{\rho}})^{*}(\sR_{2}/\sR_{1})\otimes{\sD}(-t{M}_{\hcoY_{2}}+F_{\rho})|_{F_{\rho}}\to0\]
 for a locally free sheaf $\sD$ on $\hcoY_{2}$. Set ${\sD}=\sO_{\hcoY_{2}}$,
$\Lrho_{\hcoY_{2}}^{\;*}\sQ^{*}$, $\Lrho_{\hcoY_{2}}^{\;*}\sS(-L_{\hcoY_{2}})$,
or $\Lpi_{\hcoY_{2}}^{\;*}\Omega(1)$. We show the vanishing of \begin{equation}
H^{\bullet}(F_{\rho},(\Lrho_{\hcoY_{2}}|_{F_{\rho}})^{*}(\sR_{2}/\sR_{1})\otimes{\sD}(-t{M}_{\hcoY_{2}}+F_{\rho})|_{F_{\rho}}).\label{eq:Frho}\end{equation}
 By Proposition \ref{cla:mislead}, we have only to consider the case
where ${\sD}=\sO_{\hcoY_{2}}$, $\Lrho_{\hcoY_{2}}^{\;*}\sQ^{*}$
or $\Lpi_{\hcoY_{2}}^{\;*}\Omega(1)$. For $0\leq t\leq4$, the vanishing
of (\ref{eq:Frho}) follows from the Leray spectral sequence for ${\Lrho}_{\hcoY_{2}}|_{F_{\rho}}\colon F_{\rho}\to\Prt_{\rho}$
since ${\Lrho}_{\hcoY_{2}}|_{F_{\rho}}$ is a $\mP^{5}$-bundle and
the restriction of $(\Lrho_{\hcoY_{2}}|_{F_{\rho}})^{*}(\sR_{2}/\sR_{1})\otimes{\sD}(-t{M}_{\hcoY_{2}}+F_{\rho})|_{F_{\rho}}$
to its fiber is a direct sum of $\sO_{\mP^{5}}(-(t+1))$ by Proposition
\ref{cla:M}. Therefore we may assume that $t=5$ from now on. Each
of the cohomology groups (\ref{eq:Frho}) is Serre dual to \begin{equation}
H^{12-\bullet}(F_{\rho},(\Lrho_{\hcoY_{2}}|_{F_{\rho}})^{*}(\sR_{2}/\sR_{1})^{*}\otimes{\sD}^{*}(\rho_{\hcoY_{2}}^{*}(-\det\sQ-L_{\hcoY_{3}}))|_{F_{\rho}})\label{eq:dammy}\end{equation}
 by (\ref{eq:canY2}) and Proposition \ref{cla:M}. Since $\Lrho_{\hcoY_{2}}$
is the blow-up of a smooth variety and ${\sD}$ is the pull-back of
a locally free sheaf $\overline{\sD}$ on $\hcoY_{3}$, each of the
cohomology groups (\ref{eq:dammy}) is isomorphic to \begin{equation}
H^{12-\bullet}(\Prt_{\rho},(\sR_{2}/\sR_{1})^{*}\otimes{\overline{\sD}}^{*}(-\det\sQ-L_{\hcoY_{3}})|_{\Prt_{\rho}}).\label{eq:dammy2}\end{equation}
 Using Proposition \ref{cla:det} (1) and (2), we can write \[
(\sR_{2}/\sR_{1})^{*}\otimes{\overline{\sD}}^{*}(-\det\sQ-L_{\hcoY_{3}})|_{\Prt_{\rho}}=\begin{cases}
\sO_{\Prt_{\rho}}(-H_{\Prt_{\rho}}\text{-- }3L_{\Prt_{\rho}})\text{ for }\overline{\sD}=\sO_{\hcoY_{3}}\\
(\sR_{V}/\sR_{2})^{*}(-2L_{\Prt_{\rho}})\text{ for }\overline{\sD}=\sQ^{*}\\
(\sR_{V}/\sR_{1})(-H_{\Prt_{\rho}}\text{-- }3L_{\Prt_{\rho}})\text{ for }\overline{\sD}=\Lpi_{\hcoY_{3}}^{\;*}\Omega(1),\end{cases}\]
where we use $(\sR_{2}/\sR_{1})^{*}=\sO_{\mP(T(-1))}(1)=H_{\Prt_{\rho}}$.
All the cohomology groups of the restrictions of these sheaves to
a fiber of $\Lpi_{\hcoY_{3}}|_{\Prt_{\rho}}\colon\Prt_{\rho}=F(1,2,V)\to\mP(V)$
vanish. Hence, by the Leray spectral sequence for $\Lpi_{\hcoY_{3}}|_{\Prt_{\rho}}$,
the cohomology groups (\ref{eq:dammy2})\textcolor{red}{{} }vanish,
too. \end{proof}

By the following simple {lemma,} we can reduce most of the computations
of cohomology groups to those on $\hcoY_{3}$:

\begin{lem} \label{cla:van2} $\;$

\noindent{\rm (1)} $R^{q}\Lrho_{\hcoY_{2}*}\sO_{\hcoY_{2}}(tF_{\rho})=0$
for any $t\leq5$ and $q>0$. 

\noindent{\rm (2)} $\Lrho_{\hcoY_{2}*}\sO_{\hcoY_{2}}(tF_{\rho})=\sO_{\hcoY_{3}}$
for $t\geq0$. 

\end{lem}

\begin{proof} Part (1) follows from the relative Kodaira vanishing
theorem since $tF_{\rho}-K_{\hcoY_{2}}\equiv_{\hcoY_{3}}(t-5)F_{\rho}$
is $\Lrho_{\hcoY_{2}}$-nef and $\Lrho_{\hcoY_{2}}$-big if $t\leq5$.

Part (2) is well-known. \end{proof}

\vspace{0.3cm}
 Now define an ordered collection of sheaves on $\hcoY_{3}$, \[
(\overline{\sB}_{i})_{1\leq i\leq4}=\big(\sS(-L_{\hcoY_{3}}),\Lpi_{\hcoY_{3}}^{\;*}\Omega(1),\sO_{\hcoY_{3}},\sQ^{*}\big),\]
 and set $\overline{\sC}_{ij}=\overline{\sB}_{i}^{*}\otimes\overline{\sB}_{j}$.

\begin{prop} \label{cla:cohY2-Y3} The cohomology groups on $\hcoY_{2}$
in the r.h.s. of $(\ref{eq:cohY2-1})$ and $(\ref{eq:cohY2-2})$ can
be evaluated by using \begin{equation}
H^{\bullet}(\hcoY_{2},\sC_{ij}(-t))\simeq H^{\bullet}(\hcoY_{3},\overline{\sC}_{ij}(-t\det\sQ+tL_{\hcoY_{3}}))\;\;(t=0,1,...,5)\label{eq:lastvan}\end{equation}
 for $1\leq i,j\leq4$ except the cases of $\sC_{i4}(1)\;(1\leq i\leq3)$.
$($We read $t-1$ and $t+1$ in the r.h.s. of $(\ref{eq:cohY2-1})$
as $t$ here.$)$ \end{prop} 

\begin{proof} Note that we may assume that the relation $\sC_{ij}=\Lrho_{\hcoY_{2}}^{\;*}\overline{\sC}_{ij}$
holds for $\sC_{ij}(-t)$ except $\sC_{24}(1)$ by Lemma \ref{cla:T}.
Then by the Leray spectral sequence for $\Lrho_{\hcoY_{2}}$, and
Proposition \ref{cla:M} and Lemma \ref{cla:van2}, we have the claimed
isomorphisms for the range of $t$ in (\ref{eq:cohY2-1}) and (\ref{eq:cohY2-2})
except the cases $\sC_{i4}(1)\;(1\leq i\leq3)$, for which $t=-1$.
\end{proof}

\vspace{0.5cm}

\subsection{Calculations on $\hcoY_{3}$}

Here first we calculate the r.h.s. of (\ref{eq:lastvan}) postponing
the cases $C_{i4}(1)(1\leq i\leq3)$ to the latter half of this subsection.
We use the Leray spectral sequence associated with $\Lpi_{\hcoY_{3}}\colon\hcoY_{3}\to\mP(V)$.

Let $G\simeq\mathrm{G}(3,6)$ be a fiber of $\Lpi_{\hcoY_{3}}$. Since
$L_{\hcoY_{3}}$, $\Lpi_{\hcoY_{3}}^{\;*}T(-1)$ and $\Lpi_{\hcoY_{3}}^{\;*}\Omega(1)$
are pull-backs of sheaves on $\mP(V)$, and also $\det Q|_{G}=\sO_{G}(1)$
holds, we can write the restrictions of $\overline{\sC}_{ij}(-t\det Q+tL_{\hcoY_{3}})$
to $G$ by using the following sheaves: \begin{equation}
\begin{aligned} & \sS^{*}|_{G}\otimes\sS|_{G}(-t),\;\;\;\sS^{*}|_{G}(-t),\;\;\;\sS^{*}|_{G}\otimes\sQ^{*}|_{G}(-t),\;\;\;\\
 & \sS|_{G}(-t),\;\;\;\sO_{G}(-t),\;\;\;\sQ^{*}|_{G}(-t),\;\;\;\sQ|_{G}\otimes\sQ^{*}|_{G}(-t),\\
 & \sQ|_{G}\otimes\sS|_{G}(-t),\;\;\;\sQ|_{G}(-t)\end{aligned}
\quad\begin{aligned} & (0\leq t\leq5)\\
 & (0\leq t\leq5)\\
 & (1\leq t\leq5)\end{aligned}
\label{eq:coh-fiberG}\end{equation}
 where $(-t)$ represents the twist by $\sO_{G}(-t)$. 

\begin{lem} \label{lem:nonvan} All the cohomology groups of the
sheaves in $(\ref{eq:coh-fiberG})$ vanish except \[
H^{0}(\sS^{*}|_{G}\otimes\sS|_{G}),\;\; H^{0}(\sS^{*}|_{G}),\;\; H^{0}(\sO_{G}),\text{ and }H^{0}(\sQ|_{G}\otimes\sQ^{*}|_{G}).\]
 \end{lem}

\begin{proof} We can verify these properties by using the Bott theorem
\ref{thm:Bott}. \end{proof}

\begin{prop} Consider the cohomologies $H^{\bullet}(\hcoY_{2},\sC_{ij}(-t))$
except the cases $\sC_{i4}(1)\,(1\leq i\leq3)$ as in Proposition
$\ref{cla:cohY2-Y3}$. The r.h.s. of $(\ref{eq:lastvan})$ vanishes
except possibly \begin{equation}
H^{\bullet}(\hcoY_{2},\sC_{ij})\simeq H^{\bullet}(\mP(V),{\Lpi_{\hcoY_{3}*}}\overline{\sC}_{ij})\;\;\text{with }1\leq i,j\leq3\text{ or }i=j=4.\label{eq:non-vCij}\end{equation}
 \end{prop}

\begin{proof} This follows from the isomorphisms (\ref{eq:lastvan})
and Lemma \ref{lem:nonvan}. \end{proof}

\vspace{0.3cm}
 For the evaluations of the r.h.s. of (\ref{eq:non-vCij}), we can
use the Bott theorem for $\Lpi_{\hcoY_{3}}\colon\hcoY_{3}\to\mP(V)$.
All non-vanishing sheaves turns out to be \[
\begin{matrix}\Lpi_{\hcoY_{3}*}\overline{\sC}_{11}, & \Lpi_{\hcoY_{3}*}\overline{\sC}_{12}, & \Lpi_{\hcoY_{3}*}\overline{\sC}_{13}\\
 & \Lpi_{\hcoY_{3}*}\overline{\sC}_{22}, & \Lpi_{\hcoY_{3}*}\overline{\sC}_{23}\\
 &  & \Lpi_{\hcoY_{3}*}\overline{\sC}_{33}\end{matrix}\;\;\;\;\simeq\;\;\begin{matrix}\sO, & T(-1)^{\wedge2}\otimes\Omega(1), & T(-1)^{\wedge2}\\
 & T(-1)\otimes\Omega(1), & T(-1)\\
 &  & \sO\end{matrix}\]
 and $\Lpi_{\hcoY_{3}*}\overline{\sC}_{44}\simeq\sO$.

We can compute the cohomology groups on $\mP(V)$ by applying the
Bott theorem \ref{thm:Bott} again. For the calculations, we first
evaluate irreducible decompositions, for example, \[
T(-1)^{\wedge2}\otimes\Omega(1)\simeq\Sigma^{(0,0,0,-1)}\Omega(1)\oplus\Sigma^{(1,0,-1,-1)}\Omega(1)\]
 by the Littlewood-Richardson rule. In this way, we finally obtain
the following non-vanishing cohomology groups: \[
\begin{matrix}H^{0}(\hcoY_{2},\sC_{11}), & H^{0}(\hcoY_{2},\sC_{12}), & H^{0}(\hcoY_{2},\sC_{13})\\
 & H^{0}(\hcoY_{2},\sC_{22}), & H^{0}(\hcoY_{2},\sC_{23})\\
 &  & H^{0}(\hcoY_{2},\sC_{33})\end{matrix}\;\;\simeq\begin{matrix}\mC, & V, & \wedge^{2}V\\
 & \mC, & V\\
 &  & \mC\end{matrix}\]
 and $H^{0}(\hcoY_{2},\sC_{44})\simeq\mC$.

\vspace{0.5cm}
 Now we turn our attention to the cases $H^{\bullet}(\hcoY_{2},\sC_{i4}(1))\;(1\leq i\leq3)$
for which the isomorphism (\ref{eq:lastvan}) does not apply.

\pagebreak{}

${\bf (1)}$ $H^{\bullet}(\widetilde{\hcoY},\widetilde{\sC}_{14})\simeq H^{\bullet}(\hcoY_{2},\sC_{14}(1))\simeq H^{0}(\hcoY_{2},\sC_{14}(1))\simeq\ft{S}^{2}V$.

\noindent \;

For $\sC_{14}(1)=\Lrho_{\hcoY_{2}}^{\;*}\overline{\sC}_{14}(1)$,
we have \[
\sC_{14}(1)=\Lrho_{\hcoY_{2}}^{\;*}(\sS^{*}(L_{\hcoY_{3}})\otimes\sQ^{*})(M_{\hcoY_{2}})\simeq\Lrho_{\hcoY_{2}}^{\;*}(\sS^{*}\otimes\sQ^{*}(\det\sQ))(-F_{\rho})\]
 by Proposition \ref{cla:M}. Consider the exact sequence \[
\begin{aligned}0\to\Lrho_{\hcoY_{2}}^{\;*}(\sS^{*}\otimes\sQ^{*}(\det\sQ))(-F_{\rho}) & \to\Lrho_{\hcoY_{2}}^{\;*}(\sS^{*}\otimes\sQ^{*}(\det\sQ))\\
 & \to(\Lrho_{\hcoY_{2}}|_{F_{\rho}})^{*}(\sS^{*}\otimes\sQ^{*}(\det\sQ)|_{\Prt_{\rho}})\to0.\end{aligned}
\]
 We evaluate the cohomology of the middle term by $H^{\bullet}(\hcoY_{3},\sS^{*}\otimes\sQ^{*}(\det\sQ))\simeq\oplus_{i}H^{\bullet-i}(\mP(V),R_{\Lpi_{\hcoY_{3}}*}^{i}(\sS^{*}\otimes\sQ^{*}(\det\sQ)))$.
By the Bott theorem for $\Lpi_{\hcoY_{3}}\colon\hcoY_{3}\to\mP(V)$,
it is easy to see that the only non-trivial term comes from ${\Lpi_{\hcoY_{3}*}}(\sS^{*}\otimes\sQ^{*}(\det\sQ))\simeq\Sigma^{(1,1,0,0,0,-1)}T(-1)^{\wedge2}.$
To use the Bott theorem again for the cohomology over $\mP(V)$, we
note the following plethysm of the Schur functors: \[
\Sigma^{(1,1,0,0,0,-1)}T(-1)^{\wedge2}\simeq\Sigma^{(2,0,0,0)}T(-1)\oplus\Sigma^{(1,1,1,-1)}T(-1)\oplus\Sigma^{(2,1,0,-1)}T(-1).\]
 Then the only non-vanishing result comes from the first summand,
and we finally obtain $H^{\bullet}(\hcoY_{3},\sS^{*}\otimes\sQ^{*}(\det\sQ))=H^{0}(\mP(V),\Sigma^{(2,0,0,0)}T(-1))\simeq\ft{S}^{2}V$.

Now, let us note that $H^{\bullet}(F_{\rho},({\Lrho_{\hcoY_{2}}|_{F_{\rho}}})^{*}(\sS^{*}\otimes\sQ^{*}(\det\sQ)|_{\Prt_{\rho}}))\simeq H^{\bullet}(\Prt_{\rho},\sS^{*}\otimes\sQ^{*}(\det\sQ)|_{\Prt_{\rho}}).$
By Propositions \ref{cla:mislead} and \ref{cla:det}, we have $\sS^{*}\otimes\sQ^{*}(\det\sQ)|_{\Prt_{\rho}}\simeq(\sR_{V}/\sR_{2})\otimes(\sR_{V}/\sR_{2})^{*}(2H_{\Prt_{\rho}}+L_{\Prt_{\rho}})$.
In a similar way to the above calculations, by considering the fibration
$\Prt_{\rho}\to\mP(V)$, we see that \[
H^{\bullet}(\Prt_{\rho},(\sR_{V}/\sR_{2})\otimes(\sR_{V}/\sR_{2})^{*}(2H_{\Prt_{\rho}}+L_{\Prt_{\rho}}))\simeq H^{\bullet}(\Prt_{\rho},\sO_{\Prt_{\rho}}(2H_{\Prt_{\rho}}+L_{\Prt_{\rho}})).\]
The r.h.s. vanishes for $\bullet\not=0$. For $\bullet=0$, we note
that $H^{0}(\Prt_{\rho},\sO_{\Prt_{\rho}}(2H_{\Prt_{\rho}}+L_{\Prt_{\rho}}))$
is isomorphic to $H^{0}(\mP(V),\ft{S}^{2}T(-1)\otimes\sO(-1))$. Then
this vanishes by the Bott theorem \ref{thm:Bott} on $\mP(V)$.

This completes our calculation of $H^{\bullet}(\widetilde{\hcoY},\widetilde{\sC}_{14})$.

\vspace{0.5cm}
 ${\bf (2)}$ $H^{\bullet}(\widetilde{\hcoY},\widetilde{\sC}_{34})\simeq H^{\bullet}(\hcoY_{2},\sC_{34}(1))\simeq0$.

\noindent \;

\;

For $\sC_{34}(1)=\Lrho_{\hcoY_{2}}^{\;*}\overline{\sC}_{34}(1)$,
we have \[
\sC_{34}(1)=\Lrho_{\hcoY_{2}}^{\;*}(\sQ^{*})(M_{\hcoY_{2}})\simeq\Lrho_{\hcoY_{2}}^{\;*}(\sQ^{*}(\det\sQ-L_{\hcoY_{3}}))(-F_{\rho}).\]
 by Proposition \ref{cla:M}. Since the following calculations proceed
exactly in the same ways as above, we only sketch them.

First, we have $H^{\bullet}(\hcoY_{2},\Lrho_{\hcoY_{2}}^{\;*}(\sQ^{*}(\det\sQ-L_{\hcoY_{3}})))\simeq H^{\bullet}(\hcoY_{3},\sQ^{*}(\det\sQ-L_{\hcoY_{3}}))$,
and then evaluate this by the Bott theorem to be $H^{\bullet}(\mP(V),\Sigma^{(1,1,0,0,0,0)}T(-1)^{\wedge2}\otimes\sO(-1))$.
We use the plethysm $\Sigma^{(1,1,0,0,0,0)}T(-1)^{\wedge2}\simeq\Sigma^{(2,1,1,0)}T(-1)$.
Then we evaluate $H^{\bullet}(\mP(V),\Sigma^{(1,1,0,0,0,0)}T(-1)^{\wedge2}\otimes\sO(-1))\simeq H^{\bullet}(\mP(V),\Sigma^{1,0,0,-1}T(-1))=0$
by the Bott theorem.

Second, we note the isomorphism \[
H^{\bullet}(F_{\rho},({\Lrho_{\hcoY_{2}}|_{F_{\rho}}})^{*}(\sQ^{*}(\det\sQ-L_{\hcoY_{3}})|_{\Prt_{\rho}}))\simeq H^{\bullet}(\Prt_{\rho},\sQ^{*}(\det\sQ-L_{\hcoY_{3}})|_{\Prt_{\rho}}).\]
 By Proposition \ref{cla:det}, we have now $\sQ^{*}(\det\sQ-L_{\hcoY_{3}})|_{\Prt_{\rho}}\simeq(\sR_{V}/\sR_{2})(H_{\Prt_{\rho}})$.
Then we see that $H^{\bullet}(\Prt_{\rho},\sQ^{*}(\det\sQ-L_{\hcoY_{3}})|_{\Prt_{\rho}})\simeq H^{\bullet}(\mP(V),\Sigma^{(1,0,0,-1)}\Omega(1))$,
which vanish for all $\bullet$.

This completes our calculation of $H^{\bullet}(\widetilde{\hcoY},\widetilde{\sC}_{34})$.

\vspace{0.5cm}
 ${\bf (3)}$ $H^{\bullet}(\widetilde{\hcoY},\widetilde{\sC}_{24})\simeq H^{\bullet}(\hcoY_{2},\sC_{24}(1))\simeq H^{0}(\hcoY_{2},\sC_{24}(1))\simeq V$.

\noindent \;

\;

Finally, for $\sC_{24}(1)=\sT_{2}\otimes\Lrho_{\hcoY_{2}}^{\;*}\sQ^{*}(M_{\hcoY_{2}})$,
we consider the following exact sequence which we derive from (\ref{eq:Eb}):
\[
\begin{aligned}0\to\Lpi_{\hcoY_{2}}^{\;*}T(-1)\otimes & \Lrho_{\hcoY_{2}}^{\;*}{\sQ}^{*}({M}_{\hcoY_{2}})\to\sT_{2}\otimes\Lrho_{\hcoY_{2}}^{\;*}{\sQ}^{*}({M}_{\hcoY_{2}})\\
 & \to(\Lrho_{\hcoY_{2}}|_{F_{\rho}})^{*}(\sR_{2}/\sR_{1})^{*}\otimes\Lrho_{\hcoY_{2}}^{\;*}{\sQ}^{*}({M}_{\hcoY_{2}}+F_{\rho})|_{F_{\rho}}\to0.\end{aligned}
\]
 Then for our purpose it suffices to compute $H^{\bullet}(\hcoY_{2},\Lpi_{\hcoY_{2}}^{\;*}T(-1)\otimes\Lrho_{\hcoY_{2}}^{\;*}{\sQ}^{*}({M}_{\hcoY_{2}}))$
and also $H^{\bullet}(F_{\rho},(\Lrho_{\hcoY_{2}}|_{F_{\rho}})^{*}(\sR_{2}/\sR_{1})\otimes\Lrho_{\hcoY_{2}}^{\;*}{\sQ}^{*}({M}_{\hcoY_{2}}+F_{\rho})|_{F_{\rho}})$.
We can compute the former in a similar way to the above two cases,
and we see that they all vanish. Using (1) and (2) of Propositions~\ref{cla:det},
\ref{cla:M}, we see that the latter cohomologies are isomorphic to
$H^{\bullet}(\Prt_{\rho},(\sR_{V}/\sR_{2}))$. Now, from the defining
exact sequence (\ref{eq:exact-sequence-Rv}) of $\sR_{V}/\sR_{2}$,
we obtain $H^{\bullet}(\Prt_{\rho},(\sR_{V}/\sR_{2}))\simeq H^{\bullet}(\mP(V),T(-1))$,
which vanish except $H^{0}(\mP(V),T(-1))\simeq V$.

This completes our calculation of $H^{\bullet}(\widetilde{\hcoY},\widetilde{\sC}_{24})$.

\vspace{0.3cm}
 Now our calculations of the cohomology groups (\ref{eq:lastvan})
and the cases (1)-(3) above complete our proof of Theorem \ref{thm:Gvan}.

\newpage{}

\appendix

\section{{The {}``double spin'' coordinates of $\mathrm{G}(3,6)$ } }

In this appendix, we set $V_{4}=\mathbb{C}^{4}$ with the standard
basis. We can write the irreducible decomposition (\ref{eq:spin})
as \[
\wedge^{3}(\wedge^{2}V_{4})=\Sigma^{(3,1,1,1)}V_{4}\,\oplus\,\Sigma^{(2,2,2,0)}V_{4}\simeq\mathsf{S}^{2}V_{4}\,\oplus\,\mathsf{S}^{2}V_{4}^{*},\]
 where $\Sigma^{\beta}$ is the Schur functor. We define the projective
space $\mathbb{P}(\wedge^{3}(\wedge^{2}V_{4}))=\mathbb{P}(\mathsf{S}^{2}V_{4}\,\oplus\,\mathsf{S}^{2}V_{4}^{*})$.
The homogeneous coordinate of $\mathbb{P}(\mathsf{S}^{2}V_{4}\,\oplus\,\mathsf{S}^{2}V_{4}^{*})$
is naturally introduced by $[v_{ij},w_{kl}]$, where $v_{ij}$ and
$w_{kl}$ are entries of $4\times4$ symmetric matrices. Let $\mathcal{I}=\left\{ \{i,j\}\mid1\leq i<j\leq4\right\} $
the index set to write the standard basis of $\wedge^{2}V_{4}$, then
the homogeneous coordinate of $\mathbb{P}(\wedge^{3}(\wedge^{2}V_{4}))$
is naturally given by the $[p_{IJK}]$ where $p_{IJK}$ is totally
anti-symmetric for the indices $I,J,K\in\mathcal{I}.$ These two coordinates
are related by the above irreducible decomposition. Focusing the different
symmetry properties of the Schur functors, it is rather straightforward
to decompose $p_{IJK}$ into the two components. When we use the signature
function defined by ${\bf e}_{i_{1}}\wedge{\bf e}_{i_{2}}\wedge{\bf e}_{i_{3}}\wedge{\bf e}_{i_{4}}=\epsilon^{i_{1}i_{2}i_{3}i_{4}}{\bf e}_{1}\wedge{\bf e}_{2}\wedge{\bf e}_{3}\wedge{\bf e}_{4}$
with a basis ${{\bf e}_{1},..,{\bf e}_{4}}$ of $V_{4}$, they are
given by \begin{equation}
v_{ij}=\frac{1}{6}\sum_{k,l,m,n}\epsilon^{klmn}p_{[ik][jl][mn]},\quad w_{kl}=\frac{1}{6}\sum_{a,b,c}\sum_{m,n,q}\epsilon^{kabc}\epsilon^{lmnq}p_{[am][bn][cq]},\label{eq:vw-general-fromula}\end{equation}
where the square brackets in $p_{[ij][kl][mn]}$ represents the anti-symmetric
extensions of the indices, i.e., $p_{[ij][J][K]}=p_{\{ij\}[J][K]}$
for $i<j$ while $p_{[ij][J][K]}=-p_{\{ji\}[J][K]}$ for $i\geq j$.
For convenience, we write them in the following (symmetric) matrices:
\begin{equation}
\begin{aligned}v=(v_{ij})=\left(\begin{matrix}2p_{{\bf 124}} & p_{{\bf 134}}+p_{{\bf 125}} & p_{{\bf 234}}+p_{{\bf 126}} & p_{{\bf 146}}-p_{{\bf 245}}\\
 & 2p_{{\bf 135}} & p_{{\bf 235}}+p_{{\bf 136}} & p_{{\bf 156}}-p_{{\bf 345}}\\
 &  & 2p_{{\bf 236}} & p_{{\bf 256}}-p_{{\bf 346}}\\
 &  &  & 2p_{{\bf 456}}\end{matrix}\right),\\
w=(w_{kl})=\left(\begin{matrix}2p_{{\bf 356}} & -p_{{\bf 346}}-p_{{\bf 256}} & p_{{\bf 345}}+p_{{\bf 156}} & p_{{\bf 235}}-p_{{\bf 136}}\\
 & 2p_{{\bf 246}} & -p_{{\bf 245}}-p_{{\bf 146}} & p_{{\bf 126}}-p_{{\bf 234}}\\
 &  & 2p_{{\bf 145}} & p_{{\bf 134}}-p_{{\bf 125}}\\
 &  &  & 2p_{{\bf 123}}\end{matrix}\right),\end{aligned}
\label{eq:vw-plucker}\end{equation}
 where we ordered the index set $\mathcal{I}$ as $\{{\bf 1},{\bf 2},...,{\bf 6}\}=\{\{1,2\},\{1,3\},\{2,3\},\{1,4\},$
$\{2,4\},$ $\{3,4\}\}$. Inverting the relations (\ref{eq:vw-plucker}),
we can write the Pl\"ucker relations among $p_{IJK}$ in terms of
the entries of $v$ and $w$. After some algebra, we find: \vspace{0.3cm}
 \textbf{Proposition A.1} \textit{The Pl\"ucker ideal $I_{G}$ of
$\mathrm{G}(3,6)\subset\mathbb{P}(\wedge^{3}(\wedge^{2}V_{4}))$ is
generated by \begin{equation}
\begin{aligned}|v_{IJ}|-\epsilon_{I\check{I}}\epsilon_{J\check{J}}|w_{\check{I}\check{J}}|\qquad(I,J\in\mathcal{I}),\qquad\qquad\\
(v.w)_{ij},\;\;(v.w)_{ii}-(v.w)_{jj}\;\;(i\not=j,1\leq i,j\leq4),\end{aligned}
\label{eq:Ivw}\end{equation}
 where $\check{I}$ represents the complement of $I$, i.e., $x\in\mathcal{I}$
such that $x\cup I=\{1,2,3,4\}$ and similarly for $\check{J}$. $|v_{IJ}|$
and $|w_{IJ}|$ represent the $2\times2$ minors of $v$ and $w$,
respectively, with the lows and columns specified by $I$ and $J$.
$\epsilon_{I\check{I}}$ is the signature of the permutation of the
'ordered' union $I\cup\check{I}$. $(v.w)_{ij}$ is the $ij$-entry
of the matrix multiplication $v.w$. }

\noindent \vspace{0.5cm}

For all $[v,w]\in V(I_{G})\simeq\mathrm{G}(3,6)$, we show the following
relations (I.1)-(I.5):

\vspace{0.2cm}

\noindent \textbf{(I.1)} $\det\, v=\det\, w$.

By the Laplace expansion of the determinant of $4\times4$ matrix
$v$, we have $\det\, v=\sum_{J\in\mathcal{I}}\epsilon_{J\check{J}}|v_{IJ}||v_{\check{I}\check{J}}|$.
Then, using the first relations of (\ref{eq:Ivw}), we obtain the
equality.

\vspace{0.2cm}

\noindent \textbf{(I.2)} $v.w=\pm\sqrt{\det\, w}\, id_{4}$, where
$id_{4}$ is the $4\times4$ identity matrix.

Note that the second line of (\ref{eq:Ivw}) may be written in a matrix
form $v.w=d\, id_{4}$ with $d=(v.w)_{11}=\cdots=(v.w)_{44}$. Then,
by {(I.1)}, we have $\det v\cdot w=(\det\, w)^{2}=d^{4}$ and hence
$d^{4}-(\det\, w)^{2}=(d^{2}-\det\, w)(d^{2}+\det\, w)=0$. We consider
a special case; $v=a\, id_{4}$, $w=a\, id_{4}$. Then $d=(v.w)_{11}=a^{2}$.
Therefore $d^{2}=a^{4}=\det\, w$ must holds for all since $V(I_{G})\simeq\mathrm{G}(3,6)$
is irreducible. Hence $d=\pm\sqrt{\det\, w}$ as claimed.

\vspace{0.2cm}

\noindent \textbf{(I.3)} ${\rm rk}\, w\not=3$ and also ${\rm rk}\, v\not=3$.

Assume ${\rm rk}\, w=3$, then from (I.2) we have $v.w=0$, which
implies ${\rm rk}\, v\leq1$. However, this contradicts the first
relations of (\ref{eq:Ivw}). Hence ${\rm rk}\, w\not=3$. By symmetry,
we also have ${\rm rk}\, v\not=3$.

\vspace{0.2cm}

\noindent \textbf{(I.4)} ${\rm rk}\, w=2\Leftrightarrow{\rm rk}\, v=2$.

When ${\rm rk}\, w=2$, we see ${\rm rk}\, v\geq2$ by the first relations
of (\ref{eq:Ivw}). From (I.1) and (I.3), we must have ${\rm rk}\, v=2$.
The converse follows in the same way.

\vspace{0.2cm}

\noindent \textbf{(I.5)} ${\rm rk}\, w\leq1\Leftrightarrow{\rm rk}\, v\leq1$.

This is immediate from the the first relations of (\ref{eq:Ivw}).

\vspace{1cm}



$\;$

$\;$

\noindent {\footnotesize  Graduate School of Mathematical Sciences, University of Tokyo \hfill\break  Megro-ku, Tokyo 153-8914, Japan }

\noindent{\footnotesize e-mail addresses: hosono@ms.u-tokyo.ac.jp, takagi@ms.u-tokyo.ac.jp} 

\begin{thebibliography}{35}
\bibitem{BC} L.~Borisov and A.~Caldararu, \textit{The Pfaffian-Grassmannian
derived equivalence}, J. Algebraic Geom. \textbf{{18}} (2009), no.
2, 201-222.

\bibitem{Bondal} A.~Bondal, \textit{Representations of associate
algebras ans coherent sheaves}, Math. USSR. Izvestiya \textbf{34}
(1990), no. 1, 23--42.

\bibitem{BO} A.~Bondal and D.~Orlov, \textit{Semiorthogonal decomposition
for algebraic varieties}, arXiv:alg-geom/9506012

\bibitem{Bo} R.~Bott, \textit{Homogeneous vector bundles}, Ann.
of Math. (2) \textbf{66} (1957), 203--248


{\small \par}

\bibitem{Ch} P.~E.~Chaput, \textit{Scorza varieties and Jordan
algebras}, Indag. Math. (N.S.) {\textbf{14}} (2003), no.~2, 169--182.

\bibitem{Co} F.~Cossec, \textit{Reye congruences}, Trans. Amer.
Math. Soc. \textbf{280} (1983), 737--751.

\bibitem{CR} A.~Corti and M.~Reid, \textit{Weighted Grassmannians},
Algebraic geometry, 141--163, de Gruyter, Berlin, 2002.

\bibitem{D} M.~Demazure, \textit{A very simple proof of Bott's theorem},
Invent. Math. \textbf{33} (1976), no. 3, 271--272.

\bibitem{Fuj} O.~Fujino, \textit{Applications of Kawamata's positivity
theorem}, Proc. Japan Acad. Ser. A Math. Sci. \textbf{75} (1999),
no. 6, 75--79.

\bibitem{Fujita} T.~Fujita, \textit{Classification theories of polarized
varieties}, London Mathematical Society Lecture Note Series, \textbf{155},
Cambridge University Press, Cambridge, 1990. xiv+205 pp.

\bibitem{Ful} W.~Fulton, \textit{Intersection theory}, Ergebnisse
der Mathematik und ihrer Grenzgebiete (3), {\textbf{2}}. Springer-Verlag,
Berlin, 1984. xi+470 pp.

\bibitem{FH} W. Fulton and J. Harris, \textit{Representation Theory,
a first course}, GTM 129, Springer-Verlag 1991.

\bibitem{GKZ} I.~M.~Gel'fand, A.~V.~Zelevinski, and M.~M.~Kapranov,
\textit{Discriminants, Resultants and Multidimensional Determinants},
Birkh\"auser Boston, 1994.


{\small \par}

\bibitem{M2}D. R. Grayson and M. E. Stillman, Macaulay2, a software
system for research in algebraic geometry, Available at http://www.math.uiuc.edu/Macaulay2/. 

\bibitem{Ha} R.~Hartshorne, \textit{Algebraic Geometry}, GTM \textbf{52},
Springer-Verlag, New York-Heidelberg, 1977. xvi+496 pp.

\bibitem{HoTa1} S.~Hosono and H.~Takagi, \textit{Mirror symmetry
and projective geometry of Reye congruences I}, arXiv:1101.2746, to
appear in Journal of Algebraic Geometry.

\bibitem{HoTa2} S.~Hosono and H.~Takagi, \textit{Determinantal
quintics and mirror symmetry of Reye congruences}, preprint, arXiv:1208.1813,
submitted.

\bibitem{HoTa3} S.~Hosono and H.~Takagi, \textit{Double quintic
symmetroids, Reye congruences, and their derived equivalence}, preprint,
arXiv:1302.5883, submitted.


{\small \par}

\bibitem{HoTa4} S.~Hosono and H.~Takagi, \textit{On homological
projective duality for $\ft{S}^{2}\mP^{n}$}, in preparation. 


\bibitem{Huy} D.~Huybrechts, \textit{Fourier-Mukai Transforms in
Algebraic Geometry}, Oxford Mathematical Monographs, Oxford 2006.


{\small \par}

\bibitem{IM} A.~Iliev and L.~Manivel, \textit{Fano manifolds of
degree ten and EPW sextics}, Annales scientifiques de l'Ecole Normale
Sup\'erieure \textbf{44} (2011), 393--426.

\bibitem{Ka} Y.~Kawamata, \textit{Small contractions of four-dimensional
algebraic manifolds}, Math. Ann. 284 (1989), no. 4, 595--600.

\bibitem{KMM} Y.~Kawamata, K.~Matsuda, and K.~Matsuki, \textit{Introduction
to the minimal model problem.} Algebraic geometry, Sendai, 1985, 283--360,
Adv. Stud. Pure Math., \textbf{10}, North-Holland, Amsterdam, 1987.

\bibitem{HypSec} A.~Kuznetsov, \textit{Hyperplane sections and derived
categories}, Izv. Ross. Akad. Nauk Ser. Mat. 70 (2006), no. 3, 23--128;
translation in Izv. Math. 70 (2006), no. 3, 447--547.

\bibitem{HPD1} A.~Kuznetsov, {\textit{Homological projective duality}},
Publ. Math. Inst. Hautes \'Etudes Sci. No. 105 (2007), 157--220.

\bibitem{Lef} A.~Kuznetsov, \textit{Lefschetz decompositions and
categorical resolutions of singularities}, Selecta Math. (N.S.) \textbf{13}
(2008), no.~4, 661--696.

\bibitem{KuzGrIsotropic} A.~Kuznetsov, \textit{Exceptional collections
for Grassmannians of isotropic lines}, Proc. London Math. Soc. (3)
\textbf{97 }(2008) 155--182.


{\small \par}

\bibitem{HPD2} A.~Kuznetsov, \textit{Homological projective duality
for Grassmannians of lines}, arXiv:math/0610957.


{\small \par}

{\small \par}

\bibitem{Mu} S.~Mukai, \textit{ New developments in the theory of
Fano threefolds: vector bundle method and moduli problems}, Sugaku
Expositions 15 (2002), no. 2, 125150.

\bibitem{Ol} C.~Oliva, \textit{Algebraic cycles and Hodge theory
on generalized Reye congruences}, Compositio Math. \textbf{92}(1994),
1--22.


{\small \par}

\bibitem{Sa} A.~Samokhin, \textit{Some remarks on the derived categories
of coherent sheaves on homogeneous spaces}, J. Lond. Math. Soc., II.
Ser. \textbf{76}, no.~1, 122-134 (2007),

\bibitem{Tk} H.~Takagi, \textit{Classification of primary $\mQ$-Fano
threefolds with anti-canonical Du Val $K3$ surfaces. I}, J. Algebraic
Geom. \textbf{15} (2006), no. 1, 31--85.


{\small \par}

\bibitem{Tj} A.~N.~Tjurin, \textit{On intersections of quadrics},
Russian Math. Surveys \textbf{30} (1975), 51--105.

\bibitem{Z} F.~L.~Zak, \textit{Tangents and secants of algebraic
varieties}, Translations of Mathematical Monographs, {\textbf{1}27}.
American Mathematical Society, Providence, RI, 1993. viii+164 pp.

\bibitem{W} J.~Weyman, \textit{Cohomology of vector bundles and
}, Cambridge Tracts in Mathematics, \textbf{149}. Cambridge University
Press, Cambridge, 2003. xiv+371 pp. 

\end{thebibliography}
\end{document}